\documentclass{article}
\pdfoutput=1
\usepackage[utf8]{inputenc}
\usepackage[left=2cm, right=2cm, top=2cm]{geometry}
\usepackage{amsmath, amsthm, comment, amssymb, tabu, thmtools}
\usepackage{tikz-cd} 
\usepackage{amsfonts}
\usepackage{hyperref}
\usepackage{mathtools}
\usepackage[numbers]{natbib}
\usepackage{subcaption}
\usepackage{graphicx}
\usepackage{bbm}
\DeclareMathOperator{\ind}{ind}

\newcommand{\R}{\mathbb{R}}
\newcommand{\C}{\mathbb{C}}
\newcommand{\Z}{\mathbb{Z}}
\DeclareMathOperator{\coker}{coker}
\newcommand{\p}{\mathfrak{p}}
\newcommand{\RP}{\mathbb{RP}}
\newcommand{\CP}{\mathbb{CP}}
\newcommand{\D}{\mathbb{D}}

\newcommand{\Grad}{\nabla}
\renewcommand{\P}{\mathcal{P}}
\newcommand{\M}{\mathcal{M}}
\newcommand{\tensor}{\otimes}
\newcommand{\boxtensor}{\boxtimes}
\renewcommand{\check}{\vee}
\declaretheoremstyle[
headfont=\bfseries,
bodyfont=\itshape,
]{myplain}
\declaretheorem[
style=myplain,
name=Propostion,
]{proposition}

\theoremstyle{plain}
\newtheorem{corollary}{Corollary}
\newtheorem{lemma}{Lemma}

\theoremstyle{definition}
\newtheorem{definition}{Definition}

\theoremstyle{remark}
\newtheorem{remark}{Remark}
\newtheorem{example}{Example}

\newtheorem*{conjecture}{Conjecture}

\title{Floer homology via Twisted Loop Spaces}
\author{Semon Rezchikov}
\date{\today}

\begin{document}

\maketitle
\begin{abstract}
	Answering a question of Witten, we introduce a novel method for defining an integral version of Lagrangian Floer homology, \emph{removing} the standard restriction that the Lagrangians in question must be relatively Pin. Using this technique, we derive stronger bounds on the self-intersection of certain exact Lagrangians $\RP^2 \times L'$ than those that follow from traditional methods.  We define a integral version of Lagrangian Floer homology \emph{all} oriented closed exact Lagrangians $L$ in a Liouville domain and prove a general self-intersection bound coming from the algebraic properties of the diagonal bimodule of a \emph{twist} of the $dg$-algebra of chains on the based loop space of $L$.
\end{abstract}
\section*{Introduction}

The literature on coefficients in Lagrangian Floer homology is built around the following idea: in order to make sense of $HF^*(L)$ as an invariant away from characteristic two, one must require that certain \emph{conditions} on $L$ hold, which coherently determine relative orientations of the moduli spaces of holomorphic disks with boundary on $L$. The most common such condition is that $L$ is relatively Spin in the ambient symplectic manifold \cite{FOOO}. If such a condition holds, then the Floer differential is extracted from a signed count of rigid elements in the Floer moduli spaces, where the sign is determined by the coherent orientations. In this paper, we develop a \emph{new} approach to the problem of defining Lagrangian Floer homology outside of characteristic two which does \emph{not} depend the existence of coherent orientations of Floer theoretic moduli spaces. The theme of the paper is that \emph{even when Floer moduli spaces may not be oriented}, they still define \emph{meaningful operations} on a \emph{variant} of the Floer complex. 

The investigation in this paper was prompted by a suggestion of Witten:
\begin{conjecture}(Witten)
	Suppose $L$ is a Lagrangian submanifold admitting a $Spin^c$ structure with associated complex line bundle $\lambda$. Then the Floer homology of $L$ should be defined over $\Z$ whenever $L$ admits a $Spin^c$ connection so that the associated connection on $\lambda$ is flat.
\end{conjecture}

In Section \ref{sec:simplified-construction}, we answer Witten's question in the affirmative by constructing an integral version of Lagrangian Floer homology for Lagrangians $L$ satisfying a weaker condition than the one proposed by Witten. Consider the map 
\begin{equation}\Omega: H^2(L, \Z/2) \to H^1(\Omega L, \Z/2)
\end{equation} 
which maps a chain level representative 
$\phi \in H^2(L, \Z/2)$ to 
\begin{equation}\Omega(\phi)(\gamma) = \phi (ev(\gamma \times S^1))),
\end{equation}
 where $\gamma$ is a 1-chain on the based loop space $\Omega L$, and $ev: \Omega L \times S^1 \to L$ is the evaluation map. The second Steifel-Whitney class $w_2(L)$ thus determines a class $\Omega w_2(L) \in H^1(\Omega L; \Z/2)$ represented by a real line bundle on $\Omega L$. The condition needed in Section \ref{sec:simplified-construction} is
\begin{equation}
\label{assumption:omega}
\begin{minipage}{0.9\textwidth}
The line bundle corresponding to $\Omega w_2(L)$ is trivial, or in other words, $\Omega w_2(L) = 0$.
\end{minipage}
\end{equation}

The original observation that led to the construction of Section \ref{sec:simplified-construction} was that even though there was no canonical way to assign \emph{signs} to holomorphic curves bounding $\RP^2$ because $\RP^2$ is not \emph{Pin}, there is a canonical procedure which assigns \emph{Gaussian integers} to such holomorphic curves. In the formalism of this paper this observation amounts to the following: Assumption (\ref{assumption:omega}), satisfied by $L = \RP^2$, determines a twist $\Z[\pi_1(L)]^{tw}$ of the group ring of the fundamental group of $L$ (see Section \ref{sec:twisted-fundamental-group}), and there is an isomorphism of unital rings 
\begin{equation} \epsilon: \Z[\pi_1(\RP^2)]^{tw} \to \Z[i] 
\end{equation}
sending the nontrival loop on $\RP^2$ to $i$. Section \ref{sec:augmentations-and-proof-of-propositions} explains how Floer theory assigns a complex of $R$-modules corresponding to any augmentation to $R$ of the twisted group ring of a Lagrangian satisfying Assumption (\ref{assumption:omega}). This generalizes the standard construction which assigns Floer homology groups to spin Lagrangians $L$ equipped with unitary local systems $\eta$, which correspond to maps 
\[\epsilon_\eta: \Z[\pi_1(L)]^{tw} = \Z[\pi_1(L)] \to \C. \]

The construction in Section \ref{sec:simplified-construction} lets one prove lower bounds on the self-intersection numbers of exact Lagrangians which are explicit and are stronger than those that could previously be proven. The proposition below is proven in Section \ref{sec:example-prop-proofs-subsec}:
 \begin{proposition}
 \label{prop:example}
 Let $L$ be a closed exact Lagrangian in a Liouville domain satisfying Assumption (\ref{assumption:omega}). 	
 
 If $\Z[\pi_1(L)]^{tw}$ admits an augmentation $\epsilon$ to a field $k$, then the number of intersection points of $L$ with any transversely intersecting Lagrangian that is Hamiltonian isotopic to $L$ is bounded from below by 
 \[\dim H_*(L, k_\epsilon), \]
 where $k_\epsilon$ is a certain $k$-local system depending on $\epsilon$. 
 \end{proposition}
 
 We now give a concrete application of Proposition \ref{prop:example}.
 \begin{example} 
 Let $n \geq 2$, let $p$ be a prime number not equal to $2$, let $r$ an integer greater than zero, and let $L_2 = L(p)^{\# r}$ be a connect sum of $r$ lens spaces 
 \[L(p) := L(p, 1,1, \ldots, 1) \simeq S^{2n-1}/({\Z/p}).\] 
 Take any manifold $L_1$ such that $w_2(L) \neq 0$, but $\Omega w_2(L) = 0$. Moreover, choose $L_1$ so that there exists an augmentation 
 \begin{equation}\epsilon: \Z[\pi_1(L)]^{tw} \to k, \text{ where }k\text{ is a field with }char k \neq 0, p,
 \end{equation} such that the local system $k_\epsilon$ associated to $\epsilon$ by Proposition \ref{prop:example} satisfies
 \[\dim_k H_*(L, k_\epsilon) > 1. \]These constraints are satisfied, for example, by $L_1 = \RP^2$; see Lemma \ref{lemma:RP2-good} in Section \ref{sec:example-prop-proofs-subsec}. Notice that in this case, $L = L_1 \times L_2$ is not Pin. One can produce a Liouville domain $M \supset L$ such that $L$ is not \emph{relatively Pin} in $M$ by by enlarging $T^*L$ via subcritical Weinstein handle attachment \cite{cieliebak2012stein} to a Weinstein domain $M$ with  $H^2(W; \Z/2) = 0$.  In this setting the Lagrangian Floer Homology of $L$ is only defined with coefficients in a ring $R$ with $\text{char } R = 2$ due to the lack of a coherent orientation of the moduli spaces of Floer trajectories. The PSS map \cite{PSS96}, \cite{peter-albers} then gives an isomorphism 
 \begin{equation}HF(L, L; R) \simeq H^*(L_1 \times L_2; R) \simeq H^*(L_1; R) \tensor_R H^*(L_2; R).
 \end{equation} However, $H^*(L; R) \simeq H^*(S^{2n-1}; R)$ is a free rank $2$ $R$-module; thus, if $L_1$ was chosen to be $\RP^2$ then standard methods in Lagrangian Floer homology only show that the intersection of $L$ with any transverse Hamiltonian isotopy of itself have at least $3\cdot2=6$ distinct points.  However, Proposition \ref{prop:example} allows us to achieve a much better self-intersection bound, due to the following property of the local system $k_\epsilon$: 
\begin{proposition}
 \label{prop:example2}
 	Let the notation be that of Proposition \ref{prop:example}. 
 	
 If $L = L_1 \times L_2$ with $L_2$ spin, then $\Z[\pi_1(L)]^{tw} = \Z[\pi_1(L_1)]^{tw} \tensor_\Z \Z[\pi_1(L_2)]$, and $k_{\epsilon \tensor \epsilon_0} \simeq \pi_1^*(k_{\epsilon})$, with $\epsilon_0$ the canonical augmentation $\Z[\pi_1(L_2)] \to \Z \to k$ and  $\pi_1: L_1 \times L_2 \to L_1$ the projection. In particular, by the Kunneth formula,
 \[\{\# \text{self intersection points of }L\} \geq \dim H_*(L_1, k_\epsilon)\cdot\dim H_*(L_2, k). \]
\end{proposition}
 
We have chosen $L_1$ and $L_2$ to satisfy the conditions of Proposition \ref{prop:example2}. Since $\dim H_*(L_1, k_\epsilon) > 1$ and $\dim H_*(L_2, k) > r$, Proposition \ref{prop:example2} shows that our construction gives a family of $d$-dimensional exact Lagrangians in Liouville domains for which the self-intersection bounds proven by the methods of this paper are \emph{arbitrarily stronger} than those that can be derived from the usual approach to Lagrangian Floer Homology, in every odd dimension greater than $4$. Crossing $L_2$ with $S^1$, and re-running the above construction, shows that there are similar examples in every dimension greater than $4$.
\end{example}

We summarize the contents of Section \ref{sec:simplified-construction}. There is a  a Floer complex 
\begin{equation}
\label{eq:introducting-the-simplified-floer-complex}
CF(L_0, L_1; \pi_1^{tw}) := CF(L_0, L_1; \Z[\pi_1(L_0)]^{tw} \tensor \Z[\pi_1(L_1)]^{tw})\text{, defined in  (\ref{eq:floer-complex-simple}),}
\end{equation} for every pair of closed exact Lagrangians $L_0, L_1$ in a Liouville domain satisfying Assumption (\ref{assumption:omega}), which is a complex of $(\Z[\pi_1(L_0)]^{tw}, \Z[\pi_1(L_1)]^{tw})$-bimodules that is independent of Hamiltonian isotopy in the homotopy category of bimodules (Lemma \ref{lemma:floer-simplified-bimodule-invariance}). When $L_0$ and $L_1$ are Hamiltonian isotopic, Section \ref{sec:morse-theory-analog} provides a comparison of this complex with the corresponding Morse-theoretic complex. Section \ref{sec:augmentations-and-proof-of-propositions} incorporates the augmentation $\epsilon$ of Prop. \ref{prop:example} into the theory; the choice of augmentation  allows one to construct a smaller Floer complex
\[CF(L_0, L_0; \epsilon)\]
 which is a complex of $k$-vector spaces instead of a complex of bimodules over twisted group rings. We finally prove Propositions \ref{prop:example} and \ref{prop:example2}, and verify that Witten's condition implies Assumption (\ref{assumption:omega}), in Section \ref{sec:example-prop-proofs-subsec}.

In the same way that the existence of the augmentation $\epsilon$ allows us to simplify the complex of bimodules  $CF(L_0, L_1; \pi_1^{tw})$ (Eq. \ref{eq:introducting-the-simplified-floer-complex}) to a smaller complex of $k$-vector spaces $CF(L_0, L_0; \epsilon)$, the complex $CF(L_0, L_1; \pi_1^{tw})$ should be thought of as a simplification of a larger algebraic structure that exists even when the assumption $\Omega w_2(L_0)$ that is needed in Proposition \ref{prop:example} does not hold. Just as augmentations of the ring $\Z[\pi_1(L_0)]$ correspond to local systems on $L_0$, one should think of augmentations of $\Z[\pi_1(L_0)]^{tw}$ as \emph{local sytems on $L_0$ banded by the gerbe $w_2(L_0)$}, although we do not formalize this point of view in this paper. From the perspective of higher algebra, one can think of  $C_*(\Omega L_0, \Z)$, the Pontrjagin dg-algebra of chains on the based loop space, as a derived version of $\Z[\pi_1(L_0)]$. In Section \ref{sec:twisted-fundamental-group}, we  define dg-algebras $C_*(\P_{x,x}L_i, \p_{x,x})$ which are twists of the algebra $C_*(\Omega L_i; \Z)$, and are a derived analog of $\Z[\pi_1(L_i)]^{tw}$. One can think of dg-modules over $C_*(\Omega L_0, \Z)$ as ``derived local systems on $L_0$''; correspondingly, we informally think of dg-modules over $C_*(\P_{x,x}L_i, \p_{x,x})$ as ``derived local systems on $L_i$ banded by the gerbe $w_2(L_i)$''. From this point of view, the assumption $\Omega w_2 = 0$ comes about naturally as the condition that the gerbe defined by $w_2(L)$ supports an un-derived local system. In the general case, $w_2(L_i)$ may not support any un-derived local systems; but the gerbe \emph{always} supports a universal \emph{derived} local system, namely, $C_*(\P_{x,x}L_i, \p_{x,x})$, and we can make sense of Floer homology with coefficients in this universal derived local system.

Thus, making the technical assumption that $L_0, L_1$ are \emph{oriented}, we define in section \ref{sec:floer_complex} a complex 
\[ CF^*(\Omega L_0, \Omega L_1; H, J) \]
depending on Floer data $(H, J)$ for the Lagrangians $L_i$, which is a iterated extension of free bimodules over $(C_*(\P_{x,x}L_0, \p), C_*(\P_{x,x}L_1, \p))$ (see Section \ref{sec:bimodule-structure}) and is well defined up to a quasi-isomorphism which is canonical in the homotopy category of bimodules see (Prop. \ref{prop:bimodule-structures-invariance}). Generalizing Proposition \ref{prop:example}, in the case of $L_0 = L_1$, we compare this complex with a corresponding Morse theoretic complex defined in Section \ref{sec:morse_complex}, which we subsequently compute in terms of algebraic topology in Section \ref{sec:computing-the-morse-complex}. This computation leads to the following statement, proven in Section \ref{sec:proof-of-proposition-resolution}:

\begin{proposition}
\label{prop:resolution}
Let $L_0$ be an oriented closed exact Lagrangian in a Liouville domain. Then the minimal \emph{size} $s$ of a iterated extension of free modules (Definition \ref{def:twisted-complex}) that is quasi-isomorphic to the diagonal bimodule of $C_*(\P_{x,x}L_0, \p)$ is a bound from below on the number of intersection points with $L_0$ of any transversely intersecting Hamiltonian isotopy $L_1$ of $L_0$.
\end{proposition}

We expect that this proposition admits a modified version when $L_0$ is unoriented; the orientation assumptions made are not essential for the arguments of the paper, and are put in place  to avoid setting up some tedious homological algebra related to the grading shifts coming from nonorientability, which interact inconveniently with the language of iterated extensions of free modules or of twisted complexes. 

We do not know if the result of Proposition \ref{prop:resolution} is optimal. It is natural to imagine homotopical improvements of the proposition which could be proven by combining the ideas of this paper with the methods of Floer homotopy theory \cite{Cohen1995}. The proposition suggests an interesting question in pure algebraic topology:

\textbf{Question}: What is the set of manifolds $\mathcal{S}$ such that for any $L_0 \in \mathcal{S}$, the quantity $s$ defined in Proposition \ref{prop:resolution} (which only depends on the algebraic topology of $L_0$) is equal to the minimal number of critical points of a Morse function on $L_0$? 

\begin{remark}
	 Smale's work on the existence of Morse functions with minimal numbers of critical points shows that $\mathcal{S}$ contains all simply connected Pin manifolds $L_0$ of dimension at least $6$. 
\end{remark}

\textbf{Acknoledgements.} I thank, first and foremost, my advisor Mohammed Abouzaid, who told me about Witten's question, answered many questions about Floer theory, encouraged me to generalize the initial results, and suggested many improvements to the exposition in an earlier draft of the paper. I thank Paul Seidel for several interesting conversations. I also thank Luis Diogo for a helpful early discussion of signs for Floer homology of curves on surfaces.
\numberwithin{lemma}{subsection}
\numberwithin{remark}{subsection}
\numberwithin{definition}{subsection}
\numberwithin{proposition}{subsection}
\numberwithin{corollary}{subsection}
\section{A natural category of paths associated to a pair of Lagrangian submanifolds}
\label{sec:natural-category}
\subsection{Technicalities on path spaces}
\label{sec:technicalities-on-path-spaces}
Let $X$ be a topological space, and $x, y \in X$. The space of \emph{Moore paths in $X$ from $x$ to $y$} is
\[ \mathcal{P}_{x,y}X := \{f: [0, r] \to X\; | r \in [0, \infty), f(0) = x, f(r) = y\}; \]
every such $f$ has a canonical extension to a map $\bar{f}: [0, \infty) \to X$ by requiring that $f(\xi) = f(r)$ if $\xi > r$, and the topology on $\mathcal{P}_{x,y}$ is the subspace topology induced by the corresponding inclusion of $\mathcal{P}_{x,y}X$ into the space of continuous functions $C^0([0, \infty), X)$. Concatenation of paths defines a continuous composition operation, for any triple $(x, y, z) \in X$, of the form
\[ \alpha_{x, y, z}: \P_{x,y}X \times \P_{y,z}X \to \P_{x,z}X \]
which is associative in the sense that for any $x,y, z, w \in X$, the two functions 
\[ \P_{x,y}X \times \P_{y,z}X, \times \P_{z, w}X\to \P_{x,w}X \]
given by $\alpha_{x, z, w} \circ (\alpha_{x,y,z} \times 1)$ and $\alpha_{x,y, w} \circ (1 \times \alpha_{y,z, w})$ are equal. The elements of $\P_{x,x}X$ given by a constant path at $x$ of length (``$r$'') $0$ is are units with respect to this composition, making $\P_{x,y}X$ into the morphism space $\mathcal{P} X(x, y)$ of a topological category $\P X$, the \emph{category of Moore paths on $X$}, with objects given by the points of $X$ in the discrete topology.

\subsection{A category of pairs of paths}
\label{sec:def-pl}

Let $L = (L_0, L_1)$ be a pair of manifolds equipped with a basepoint pair 
\begin{equation}
y_b = (y_b(0), y_b(1)) \in L_0 \times L_1.
\end{equation}

We define a topological category $\P L$ with objects $L_0 \times L_1$, as follows. For any object $y \in \P L$, let $y(0)$ and $y(1)$ refer to the corresponding points on $L_0$ and $L_1$, respectively.  Let 
$(y'', y', y) \in \mathcal{C}$ be a general triple of objects.

The morphism spaces in $\P L$ are defined to be 
\begin{equation}
\mathcal{P} L(y', y) := 
\P_{y(0), y'(0)}L_0 \times \P_{y'(1), y(1)}L_1,
\end{equation}
 and the composition map 
 \begin{equation}
c_{y'', y', y}: \mathcal{P} L(y'', y') \times \mathcal{P} L(y', y) \to \mathcal{P} L(y'', y)
 \end{equation} is the map 
\begin{gather*}
\P_{y'(0), y''(0)}L_0 \times \P_{y''(1), y'(1)}L_1 \times \P_{y(0), y'(0)}L_0 \times \P_{y'(1), y(1)}L_1 \simeq \\ 
\P_{y(0), y'(0)}L_0 \times \P_{y'(0), y''(0)}L_0 \times \P_{y''(1), y'(1)}L_1 \times \P_{y'(1), y(1)}L_1 \xrightarrow{f} \P_{y(0), y''(0)}L_0 \times \P_{y''(1), y(1)}L_1 ; 
\end{gather*}
where the first isomorphism just exchanges factors, and the second map 
\begin{equation}
f := \alpha_{y(0), y'(0), y''(0)} \times \alpha_{y''(1), y'(1), y(1)}
\end{equation}
is the cartesian product of the path concatenation maps on the first two factors and the last two factors, respectively. Figure \ref{fig:graphical_PL_1} gives a graphical representation of the Hom spaces in $\P L$, and Figure \ref{fig:graphical_PL_2} gives a graphical description of composition in this category.

\begin{figure}
\centering
		\includegraphics[width=0.25\linewidth]{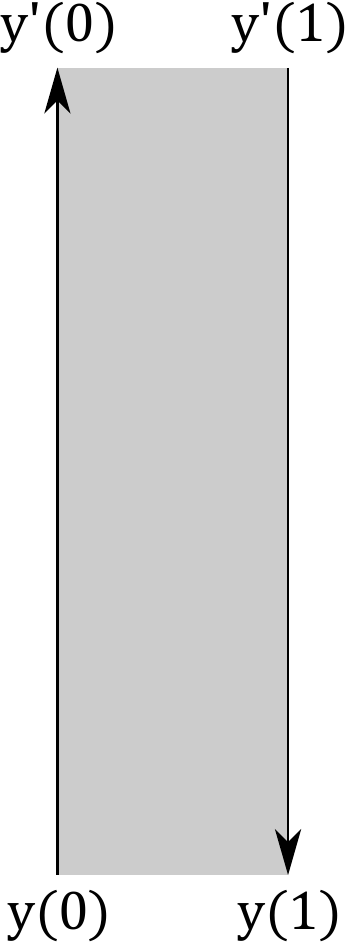}
		\caption{The morphism space $\mathcal{P} L(y',y)$ in $\P L$.} 
	\label{fig:graphical_PL_1}
\end{figure}

\begin{figure}
\includegraphics[width=0.9\linewidth]{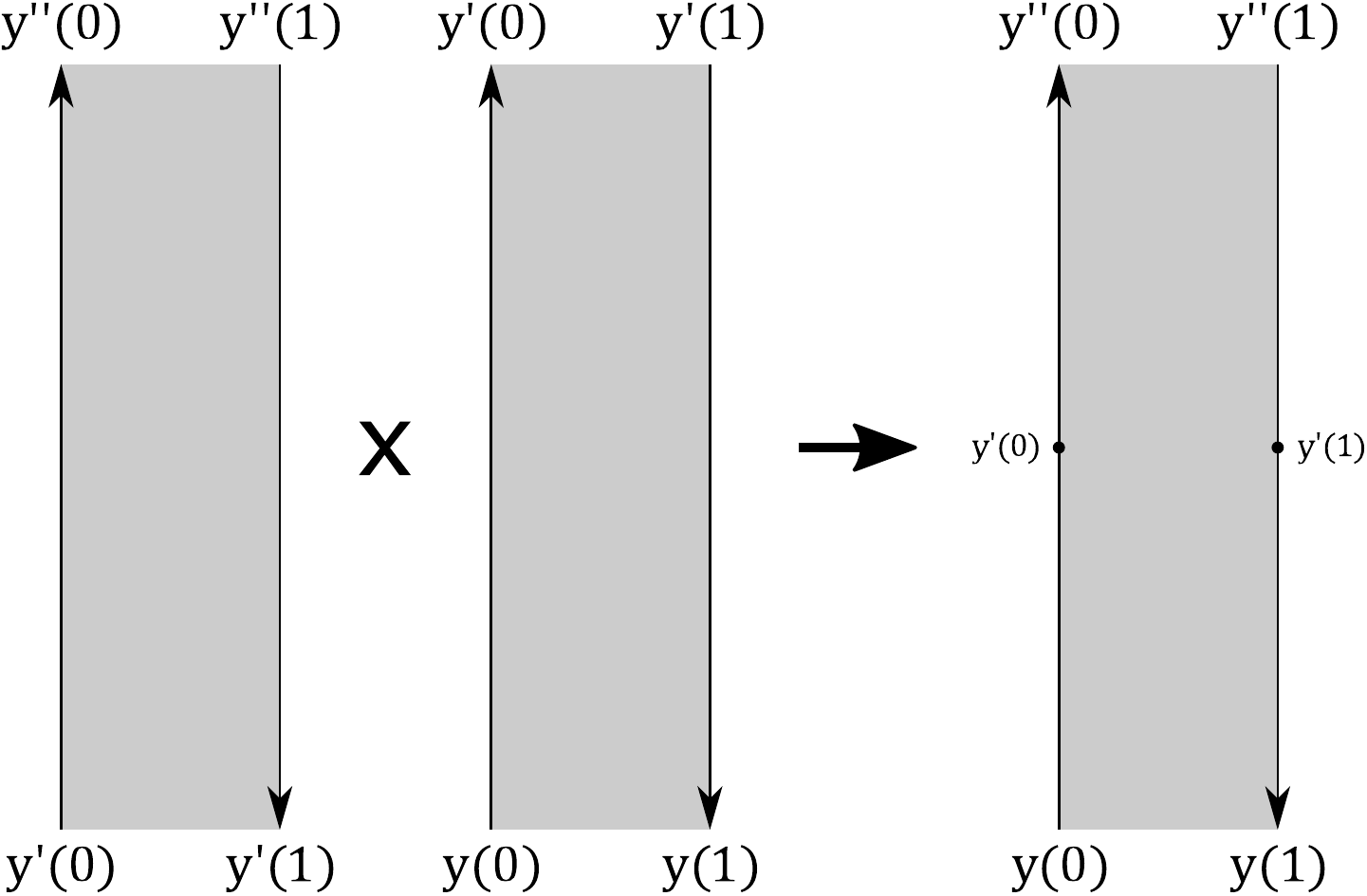}
\caption{Composition in $\mathcal{P}L$.}
\label{fig:graphical_PL_2}
\end{figure}

Let $C_*(\,\cdot\,): Top \to Ch$ be the functor sending a topological space to its asociated complex of singular chains with $\Z$-coefficients. This functor is lax monoidal; the Eilenberg-Zilber map \cite{may_simplicial} gives rise to maps
\begin{equation}
EZ: C_*(X) \tensor C_*(Y) \to C_*(X \times Y) 
\end{equation} 
for every pair of topological spaces $X, Y$, which is natural in both variables, is an isomorphism on homology, and has the property that for any three topological spaces $X, Y, Z$, the two different natural maps
\[ C_*(X) \tensor C_*(Y) \tensor C_*(Z) \to C_*(X \times Y \times Z) \]
that can be constructed of the Eilenberg-Zilber map, are equal as maps of chain complexes.

Using the lax-monoidal structure of $C_*(\,\cdot\,)$, we obtain $C_*(\P L)$, the $dg$-category with objects the points of $X$ and morphism complexes
\[ C_*(\P L)(x, y) = C_*(\P L(x,y))\]
and composition maps defined by composing the Eilenberg-Zilber map  $C_*(\P L(x,y)) \times  C_*(\P L(y,z)) \to C_*(\P L(x,y) \times \P L(y,z))$with $C_*(c_{x,y,z})$. We wish to twist this $dg$-category by a certain ``multiplicative local system''; to explain this twist we must introduce some notation. 

\subsection{Torsors, line bundles, and local systems}\label{sec:torsors-line-bundles-and-local-systems}
Given a $\Z/2$-torsor $E$ over a topological space $X$, we can form its \emph{associated $\Z$-local system} $|E|$ which at $x \in X$ has stalk given by the quotient of the free abelian group on the two-element fiber $E_x$ by the relation $e = (-1)(-e)$, where $e \in E_x$ and $(-e)$ is the image of $e$ under $-1 \in \Z/2$.
From a $\Z$-local system $|E|$ we can then form a real line bundle $|E| \tensor_\Z \R$, which comes equipped with a canonical Riemannian metric with the property that the image of $|E|$ in $|E| \tensor_\Z \R$ is the set of vectors of length $1$. This line bundle has $E$ as its $\Z/2$-torsor of orientations.

Given $\Z/2$ torsors $E \to X$, $E' \to X'$, we can form the external tensor product of local systems $|E| \boxtensor|E'|$ on $X \times X'$. We can also form a $\Z/2 \times \Z/2$-torsor $E \times E'$ over $X \times X'$ by taking the product of the maps $E \to X, E'\to X'$; the quotient of $\Z/2 \times \Z/2$ by the diagonal embedding of $\Z/2$ is canonically identified with $\Z/2$, and the map $E \times E' \to X\times X'$ factors through the quotient by the diagonal $\Z/2$ action $(E \times E')/(\Z/2) \to X \times X'$, which makes $(E \times E')/\Z/2$ into a $\Z/2$ torsor on $X \times X'$ via the (equal) left or right $\Z/2$ actions. It is elementary to check that $|(E \times E')/(\Z/2)|$ is canonically isomorphic to $|E| \boxtensor_\Z |E'|$.

\subsection{Definition of $C_*(\P L, \p)$}
\label{sec:natural-local-system}

In this section we will define a dg-category $C_*(\P L, \p)$ which will be referenced repeatedly in the rest of the paper. 

\begin{remark}
Our conventions for $dg$-algebra are described in Appendix \ref{sec:dg-algebra-conventions}. We note here that in our conventions, the differential in a $dg$-category \emph{decreases degree}.
\end{remark}

\noindent The category $C_*(\P L, \p)$ depends on an additional datum beyond the data needed to define $\mathcal P L$. We thus fix
\begin{equation}
\label{eq:choose_Pin_structures}
\begin{minipage}{0.9\textwidth}
a Pin structure on the tangent spaces 
$T_{y(i)}L_i$, for $i = 0,1$ for every object $y \in Ob(\P L)$. 
\end{minipage}
\end{equation}
\begin{remark}
	This is not the same as a Pin structure on either of the $L_i$, as we do not require the Pin structures chosen above to vary smoothly from point to point.
	In Remark \ref{rk:expanding_C} we explain how the construction depends on the above choice  and the choice of points $\mathcal{C}$. 
\end{remark}

Recall that the morphism spaces of $\P L$ consist of pairs of paths $\gamma_0$ and $\gamma_1$ in $L_0$ and $L_1$ with fixed endpoints; each path in the pair is equipped with a natural vector bundle $\gamma_i^*TL_i \to \text{Domain}(\gamma_i)$, $i = 0, 1$ given by the pullbacks  of the tangent bundle of $L$. The construction of the category $C_*(\P L, \p)$ will involve the use of Pin structures on these vector bundles; in Appendix \ref{sec:gluing_pin_structures_relative_to_ends} we introduce the notions of \emph{Pin structures at the ends} of such a bundle and \emph{Pin structures relative to the ends} on such a bundle, and will use these notions and the notation introduced in that section freely in this construction.  The choice in Eq. \ref{eq:choose_Pin_structures} equips each vector bundle $\gamma_i^*TL_i$ with a Pin structure at the ends. Thus, we have a pair of $\Z/2$-torsors $\Pi(\gamma_i^*TL_i)$, $i = 0, 1$, of Pin structures relative to the ends on $\gamma_i^*TL_i$, for each morphism in $\P L$. The unions

\[\Pi_i(x,y) := \bigcup_{(\gamma_0, \gamma_1) \in \P L(x,y)} \Pi(\gamma_i^*TL_i), \;\;i = 0, 1 \] 
have canonical maps to $\P L(x,y)$ sending $\Pi(\gamma_i^*TL_i)$ to $\gamma_i$, and have a unique topology under which these maps are local homeomorphisms and the $\Z/2$ action is continuous, making them into $\Z/2$-torsors over $\P L(x,y)$.  Define the local system
\begin{equation}
\label{eq:def-local-system-of-pin-structures}
\p_{x,y} = |\Pi_0(x,y)| \tensor |\Pi_1(x,y)|. 
\end{equation}
\begin{remark}
While we will be careful with the notation in this section, throughout this paper, we will occasionally abuse notation and write $\mathfrak{p}$ for $\mathfrak{p}_{x,y}$ whenever the endpoints $x,y$ are clear from the context.
\end{remark}

The operation of gluing Pin structures relative to the ends (see Appendix \ref{sec:gluing_pin_structures_relative_to_ends}) gives maps 
\begin{equation}
\Pi_i(x,y) \times \Pi_i(y, z) \to \Pi_i(x,z)
\end{equation} 
covering the composition map $c_{x,y,z}$ in $\P L$ which, over corresponding points, is a map of $(\Z/2 \times \Z/2)/\Z/2$-torsors, giving an isomorphism
\begin{equation}
\Pi_i(x,y) \times \Pi_i(y, z)/\Z/2 \to c_{x,y,z}^*\Pi_i(x,z)
\end{equation}
 of $\Z/2$ torsors, and thus an isomorphism of $\Z$-local systems 
 \begin{equation}
 \p_{x,y} \boxtensor \p_{y,z} \to c^*_{x,y,z} \p_{x,z}.
 \end{equation} 

Given a space $X$ and a $\Z$-local system $\eta_X$ on $X$ we write $C_*(X, \eta_X)$ for the chain complex of singular chains with coefficients in $\eta_X$. The Eilenberg-Zilber map then gives a homology equivalence $C_*(X, \eta_X) \tensor C_*(Y, \eta_Y) \to C_*(X \times Y, \eta_X \boxtensor \eta_Y)$ for any pair of topological spaces $X, Y$ equipped with $\Z$-local systems $\eta_X, \eta_Y$. We define the dg category $C_*(\P L, \p)$ with objects the objects of $\P L$ and morphism complexes
\[ C_*(\P L, \p) (x,y) = C_*(\P L(x,y), \p_{x,y}) \]
with composition given by the Eilenberg-Zilber map followed by the map 
\begin{align*} C_*(\P L(x,y) \times \P L(y,z), \p_{x,y} \tensor_\Z \p_{y,z}) &\simeq C_*(\P L(x,y) \times \P L(y,z), c^*_{x,y,z} \p_{x,z}) \\ &\xrightarrow{c_{x,y,z}} C_*(\P L(x,z), \p_{x,z}). 
\end{align*}
This is associative because of the associativity property of the gluing maps for Pin structures relative to the ends, as explained in Appendix \ref{sec:gluing_pin_structures_relative_to_ends}. 

\begin{remark}
\label{rk:expanding_C}
The above construction of $C_*(\P L, \p)$ depends on the choices of Pin structures in (\ref{eq:choose_Pin_structures}). Given a different choice of Pin structures resulting category $C_*(\P L, \p)'$, one gets an isomorphism of dg categories by choosing isomorphisms between the two choices of Pin structure for every point of $L_0$ and of $L_1$. The non-canonicity of the choice in (\ref{eq:choose_Pin_structures}) does not affect any of the results of the paper.
\end{remark}

\subsection{Twisted fundamental group(oid)}
\label{sec:twisted-fundamental-group}
Let $L$ be a manifold with a basepoint $x \in L$ and a choice of Pin structure on $T_bL$. Then the local system over $\P_{x,x}L$
\[ \p_{x,x} = \left|\bigcup_{\gamma \in \P_{x,x}L} \Pi(\gamma^*TL)\right| \]
defines a characteristic class in $[\p_{x,x}] = w_1(\p_{x,x}) \in H^1(\P_{x,x}, \Z/2)$. This characteristic class can also be computed as follows: there is an evaluation map 
\[ \P_{x,x}L \times S^1 \to L \]
giving a pullback map
\[ H^*(L; \Z/2) \to H^*(\P_{x,x}L \times S^1; \Z/2)  = H^*(\P_{x,x}L; \Z/2) \oplus H^{*-1}(\P_{x,x}L; \Z/2); \]
let $\Omega: H^*(L; \Z/2) \to H^{*-1}(\P_{x,x}L; \Z/2)$ be the composition of the above map with the projection to the second component.
\begin{proposition}There is an equality of cohomology classes
\[[\p_{x,x}] = \Omega w_2(L). \]
\end{proposition}
\begin{proof}
The monodromy $\p_{x,x}$ along a loop $r: S_1 \to \P_{x,x}L$ corresponding to a map $\lambda: S_1 \times S_1 \to L$ is trivial exactly when the coresponding bundle $T^*\lambda$ admits a global Pin structure, which is exactly when $\lambda^*w_2(L) = r^* \Omega w_2(L) = 0$.
\end{proof}

Arguing as in Section \ref{sec:natural-local-system}, we can define an associative unital multiplication
\begin{equation}
\label{eq:multiplication}
c: C_*(\P_{x,x}L, \p_{x,x})^{\tensor 2} \to C_*(\P_{x,x}L, \p_{x,x})
\end{equation} 
by concatenating Pin structures relative to the ends. (See Remark \ref{rk:unital-pin-structure} for a description of the unit in this algebra.)  This is the ``twist'' of the algebra of chains on the based loop space which was mentioned in the introduction.

\subsubsection{Taking $H_0$}
\label{sec:taking-h0}
If $\Omega w_2(L) = 0$ then $\p = \p_{x,x}$ is a trivial local system on each connected component of $\P_{x,x}L$. In that case $\p$ is the pullback of a local system $\p_{0}$ on $\pi_0(\P_{x,x}L) = \pi_1(L)$ by the map sending a path to the connected component it lies in, and the operation of gluing Pin structures relative to the ends actually gives an isomorphism
\[ a_0: \p_{0} \boxtensor \p_0 \to \alpha_0^* \p_0 \]
where $\alpha_0$ is the multiplication on the group $\pi_0(\P_{x,x}L)$. Furthermore, $a_0$ is associative in the sense that $a_0(1 \times a_0) = a_0(a_0 \times 1)$. Moreover, the map
\begin{equation} H_0(\P_{x,x}L, \p) \to H_0(\pi_1(L), \p_0) =: \Z[\pi_1(L)]^{tw}
\end{equation}
is an isomorphism of rings. The latter ring is a  \emph{twisted fundamental group ring of $L$} in the following sense: a choice of an identification of abelian groups $H_0(\pi_1(L), \p_0) \simeq \Z \pi_1(L)$ gives a ring structure on the latter group of the form
\[ [\gamma_1][\gamma_2] = (-1)^{\mu([\gamma_1], [\gamma_2])} [\gamma_1 \cdot \gamma_2] \]
where $\gamma_i \in \P_{x,x}L$ , $(\cdot)$ denotes concatenation of paths, and $\mu: \pi_1(L) \times \pi_1(L) \to \Z/2$ is a certain group $2$-cocycle for the group cohomology $H^2(\pi_1(L), \Z/2)$. Different choices of identification between $H_0(\pi_1(L), \p_0)$ and $\Z \pi_1(L)$  cause the cocycle $\mu$ to change by a coboundary. 

The above construction has a categorical generalization:
\begin{definition}
	\label{def:ho-pl}
	Let $H_0(\P L, \p)$ be the category enriched in abelian groups obtained 
	 by applying the monoidal functor $H_0$ to the morphism complexes of the category $C_*(\mathcal{P}L, \mathfrak{p})$.
\end{definition}

There is a \emph{projection functor} $C_*(\P L, \p) \to H_0(\P L, \p)$ which acts by the identity on objects, and on morphisms, sends zero-chains to their corresponding homology class, and sends chains of positive dimension to zero.
\subsection{Bimodule structure}
\label{sec:bimodule-structure}
For each $y \in \mathcal{C}$, the complex $C_*(\P L, \p)(y, y_b)$ is a right $C_*(\P L, \p)(y_b, y_b)$-module, and given any element $\gamma \in C_*(\P L, \p)(y', y)$, the map 
$C_*(\P L, \p)(y, y_b) \to C_*(\P L, \p)(y', y_b)$ given by left-composition with $\gamma$ is a map of $C_*(\P L, \p)(y_b, y_b)$-modules. Moreover, the standard argument proving that the fundamental groupoid of a space is a groupoid adapts to prove the elementary 
\begin{lemma}
	\label{lemma:PL-is-groupoid}
	If $\gamma = (\gamma_0, \gamma_1) \in C_*(\P L, \p)(y', y)_0$ is a degree zero morphism given by a \emph{single} pair of paths equipped with Pin structures relative to their ends, then left-composition with $\gamma$  map is a homotopy equivalence; a homotopy inverse is given by left-composition with the element $\gamma^{-1} \in C_*(\P L, \p)(y', y)_0$ given by the pair $(i^*\gamma_0, i^*\gamma_1)$, where $i: [0, \ell] \to [0, \ell]$ is the affine map that reverses the parametrization of a Moore path, and the Pin structure relative to the ends on the Moore path $i^*\gamma_j$ is given by the pullback by $i$ of that on $\gamma_j$. $\qed$
\end{lemma}

We now describe a nice model for $C_*(\P L, \p)(y_b, y_b)$. Our conventions for tensor products of $dg$-algebras, etc., are stated in Appendix \ref{sec:dg-algebra-conventions}. 
\begin{lemma}
	\label{lemma:left-compostion-is-module-map}
	Define 
	\begin{equation}
	\label{eq:base-rings}
	A_i := C_*(\P_{y_b(0),y_b(0)}(L_0, \p)), \text{ for } i = 0,1.
	\end{equation}
	There is a map of $dg$-algebras
	\begin{equation}
	\label{eq:tensor-product-to-hom-map}A_0^{op} \tensor A_1 \to C_*(\P L, \p)(y_b, y_b) 
	\end{equation}
	defined as the composition of the map that flips the directionality of simplices of paths in the first factor with the Eilenberg-Zilber map. This map is a quasi-isomorphism of $dg$-algebras.

	Thus, composition with $\gamma \in C_*(\P L, \p)(y', y)$ gives a map of right $A_0^{op} \tensor A_1$ modules, and thus a map of $(A_0,  A_1)$-bimodules. In particular, in view of Lemma \ref{lemma:PL-is-groupoid}, for all objects $y$, the bimodules 
	$C_*(\P L, \p)(y, y_b)$ are quasi-isomorphic to rank $1$ free $(A_0, A_1) $-bimodules.
\end{lemma} 

\begin{remark}
	\label{rk:mod-2-graded}
	One can view the category $C_*(\P L, \p)$ and as a $\Z/2$-graded $dg$-category instead of a $\Z$-graded $dg$-category. 
	Everything in Section \ref{sec:natural-category} makes sense when stated with $\Z/2\Z$-graded chain complexes with $\Z$-graded chain complexes. For the rest of the paper, the notation $C_*(\P L, \p)$ will refer to the $\Z/2\Z$-graded version of the above constructions.
\end{remark}
\begin{remark}
	\label{rk:homotopy-vs-qis}
	The map (\ref{eq:tensor-product-to-hom-map}) is a quasi-isomorphism because the Eilenberg-Zilber map is a quasi-isomorphism. The method of acyclic models is shows in fact that the Eilenberg-Zilber map is a homotopy equivalence; this is probably also true in this setting, but we do not verify this. 
\end{remark}
\section{A simplified construction}
\label{sec:simplified-construction}
In this section, we define certain Floer complexes associated to non-Pin Lagrangians that allow us to prove Propositions \ref{prop:example} and \ref{prop:example2}. These will only be defined in the restricted setting of those propositions, and they will be significantly easier to define and compute with than those complexes needed to prove Proposition \ref{prop:resolution}; the latter complexes are defined in Section \ref{sec:floer-theory-start}. 

\subsection{Setup}
\label{sec:simplified-setup}
Let $(M, \omega, \theta)$ be an \emph{exact symplectic manifold with convex boundary}, or a \emph{Liouville domain}: namely, $M$ is a manifold with boundary, $\omega$ is a symplectic form on $M$, $\theta$ is a $1$-form on $M$ satisfying $d\theta = \omega$, and the Liouville vector field $X$ defined by the equation 
$i_X \omega = \theta$ points outwards on $\partial M$. In an open neighborhood $U$ of $\partial M$, there is a canonically-defined function $h: U \to \R$ characterized by the requirement that $h^{-1}(1) = \partial M$ and $X.h = h$. Let $L_0, L_1$ be closed exact Lagrangian submanifolds of $M$. 

In Appendix \ref{sec:technical-floer-appendix}, we state our conventions about Floer-theoretic moduli spaces. We now make a \emph{choice}:
\begin{equation}
\label{eq:chose-regular-floer-data}
\text{Choose regular Floer data }(H, J)\text{ for }L_0, L_1.
\end{equation}
In Appendix \ref{sec:floer-moduli-spaces} we recall that for every pair of time-1 Hamiltonian chords $\gamma_\pm \in \mathcal{C}(L_0, L_1; H)$ from $L_0$ to $L_1$, we have a moduli space of broken Floer trajectories $\overline{\M}^F(\gamma_-, \gamma_+)$.  

In Appendix \ref{sec:def-orientation-lines} we review the theory of orientation lines in Lagrangian Floer theory as developed e.g. in \cite{seidel}, and introduce some notation. Specifically, we have, for every Hamiltonian chord $y \in \mathcal{C}(L_0, L_1; H)$, 
\begin{itemize}
	\item an \emph{orientation line} $o^n(y)$ for every integer $n$, which arises as the determinant line of an Cauchy-Riemann operator of Fredholm index $n$ on a disk with with one incoming boundary puncture;
	\item and a \emph{shift line} $o^S(y)$, which arises as the determinant line of a Cauchy-Riemann operator on a strip with one outgoing and one incoming puncture, and which is constructed so that gluing of determinant lines of Cauchy-Riemann operators gives a canonical isomorphism 
	\begin{equation}o^{n+1}(y) \simeq o^S(y) \tensor o^{n}(y).
	\end{equation}
\end{itemize} 

We now make the following \emph{choices}:
\begin{equation}
\label{eq:pin-structures-choice-floer}
\begin{minipage}{0.9\textwidth}
Choose basepoints for $L_0$ and $L_1$. Choose Pin structures at tangent spaces to $L_0$, $L_1$ at the endpoints of all Hamiltonian chords from $L_0$ to $L_1$, as well as at the respective basepoints.
\end{minipage}
\end{equation}

These choices allow us to define the category $C_*(\mathcal{P}L, \mathfrak{p})$ discussed in Section \ref{sec:natural-local-system}, where $L= (L_0, L_1)$, and the chosen pair of basepoints define the object $y_b$. Given any Hamiltonian chord $y \in \mathcal{C}(L_0, L_1; H)$, we will abuse notation and let $y$ denote the corresponding object $(y(0), y(1))$ of $C_*(\mathcal{P}L, \mathfrak{p})$. This section will focus on the associated category $H_0(\mathcal{P}L, \mathfrak{p})$ defined in Section \ref{sec:taking-h0}.

Finally, make one last \emph{choice}
\begin{equation}
\label{eq:trivialize-shift-line}
\begin{minipage}{0.9\textwidth}
For every Hamiltonian chord $y \in C(L_0, L_1; H)$, choose a trivialization of the shift line $o^S(y)$.
\end{minipage} 
\end{equation}

For the remainder of this section, we make an \emph{assumption}: 
\begin{equation}
\label{assumption:trivial-local-system}
\begin{minipage}{0.9\textwidth}
Assume that the characteristic classes $\Omega w_2(L_0), \Omega w_2(L_1)$ defined in Section \ref{sec:twisted-fundamental-group} are zero.
\end{minipage}
\end{equation}

\subsection{The complex}
\label{sec:floer-complex-simple}

Given the assumptions and choices of the previous section, define the free abelian group
\begin{equation}
\label{eq:floer-complex-simple}
CF(L_0, L_1; \Z[\pi_1(L_0)]^{tw} \tensor \Z[\pi_1(L_1)]^{tw}; H, J) := \bigoplus_{y \in \mathcal{C}(L_0, L_1; H)} \Z^\check_y \tensor o^0(y) \tensor H_0(\mathcal{P}L, \mathfrak{p}))(y, y_b).
\end{equation}

\begin{remark}
The notation $\Z^\check_y$ denotes a \emph{free rank-1 abelian group of $\Z/2\Z$-degree $1$ with the label $y$}, here and throughout the paper. To decrease the complexity of the signs, we will incorporate many sign manipulations into the signs implicitly introduced by commuting graded lines past one another, and explicit degree-shifting isomorphisms of trivial lines. We describe our conventions on graded lines in Appendix \ref{sec:orientation-conventions}.
\end{remark}

Given $\gamma \in CF(L_0, L_1; \Z[\pi_1(L_0)]^{tw} \tensor \Z[\pi_1(L_1)]^{tw}; H, J)$, we will write $\gamma_y$ for the component of $\gamma$ in $\Z^\check_y \tensor o^0(y) \tensor H_0(\mathcal{P}L, \mathfrak{p}))(y, y_b)$.

\begin{figure}[h]
			\centering
	\resizebox{\textwidth}{!}{
		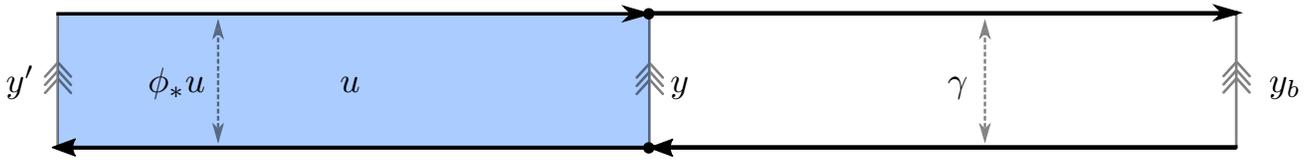}
	\caption{The differential in the Floer complex. Shaded blue region denotes a holomorphic strip. }
\label{fig:illustration-of-differential-a}
\end{figure}
\begin{figure}[h]
		\centering
		\resizebox{\textwidth}{!}{
			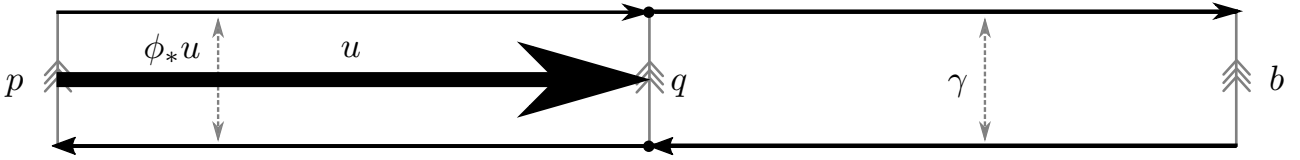}
\caption{The differential in the Morse complex. Thick black arrow denotes a negative gradient flow line. }
\label{fig:illustration-of-differential-b}
\end{figure}


This abelian group admits a differential defined as follows. Linear gluing theory for Cauchy Riemann operators, outlined in Appendix \ref{sec:gluing_and_orientations}, shows that the fiber of the orientation local system of $\overline{\M}^F(y', y)$ at a point $u$, where $u$ is an index-$n$ solution to Floer's equation, is canonically isomorphic to 
\begin{equation}
\label{eq:outcome_of_linear_gluing_theory}
\Z_{\partial/\partial s}^\check \tensor o^n(y') \tensor o^0(y)^\check \tensor \mathfrak{p},
\end{equation}
where $\mathfrak{p}$ is the abelian group of Pin structures along the boundary of $u$.  The orientation local system of a zero-dimensional manifold is canonically trivial, so we can view an index-$1$ solution $u$ to Floer's equation, with Hamiltonian chords $y', y$ at the positive and negative ends, as giving an element, 
\begin{equation}
\label{eq:phi-k-simple}
\overline{\phi}^k_* u \in \Z_{\partial/\partial s}^\check  \tensor o^{k+1}(y') \tensor o^k(y)^\check \tensor C_*(\P L, \p)(y', y)_0.
\end{equation}  
Here we use the case $k = 0$. Using the gluing isomorphism and the choice made in Equation \ref{eq:trivialize-shift-line}, we have a canonical isomorphism $o^1(y') \simeq o^S(y') \tensor o^0(y) \simeq o(y)$. Combining this isomorphism with the projection functor $C_*(\P L, \p) \to H_0(\P L, \p)$, we get that points of $\overline{\M}^F(y', y)$ lying on zero-dimensional components give elements 
\[ \phi_* u \in \Z_{\partial/\partial s}^\check  \tensor o^0(y') \tensor o^0(y)^\check \tensor Hom_{H_0(\P L, \p)}(y', y). \]
Given an element in $\gamma_y \in \Z^\check_y \tensor o^0(y) \tensor H_0(\mathcal{P}L, \mathfrak{p}))(y, y_b)$, we can define another element
\begin{equation} 
\label{eq:def-star1-simple}
\phi_* u *_1 \gamma_y \in CF(L_0, L_1; \Z[\pi_1(L_0)]^{tw} \tensor \Z[\pi_1(L_1)]^{tw}; H, J)
\end{equation}
by composing the tensor factors of $\phi_* u$ and $\gamma$ that are morphisms in $H_0(\P L, \p)$, commuting the line $\Z^\check_y$ to the left, pairing the orientation line $o^0(y')$ with its dual,  and applying the \emph{grading-shifting} isomorphism of trivial lines $\Z^\check_y \tensor \Z_{\partial/\partial_s} \simeq \Z^\check_{y'}$.

Thus define the operator $d$ on $CF(L_0, L_1; \Z[\pi_1(L_0)]^{tw} \tensor \Z[\pi_1(L_1)]^{tw}; H, J)$ to be 
\begin{equation}
\label{eq:differential-simplified} d  \gamma = \sum_{y, y' \in \mathcal{C}(L_0, L_1; H)}\sum_{\substack{ u \in \overline{\M}^F(y', y); \\ \text{ind } u = 1}} \phi_* u *_1 \gamma_y.
\end{equation}
Refer to Figure \ref{fig:illustration-of-differential-a} for a graphical description of this operator. 
To show that this operator defines a chain complex, we must verify the
\begin{lemma}
	\label{lemma:d2-zero-simple}
	The above operator defined in (\ref{eq:differential-simplified}) satisfies \[d^2 = 0.\]
\end{lemma}
\begin{proof}
	Let $u \in \M^F(y', y)$, $u' \in \M^F(y'', y)$ be a pair of index $1$ solutions to Floer's equation. Given a solution $w$ to Floer's equation, let $D_u$ denote the linearization of Floer's equation or the inhomogeneous pseudoholomorphic map equation at $u$; this is a Cauchy-Riemann operator with totally real boundary conditions, and is Fredholm on an appropriate Banach space. We will call $\det D_u$ (see Eq. \eqref{eq:determinant-line-def} in Section \ref{sec:regularity}) the determinant line of $u$. 
	
	The proof that the differential on the usual Floer complex for spin exact Lagrangians squares to zero can be summarized as follows (see \cite[Section~II.12f]{seidel}): Linear gluing theory of Cauchy-Riemann operators with totally real Pin boundary conditions gives an isomorphism $\lambda_{u, u'}: \ker D_{u'} \oplus \ker D_u \simeq \ker D_{u' \# u}$, where $u' \# u$ is an index $2$ solution to Floer's equation constructed by gluing $u', u$ via some gluing parameter. The translation action on solutions to Floer's equation gives canonical elements $(\frac{\partial}{\partial s})_{w} \in \ker D_w$ for any solution to Floer's equation; linear gluing theory then shows that $(\frac{\partial}{\partial s})_{u'} + (\frac{\partial}{\partial s})_{u}$ is sent to an \emph{inwards-pointing vector} in $\ker D_{u' \# u}/(\frac{\partial}{\partial s})_{u\# u} \simeq T_{(u' \# u)}\M^F(y'', y)$ by $\lambda_{u, u'}$. Thus, using $1$-dimensional moduli spaces to compare the contributions to $d^2$, one sees that they cancel in pairs.  
	
	This lemma follows from exactly the same argument. Indeed, given $u', u$, one can consider 
	\[\phi_* u' *_1 \phi_* u \in \Z^\check_{\partial/\partial s} \tensor o^0(y'')^\check \tensor o^0(y) \tensor Hom_{H_0(\P L, \p)}(y'', y)\]
	where $*_1$ is defined as in Equation (\ref{eq:def-star1-simple}) by commuting the right-most trivial line to the left, applying a grading-shifting isomorphism to the trivial lines, and simplifying the orientation lines and $Hom$s. Similarly, it makes sense to define 
	\[\overline\phi^1_* u' *_1 \overline\phi^0_* u \in \Z^\check_{\partial/\partial s} \tensor o^2(y'') \tensor o^0(y)^\check \tensor Hom_{H_0(\P L, \p)}(y'', y).\]
	Using the shift line we can exchange the $o^2$ for an $o^0$; this gives the same element as $\phi_* u' *_1 \phi_* u$ because gluing theory is equivariant with respect to tensoring by shift lines (see (\ref{eq:equivariance-of-gluing}), Appendix \ref{sec:gluing_and_orientations}). Now the ``inwards-pointing vector'' in $T_{(u' \# u)}\M^F(y'', y)$ orients the tangent space to $\M^F(y'', y)$ at $u' \# u$, and so gives, by linear gluing theory, an element 
	\[\phi^0_* (u' \# u) \in (\Z_{\partial/\partial s})^\check \tensor o^2(y'') \tensor o^0(y) \tensor  Hom_{H_0(\P L, \p)}(y'', y),\] 
	and the argument sketched in the previous paragraph shows \emph{precisely} that $\overline\phi^1_* u' *_1 \overline\phi^0_* u = \phi^0_* (u' \# u)$.  So we have proven that that for pairs $(u_i', u_i)_{i = 1, 2}$ lying on the boundary of a $1$-dimensional moduli space of Floer trajectories, the elements $\phi_* u_i' *_1 \phi_* u_i$ have opposite sign in $o^0(y'')^\check \tensor o^0(y) \tensor Hom_{H_0(\P L, \p)}(y', y)$. Since $(\phi_* u' *_1 \phi_* u) *_1 \gamma = \pm \phi_* u' *_1 (\phi_* u *_1 \gamma)$ with a sign independent of $u, u', \gamma$, this shows that the contributions to the differential cancel in pairs. 
\end{proof}

\subsubsection{Continuation maps : definitions}
Analogously to usual Floer theory, we can define continuation, and homotopies between continuation maps. First, we fix terminology.
\begin{definition}
	Let $\mathcal{C}$ be a category linear over $\Z$. 
	
	An \emph{ungraded chain complex} of objects in a category $\mathcal{C}$ is an element $V \in Ob(\mathcal{C})$ equipped with an operator $d \in End_{\mathcal{C}}(V)$ with $d^2 = 0$. Without qualification, an ungraded chain complex is an ungraded chain complex of abelian groups.
	
	A \emph{map of ungraded chain complexes} $(V, d_V) \to (W, d_W)$ is a map $C: V \to W$ satisfying $Cd_V = d_W W$. These define $Ch^u(\mathcal{C})$, the category of ungraded chain complexes in $\mathcal{C}$.
	
	A \emph{homotopy} of maps $C_1, C_2: (V, d_V) \to (W, d_W)$ is a map in $\mathcal{C}$, $H: V \to W$, satisfying $C_1 - C_2 = d_W H - Hd_V$. The existence of a homotopy defines an equivalence relation on maps of ungraded complexes in $\mathcal{C}$, and the composition of maps respects this equivalence relation. Define $Ho(Ch^u(\mathcal{C})))$, the \emph{homotopy category of ungraded chain complexes of objects in} $\mathcal{C}$, to be the quotient of $Ch^u(\mathcal{C})$ by this equivalence relation.
\end{definition}
We now recall a notion originally described by Conley, which gives a convenient way to state the sense in which Floer homology is an invariant:
\begin{definition}
	Given a set of ungraded chain complexes $\{V_s\}_{s \in \mathcal{S}}$ of objects in a category $\mathcal{C}$ depending on some auxiliary data $s$ ranging over some set $\mathcal{S}$, we say that $V_s$ form a \emph{connected simple system} if there is a functor $F: Gr_{\mathcal{S}} \to Ho(Ch^u(\mathcal{C}))$, where $Gr_{\mathcal{S}}$ is the category with objects $S$ and one morphism between every pair of objects, and $F(s) = V_s$. 
\end{definition}

We now proceed to define the continuation map between Floer complexes associated to different choices of regular Floer data $(H_0, J_0)$ and $(H_1, J_1)$. Equip $Z$ with boundary Floer data given by $(H_1, J_1)$ at the positive end and $(H_0, J_0)$ at the negative end, and choose a regular Floer datum on $Z$ compatible with these perturbation data.
 In Appendix \ref{sec:floer-moduli-spaces}, we recall that for every pair  $y' \in  \mathcal{C}(L_0, L_1; H_0),  y \in \mathcal{C}(L_0, L_1; H_1)$ and regular perturbation datum $(K, J)$ on $Z$ compatible with the boundary Floer data, there is a moduli space $\M^C(y', y)$ of solutions to the inhomogeneous pseudoholomorphic map equation with Gromov compactification $\overline{\M}^C(y', y)$. Linear gluing theory for Cauchy Riemann opeartors shows that the fiber of the orientation local system at a point $u \in \overline{\M}^C(\gamma_-, \gamma_+)$, where $u$ is an index $n$ solution, is canonically isomorphic to $o^n(y') \tensor o^0(y) \tensor \p$, where $\p$ is the abelian group of Pin structures along the boundary of $u$.  Thus, using the isomorphism $o^n(y') \simeq o^0(y')$ coming from the choice of orientation of the shift line of $y'$, we see that an index $0$ element $u \in \overline{\M}^C(y', y)$ defines an element 
\[ \phi_* u: o^0(y') \tensor o^0(y)^\check \tensor Hom_{H_0(\P L,\p)}(y', y). \]
For every pair $y_0, y_1 \in \mathcal{C}(L_0, L_1; H_i)$, and for every $u$ as above, and any $\gamma_y \in \Z^\check_y \tensor o^0(y) \tensor H_0(\mathcal{P}L, \mathfrak{p}))(y, y_b),$ let 
\begin{equation}
\label{eq:def-star2-simple}
\phi_* u *_2 \gamma \in  CF(L_0, L_1; \Z[\pi_1(L_0)]^{tw} \tensor \Z[\pi_1(L_1)]^{tw}; H_0, J_0)
\end{equation}
be the element obtained by composing in $H_0(\mathcal{P}L, \mathfrak{p})$, moving the trivial line $\Z^\check_y$ to to the left and identifying it with $\Z^\check_{y'}$, and pairing off the orientation lines for $y$. Define a map
\begin{equation}
C: CF(L_0, L_1; \Z[\pi_1(L_0)]^{tw} \tensor \Z[\pi_1(L_1)]^{tw}; H_1, J_1) \to CF(L_0, L_1; \Z[\pi_1(L_0)]^{tw} \tensor \Z[\pi_1(L_1)]^{tw}; H_0, J_0),
\end{equation}
depending on a choice of regular perturbation datum $(K, J)$, by 
\begin{equation}
C \gamma = \sum_{y, y' \in \mathcal{C}(L_0, L_1; H)}\sum_{\substack{ u \in \overline{\M}^C(y', y); \\ \text{ind } u = 2}} \phi_* u *_2 \gamma_y.
\end{equation}

Given a pair of choices $(K_0, J_0), (K_1, J_1)$ of perturbation data on $Z$, as in the previous paragraph, with corresponding maps $C_0, C_1$ between Floer complexes, we will now define maps
\begin{equation}
H:CF(L_0, L_1; \Z[\pi_1(L_0)]^{tw} \tensor \Z[\pi_1(L_1)]^{tw}; H_1, J_1) \to CF(L_0, L_1; \Z[\pi_1(L_0)]^{tw} \tensor \Z[\pi_1(L_1)]^{tw}; H_0, J_0).
\end{equation}
which we will show to be homotopies between $C_0, C_1$. To do this we invoke the following standard
\begin{proposition}
	\label{prop:homotopy-moduli-space}
	There exists a family of perturbation data $(K_t, J_t)$, $t \in [0,1]$ compatible with the same boundary Floer data, such that at $t=0$ and $t=1$ these perturbation data agree with the specified perturbation data $\{(K_0, J_1)\}_{i = 0, 1}$, and such that the space
	\[ \M^H(\gamma^+, \gamma^-; Z, \mathfrak{B}, K_t, J_t) := \cup_{t \in (0,1)} \M^C(\gamma^+, \gamma^-; Z, \mathfrak{B}, K_t, J_t), \]
	equipped with the topology induced from its inclusion into $C^0(Z, M) \times (0, 1)$, is a disjoint union of smooth manifolds with each connected component of dimension equal to one more than the indices of the maps comprising the component. 
\end{proposition}
\begin{proposition}
	\label{prop:homotopy-moduli-space-compactification}
	The Gromov compactification of 
	\begin{equation*}\overline{\M}^H(\gamma^+, \gamma^-; Z, \mathfrak{B}, K_t, J_t)
	\end{equation*} 
	of $\M^H(\gamma^+, \gamma^-; Z, \mathfrak{B}, K_t, J_t)$ is a disjoint union of topological manifolds with corners (see Appendix \ref{sec:topological-preliminaries} for a definition of the latter), with the union of the codimension $1$ strata of each connected component equal to 
	\begin{align*} &\M^C(\gamma^+, \gamma^-; Z, \mathfrak{B}, K_0, J_0) 
	\cup \M^C(\gamma^+, \gamma^-; Z, \mathfrak{B}, K_1, J_1) \;\cup \\
	&\M^H(\gamma^+, \gamma^0; Z, \mathfrak{B}, K_t, J_t) \times \M^F(\gamma^0, \gamma^-; H_1, J_1)\; \cup \\ & 
	\M^F(\gamma^+, \gamma^0; H_0, J_0) \times \M^H(\gamma^0, \gamma^-; Z, \mathfrak{B}, K_t, J_t) 
	\end{align*}
\end{proposition}

Choose perturbation data as in Proposition \ref{prop:homotopy-moduli-space}. Then every index $-1$ solution $u \in \overline{\M}^H(y', y)$ lying over $t \in (0,1)$ lies on a zero-dimensional component of $\overline{\M}^H(y', y)$, so $\ker D_u$ is zero dimensional and $\coker D_u$ has rank $1$. The standard transversality proof of Proposition \ref{prop:homotopy-moduli-space} shows that $u$ is transversely cut out a section of a a smooth vector bundle over an open neighborhood of $t$ with fiber $\coker D_u$, so the vector $\partial/\partial_t$ gives an element of $(\coker{D_u})^\check$ and thus an orientation of $\det D_u$. Then linear gluing theory gives us a canonical element in $o^{-1}(y') \tensor o(y) \tensor \p$ where $\p$ is the abelian group of Pin structures along $u$. The choice of orientation of the shift line then gives an element
\[ \phi_* u \in o^0(y') \tensor o^0(y) \tensor Hom_{H_0(\P L, \p)}(y', y)\]
corresponding to $u$. We define
\begin{equation}
H \gamma = \sum_{y' \in \mathcal{C}(L_0, L_1; H)}\sum_{\substack{ u \in \overline{\M}^H(y', y); \\ \text{ind } u = 1}}\phi_* u *_2 \gamma_y
\end{equation}
where $*_2$ is defined as in Equation \ref{eq:def-star2-simple}. 

We now prove that the Floer complex we have defined gives a well-defined invariant:
\begin{lemma}
	$C$ is a map of ungraded chain complexes, and $H$ is a homotopy of maps. With this structure, the complexes $\{CF(L_0, L_1; \Z[\pi_1(L_0)]^{tw} \tensor \Z[\pi_1(L_1)]^{tw}; H, J)\}_{(H, J)}$ form a connected simple system of ungraded complexes of abelian groups.
\end{lemma}
\begin{proof}
	This is very similar to the proof of Lemma \ref{lemma:d2-zero-simple}. 
    To prove that $C$ is a map of ungraded chain complexes, one analyses $1$-dimensional components of $\overline{\M}^C(y'', y)$. Gluing theory shows that if $u' \in \overline{\M}^C(y'', y')$ has index zero and  $u \in \overline{\M}^F(y', y)$ has index $1$, then gluing $u'$ and $u$ to a curve $u' \# u$  sends the translation vector $\partial/\partial s$ to the outwards-pointing normal vector in $T_{u' \# u}\overline{\M}^C(y'', y')$, while if $u' \in \overline{\M}^F(y', y)$ has index one and $u \in \overline{\M}^C(y', y)$ has index zero, then gluing these sends $\partial/\partial s$ to the inwards pointing vector. These translate to the statement that 
    \[ \phi_* u' *_2 \phi_* u  \in \Z^\check_{y''} \tensor o^0(y'') \tensor o^0(y)^\check \tensor Hom_{H_0(\P L, \p)}(y'', y)\]
    are exactly those elements in $ o^1(y'') \tensor o^0(y)^\check \tensor Hom_{H_0(\P L, \p)}(y'', y)$ constructed out of $u' \# u$ via linear gluing by orienting $T_{u' \# u}\overline{\M}^C(y'', y')$ using the appropriate inwards/outwards vectors, and then subsequently inserting a trivial line on the left and applying the isomorphism of orientation lines $o^1(y'') \simeq o^0(y)$. This means in turn that curves contributing to $dC\gamma $ either cancel among themselves or give the same contribution as a corresponding curve in $C d \gamma$, and similarly for the curves contributing to $d C \gamma$. This proves that $C$ is a map of ungraded chain complexes.
	
    To show that $H$ is a homotopy, one reasons analogously, by analyzing $1$-dimensional components of $\overline{\M}^H(y'', y)$. The claim translates into the truth of two analytic statements. The first is that if $u$ is an index $0$ strip contributing to $C_0$ or $C_1$, then the tangent space to $\overline{\M}^H$ at any strip that is $C^0$-close to $u$ is canonically isomorphic to the span of $\partial/\partial_t$, where $t$ is the coordinate of $[0,1]$. So a $1$-dimensional component of $\overline{\M}^H$ breaking at a strip continuing to $C_0$ and a strip contributing to $C_1$ implies that the contributions of that pair of strips to $C_0 - C_1$ cancels. The second analytic fact is that given an index $1$ Floer trajectory $u$ and an index $-1$ strip $u'$ contributing to $H$, if one can glue $u'$ and $u$ to a curve $u' \# u$ then the translation vector field in $\ker D_u$ is sent to an outwards pointing vector field, while if one can glue $u \# u'$, then the translation vector field in $\ker D_u$ is sent to an inwards-pointing vector field. The key point is that the image of the translation vector fields for $u$ and for $u'$ under the linearization of the gluing map point in \emph{opposite} directions near the boundary. 
	
	Abbreviate $CF(L_0, L_1; \Z[\pi_1(L_0)]^{tw} \tensor \Z[\pi_1(L_1)]^{tw}; H, J)$ to $C_\bullet(H, J)$. To finish the proof one must show that given three regular Floer data $(H, J)_{i = 0, 1, 2}$ the composition of continuations $C: C_\bullet(H_2, J_2) \to  C_\bullet(H_1, J_1)$ and $C': C_\bullet(H_1, J_1) \to  C_\bullet(H_0, J_0)$ is homotopic to some continuation $C: C_\bullet(H_2, J_2) \to  C_\bullet(H_0, J_1)$. This follows again from gluing and compactness of all moduli spaces contributing to $C, C'$: there is a sufficiently small $\epsilon > 0$ so that the perturbation datum on $Z$ coming from gluing $(K_1, J_1)$ and $(K_0, J_0)$ with gluing parameter $\epsilon$ is regular, and has zero dimensional moduli spaces in bijection with terms in $\gamma \mapsto C' \circ C$. Using the glued perturbation datum defines the desired $C''$.
\end{proof}
\subsection{Reduction to Morse theory}
\subsubsection{A Morse-theoretic analog of the Floer complex}
\label{sec:morse-theory-analog}
We can imitate the construction of Section \ref{sec:floer-complex-simple} in Morse theory. In Appendix \ref{sec:morse-theory-technical} we recall basic language and theorems in Morse theory. In particular, given a Morse-Smale pair of metric $g$ and Morse function $f$, and a pair of critical points $p, q \in Crit(f)$, we write $\overline{\M}^M(p, q; f, g)$ for the compactified moduli space of Morse trajectories from $p$ to $q$ following the downwards gradient flow of $f$. So, to define the ``Morse version'' of the complex (\ref{eq:floer-complex-simple}) for the manifold $L_0$, we need to make the following \emph{choices}:
\begin{equation}
\label{eq:morse-smale-choice}
\begin{minipage}{0.9\textwidth}
Choose a Morse-Smale pair of function and metric $(f, g)$ on the manifold $L_0$.
\end{minipage}
\end{equation}
\begin{equation}
\label{eq:morse-theory-pin-choice}
\begin{minipage}{0.9\textwidth}
Choose a base point $b \in L_0$. \\
For every point $p \in L_0$, choose a Pin structure on $T_p L_0$. 
\end{minipage}
\end{equation}

The choice made in (\ref{eq:morse-theory-pin-choice}) lets us define the categories $C_*(\P L, \p)$ and $H_0(\P L, \P)$  as in Section \ref{sec:natural-local-system} for $L = (L_0, L_0)$, with objects $Crit(f) \sqcup b$. In Appendix \ref{sec:morse-theory-orientations} we recall that orientations in Morse theory involves choices of trivialization of the orientation lines $o(T^+_p)$ of the positive eigenspaces of the Hessian at critical points $p$. Define the graded vector space

\begin{equation}
CM(L_0, \Z[\pi_1(L_0)]^{tw} \tensor \Z[\pi_1(L_0)]^{tw}; f, g) = \bigoplus_{p \in Crit(f)} \Z^\check_p \tensor o(T^+_p) \tensor H_0(\mathcal{P}L, \mathfrak{p}))(p, b).
\end{equation}
\begin{definition}
	Given $\gamma \in CM(L_0, \Z[\pi_1(L_0)]^{tw} \tensor \Z[\pi_1(L_0)]^{tw}; f, g) $, let $\gamma_y$ denote its component in $\Z^\check_p \tensor o(T^+_p) \tensor H_0(\mathcal{P}L, \mathfrak{p}))(p, b).$ Define the grading of $\gamma_p$ to be the index of $p$.
\end{definition}

We define an operator $d$ acting on $CM(L_0, \Z[\pi_1(L_0)]^{tw} \tensor \Z[\pi_1(L_0)]^{tw}; f, g)$ making it into a chain complex. Standard orientation theory for Morse moduli spaces shows that 
\[ \mathfrak{o}(\overline{\M}^M(p, q; f, g)) \simeq \Z^\check_{\partial/\partial t} \tensor o(T_p^-) \tensor o(T_q^-)^\check. \]

\begin{definition}
\label{def:single-path-chain-category}
 Let $\Pi_1 L_0$ be the topological category of Moore paths on $L_0$ defined in Section \ref{sec:technicalities-on-path-spaces}. Note that \emph{morphisms in this category are single paths}, \textbf{not} pairs of paths on $L_0$.   Define $C_*(\Pi_1 L_0)$ to be the $dg$ category obtained by applying the monoidal functor of singular chains to $\Pi_1 L_0$. 
 
 Write  $\Z \Pi_1(L_0)$ for the category obtained by applying the monoidal functor $H_0$  from chain complexes to abelian groups to  $C_*(\Pi_1 L_0)$; thus morphisms in $\Z \Pi_1(L_0)$ are formal $\Z$-linear combinations of elements in the fundamental groupoid of $L_0$. 
\end{definition}
\begin{definition}
\label{def:delta-functor}

 Define the functors
 \begin{align}
 \Delta: C_*(\Pi_1 L_0) \to C_*(\P L, \p)\\
 \Delta: \Pi_1(L_0) \to H_0(\P L, \p)
 \end{align}
 by defining them on simplicies $\{\gamma_s\}_{s \in \Delta}$ of paths from $p$ to $q$, as follows. Given such a simplex, choose a continuous family $p_s$ of Pin structures relative to the ends on $\gamma_s$; this defines a corresponding continuous family $p^{-1}_s$ of Pin structures relative to the ends on the \emph{inverse paths} $\gamma^{-1}_s$ by the condition that under the affine map 
 \[i: [0,\ell] \to [0,\ell]\]
 with $i(t) = (\ell - t)$, one has $i^*p_s = p^{-1}_s$. 
 The functor
 \begin{equation}
 \label{eq:delta-functor-chains}
 \Delta: C_*(\Pi_1 L_0) \to C_*(\P L, \p)
 \end{equation} then sends the simplex of paths $\{\gamma_s\}_{s \in \Delta}$ to the element of $C_*(\P L, \p)(p, q)$ represented by 
 the simplex of pairs of inverse paths $\{\gamma_s, \gamma_s^{-1}\}_{s \in \Delta}$ from $q \to p, p \to q$, equipped with the Pin structures relative to the ends $\{(p_s, p^{-1}_s)\}_{s \in \Delta}$. Because the elements of the local system $\p$ come from a \emph{tensor product} of local systems associated to \emph{isomorphism classes} of Pin structures on paths relative to the ends, this element does not depend on the choice of $p_s$. The functor 
 \[\Delta: \Pi_1(L_0) \to H_0(\P L, \p)\]
 is the functor induced by taking applying $H_0$ to the functor defined in (\ref{eq:delta-functor-chains})
 \end{definition}

Then given a Morse trajectory $u \in \overline{\M}^{M}(p, q)$ of index difference equal to $1$, let 
 \[ \phi_* u \in \Z^\check_{\partial/\partial t} \tensor o(T_p^-) \tensor o(T_q^-)^\check \tensor H_0(\mathcal{P}L, \mathfrak{p}))(q, b)\]
 be the tensor product of the canonical trivialization of the orientation line of a zero-dimensional manifold with the image under $\Delta$ of the path traced by $u$ viewed as a morphism in $\Pi_1(L_0)$.  Given $\gamma_p \in \Z^\check_p \tensor o(T^+_p) \tensor H_0(\mathcal{P}L, \mathfrak{p}))(p, b)$, we can then make sense of 
 \[ \phi_* u *_1 \gamma_p \in \Z^\check_q \tensor o(T^-_q) \tensor H_0(\mathcal{P}L, \mathfrak{p}))(q, b)\]
 where $*_1$ is defined as in Equation \ref{eq:def-star1-simple} by replacing the the Floer orientation lines with the corresponding Morse orientation lines. We then define an operator
 \[d: CM(L_0, \Z[\pi_1(L_0)]^{tw} \tensor \Z[\pi_1(L_0)]^{tw}; f, g) \to CM(L_0, \Z[\pi_1(L_0)]^{tw} \tensor \Z[\pi_1(L_0)]^{tw}; f, g) \]
 \[ d \gamma = \sum_{\substack{p,q \in Crit(f)\\ |p| - |q| = 1}}\sum_{u \in \overline{\M}^M(p, q) } \phi_* u *_1 \gamma_p. \]
 Refer to Figure \ref{fig:illustration-of-differential-b} for a graphical description of this operator. We have 
 \begin{lemma}
 	\label{lemma:morse-complex-d2-simple}
 	The operator $d$ makes $CM(L_0, \Z[\pi_1(L_0)]^{tw} \tensor \Z[\pi_1(L_0)]^{tw}; f, g)$ into a chain complex. $\qed$
 \end{lemma}
This can be proven by analogy to the proof of Lemma \ref{lemma:d2-zero-simple}; moreover, one can define continuation maps and homotopies and show that these Morse complexes form a connected simple system of chain complexes. We do not provide these definitions here, and leave them to the interested reader. Instead, in Section \ref{sec:morse_complex}, we define a derived version of the complex defined in this section, and write out the chain maps and homotopies explicitly. In the next section, we give a direct comparison between this Morse complex and an an associated Floer complex; this comparison will prove Lemma \ref{lemma:morse-complex-d2-simple} as a corollary. 
 
\subsubsection{Floer-Morse Comparison}
Floer's original method of comparison between Lagrangian Homology and Morse theory used the following lemma
\begin{lemma}[\cite{floer1989}, Theorem 2]
	\label{lemma:floers-lemma}
	Let $L \subset T^*L$ be a compact manifold, thought of as a Lagrangian submanifold of the symplectic manifold $T^*L$. Choose a sufficiently $C^2$-small function $H$ on $L$ that is Morse-Smale with respect to a Riemannian metric $g$ on $L$. Let $\bar{H}$ denote the constant function $[0, 1] \mapsto H$. For every critical point $x \in Crit(H)$, let $\gamma_x$ denote the constant path $[0, 1] \mapsto x \in T^*L$. Then there exists a $J$ such that $(\bar{H}, J)$ is a regular Floer datum for $(L, L)$, and such that the map 
	\[ \M^F(\gamma_x, \gamma_y) \ni u \mapsto (\tau \mapsto u(\tau + 0i) = u(\tau + 1i)) \in \M^M(x, y))\]
	is a homeomorphism for every pair of critical points $x, y \in Crit(H)$. These bijections fit together to give stratum-preserving homeomorphisms
	\[ \overline{\M}^F(\gamma_x, \gamma_y) \to \overline{\M}^M(x, y).\]
\end{lemma}
\begin{remark}
	Floer proved a slightly different lemma, and did not make any assumptions about the Morse-Smale-ness of $H$ , but the above follows immediately from his proof. 
\end{remark}
This lemma was further extended by many people, for example, in the work of Fukaya-Oh on higher-genus curves in the cotangent bundle \cite{Fukaya1997}. By using this lemma, we prove the a Floer-Morse comparison result for the complexes that we have just defined:
\begin{lemma}
	\label{lemma:morse-floer-easy}
	Let $L_0$ be an exact Lagrangian submanifold of a Liouville domain with $\Omega w_2(L_0) = 0$.  
	Given the choices in (\ref{eq:morse-smale-choice}), (\ref{eq:morse-theory-pin-choice}), there exist regular Floer data $\bar{H}, \bar{J}$ for $(L_0, L_0)$, choices as in Eq. \ref{eq:pin-structures-choice-floer}, \ref{eq:trivialize-shift-line}, and a map
	\begin{equation}
	\begin{gathered}
	Y: CF(L_0, L_0; \Z[\pi_1(L_0)]^{tw} \tensor \Z[\pi_1(L_0)]^{tw};\bar{H}, \bar{J}) \to\\ CM(L_0, \Z[\pi_1(L_0)]^{tw} \tensor \Z[\pi_1(L_0)]^{tw}; f, g).
	\end{gathered}
	\end{equation}
	which is an isomorphism of abelian groups that commutes with the differentials on both sides.
\end{lemma}
\begin{proof}
	 Let $(H, g)$ be a Morse-Smale pair on $L = L_0$, and let $U \subset T^*L$ be a Weinstein neighborhood of $L \subset M$. Write $\pi: T^*L \to L$ for the projection map. Choose a $\overline{H}$ on $M$ that agrees with $\pi^*H$ on $U$, and a $t$-dependent family of almost complex structures $\overline{J}$ on $M$ which extends the restriction to $U$ of the family $J$ arising in the statement of Lemma \ref{lemma:floers-lemma}.  Then for a sufficiently small $H$, the image $L'$ of $L$ under the time-$1$ flow of $H$ still lies in $U$. The solutions to the inhomogeneous pseudoholomorphic map equation for the Floer datum $(\overline{H}, \overline{J})$ are bijection with pseudoholomorphic strips with boudary on $L \cup L'$. since $U$ is an exact symplectic manifold with contact type boundary, Lemma $II.7.5$ of \cite{seidel} shows that all such strips lie in $U$, and thus come from solutions to the inhomogeneous map equations with target $U \subset T^*L$ with respect to the Floer datum $(H, J)$. Thus the pair $\overline{H}, \overline{J}$ is regular, and Lemma \ref{lemma:floers-lemma} actually gives stratum-preserving homeomorphisms
	\[ \overline{\M}^F(\gamma_x, \gamma_y; \overline{H}, \overline{J}) \to \overline{\M}^M(x, y; H, g).\]
	
	The Floer data $(\bar{H}, \bar{J})$ described above are those in the proposition. We make the choice in  Eq. \ref{eq:pin-structures-choice-floer} by choosing the same Pin structures on both endpoints of the (constant) Hamiltonian chord, and we make the choice in Eq. \ref{eq:trivialize-shift-line} arbitrarily. 
	
	With these choices, there is an isomorphism (see, for example, Remark 6.1 of \cite{abouzaid_wrapped})
	\begin{equation}
	o^{ind x}(\gamma_x) \simeq o(T^-_x)
	\end{equation}
	for each critical point $x \in Crit(f)$. Thus, the isomorphism $o^0(\gamma_x) \simeq o^{\text{ind} x}$ coming from the trivialization of the shift line, together with the above isomorphism, and the identifications $\Z^\check_{\gamma_x} = \Z^\check_x$, define $Y$ as a map of abelian groups. 
	
	It remains to check that $Y$ commutes with the differential. The bijection between Floer strips and Morse trajectories gives a bijection between terms, and it suffices to check that the signs agree. Now, as for the Morse trajectories in Section \ref{sec:morse-theory-analog}, the abelian groups $\overline{p}$ of Pin structures along the boundaries of the Floer strips contributing to the Floer differential are all canonically isomorphic to $\Z$. The usual comparison between Floer and Morse signs, together with the equivariance of the gluing isomorphism with respect to the shift line (see Eq. \ref{eq:equivariance-of-gluing}), shows that the Floer signs agree with the Morse signs.
	
\end{proof}

\subsection{Module structures}
Recall that in Section \ref{sec:bimodule-structure} we explain how $C_*(\P L, \p)(y, y_b)$ is a right $A_0^{op} \tensor A_1$ module via a map 
\[A_0^{op} \tensor A_1 \to C_*(\P L, \p)(y_b, y_b).\] 
Applying $H_0$ to this map shows that $Hom_{H_0(\P L, \p)}(y, y_b)$ is a free right $H_0(A_0)^{op} \tensor H_0(A_1)$ module. Since
\begin{equation}
\begin{gathered}
H_0(A_i) = \Z[\pi_1(L_i)]^{tw} \text{ for } i = 0, 1, 
\end{gathered}
\end{equation} 
this is the same as a $(\Z[\pi_1(L_0)]^{tw}, \Z[\pi_1(L_1)]^{tw})$ bimodule, and left composition with a morphism in $Hom_{H_0(\P L, \p)}(y', y)$ is a map of bimodules.

The complexes \[CF(L_0, L_1; \Z[\pi_1(L_0)]^{tw} \tensor \Z[\pi_1(L_1)]^{tw};H, J)\]and \[CM(L_0, \Z[\pi_1(L_0)]^{tw} \tensor \Z[\pi_1(L_0)]^{tw}; f, g)\] are direct sums of tensor products of free rank $1$ abelian groups, i.e. ``lines'', with  $Hom$ spaces in $H_0(\P L, \p)$, and the differentials, homotopies, and continuation maps are defined as tensor products of isomorphisms of these lines with left compositions with morphisms in $H_0(\P L, \p)$. This immediately implies the following 
\begin{lemma}
	\label{lemma:floer-simplified-bimodule-invariance}
	The complexes \[\{CF(L_0, L_1; \Z[\pi_1(L_0)]^{tw} \tensor \Z[\pi_1(L_1)]^{tw}; H, J)\}_{H, J}\] indexed by regular Floer data $(H, J)$ form a connected simple system of ungraded complexes of $(H_0(A_0), H_0(A_1))$-bimodules.
	
	The complexes \[\{CM(L_0, \Z[\pi_1(L_0)]^{tw} \tensor \Z[\pi_1(L_0)]^{tw}; f, g)\}_{f, g}\] indexed by Morse-Smale pairs $(f,g)$ form a connected simple system of complexes of $(H_0(A_0), H_0(A_1))$-bimodules.
	
	The isomorphism in Lemma \ref{lemma:morse-floer-easy} is an isomorphism of 
    $(H_0(A_0), H_0(A_1))$-bimodules.
\end{lemma}

\subsection{Augmentations and the proof of propositions}
\label{sec:augmentations-and-proof-of-propositions}
The results of the previous three sections need to be slightly generalized to prove Propositions \ref{prop:example}, and \ref{prop:example2}, as we have not yet made use of the \emph{augmentations} mentioned in the statements of those propositions. In this section we prove both propositions. Let $L = L_0$ be a closed exact Lagrangian in Liouville domain, and let \[\epsilon: \Z[\pi_1(L_0)]^{tw} \to End_\Z(M)\] make the $\Z$-module $M$ into a module over the twisted group ring, e.g. let $\epsilon$ be an augmentation to some field $k$, or a ring map to some ring $R$.

Abusing notation, we write 
\[ \Delta: \Z[\pi_1(L_0)] \to (\Z[\pi_1(L_0)]^{tw})^{op} \tensor \Z[\pi_1(L_0)]^{tw}\]
for the restriction of $\Delta: \Z\Pi_1(L_0)^* \to H_0(\P L, \p)$ (see Definition \ref{def:delta-functor}) to the automorphisms of the basepoint $b$. 
\label{sec:morse-theory-analog}

\begin{definition}
	\label{def:k-epsilon}
	Using the notation of the previous two paragraphs, we define $M_\epsilon$ be the $\Z$-local system on $L_0$ induced by the composition $\epsilon \circ \Delta$. 
\end{definition} 
\subsubsection{Modification of complexes}
The complexes $CF(L_0, L_0; \Z[\pi_1(L_0)]^{tw} \tensor \Z[\pi_1(L_0)]^{tw}), CM(L_0, \Z[\pi_1(L_0)]^{tw} \tensor \Z[\pi_1(L_0)]^{tw})$ are naturally right modules over $A := Hom_{H_0(\P L, \p)}(y_b, y_b) = \Z[\pi_1(L_0)]^{tw} \tensor \Z[\pi_1(L_0)]^{tw}$, and the action of the differentials $d$ commutes with the module structure, in the sense that 
\[ d (\gamma a) = (d \gamma) a\]
for any $\gamma$ in the appropriate complex and $a \in Hom_{H_0(\P L, \p)}(y_b, y_b)$. We then define 
\[CF(L_0, L_0; \epsilon; H, J) := CF\left(L_0, L_0; \Z[\pi_1(L_0)]^{tw} \tensor \Z[\pi_1(L_0)]^{tw}; H, J\right) \tensor_{A} M,\]
\[ CM(L_0; \epsilon; f, g) :=  CM(L_0, \Z[\pi_1(L_0)]^{tw} \tensor \Z[\pi_1(L_0)]^{tw}; f, g) \tensor_{A} M,  \]
where $M$ is thought of as a left $A$-module via $\epsilon$. An immediate corollary of the previous sections is that
\begin{lemma}
	\label{lemma:connected-simple-system-simplified}
	Extend the operators $d, C, H$ to 
	\begin{equation}
	\label{eq:cut-down-complexes} CF(L_0, L_0; \epsilon; H, J),  CM(L_0; \epsilon; f, g) 
	\end{equation}
	via $d = d \tensor_A 1, C = C \tensor_A 1, H = H \tensor_A 1$. With this structure, the vector spaces in (\ref{eq:cut-down-complexes}) form a connected simple system of ungraded chain complexes. 
\end{lemma}
We now proceed with the computation of $CM(L_0; \epsilon; f, g)$:
\begin{lemma}
	\label{lemma:compute-CM-simplified}
	The complex $CM(L_0, \epsilon; f, g)$ computes the Morse homology of $M_\epsilon$. 
\end{lemma}
\begin{proof}
	First, for every critical point $p \in Crit(f)$, we have a canonical isomorphism between the vector space $Hom_{H_0(\P L, \p)}(p, b) \tensor_A M$ and $(M_\epsilon)_p$. Indeed, since $M_\epsilon$ is defined using $\Delta$, which is part of the functor $\Delta$ defined in Section \label{sec:morse-theory-analog}, one has a canonical isomorphism 
	\begin{equation}
	\label{eq:a-comparison-map}Hom_{H_0(\P L, \p)}(p, b) \tensor_A M = Hom_{\Z\Pi_1(L_0)}(p, b) \tensor_{Hom_{\Z\Pi_1(L_0)}} M = (M_\epsilon)_p
	\end{equation}
	where the right-most equality holds essentially by definition.

    The above isomorphism extends to an isomorphism of graded vector spaces from $CM(L_0, \epsilon; f, g)$ to the usual Morse complex of the local system $M_\epsilon$ \cite[Section~2.7]{kronheimer-mrowka} by taking the sum of the maps in (\ref{eq:a-comparison-map}) over all $p \in Crit(f)$. It is immediate that this isomorphism intertwines the terms of the differentials: one has to simply check that the signs are identical. A clean way to see this is to use the fact that differential on $CM(L_0, \epsilon; f, g)$ is defined using the functor $\Delta$; using this one can remove any mention of Pin structures from the definition of $CM(L_0, \epsilon; f, g)$, after which the sign comparison is tautological.  $\blacksquare$.
\end{proof}
\subsubsection{Proof of Propositions}
\label{sec:example-prop-proofs-subsec}
\begin{proof}[Proof of Prop. \ref{prop:example}]We write $L_0$ for the Lagrangian in the proposition. A transversely intersecting Hamiltonian isotopy of $L_0$ comes from a time-dependent Hamiltonian $H_t$. By Proposition \ref{prop:hamiltonian-is-part-of-regular-floer-datum}, this is the Hamiltonian part of a regular Floer datum $(H, J)$ for the pair of Lagrangians $L = (L_0, L_0)$. The complex $CF(L_0, L_0; \epsilon; H, J)$ is, by construction, a vector space over $k$ of dimension equal to the number of intersection points of $L_0$ with its image under the flow of $H_t$. The combination of Lemma \ref{lemma:connected-simple-system-simplified} and Lemma  \ref{lemma:compute-CM-simplified} then show that the above complex is quasi-isomorphic to $CM(L_0; \epsilon; f, g)$ for some Morse-Smale pair $(f, g)$. Lemma \ref{lemma:compute-CM-simplified} then shows that this latter complex is a complex computing the homology of $k_\epsilon$. Thus we have the inequality 
\[ \dim_k CF(L_0, L_0; \epsilon; H, J) \geq \dim_k  H_*(L, k_\epsilon). \]

\end{proof}
\begin{proof}[Proof of Prop. \ref{prop:example2}]
	When $L = L_0 \times L_1$ with $L_1$ spin, we have that $w_2(L)$ is the pullback of $w_2(L_0)$ by the projection to the first factor, and thus $\Z[\pi_1(L)]^{tw} = \Z[\pi_1(L_0)]^{tw} \tensor \Z[\pi_1(L_1)]$. The claim about $k_{\epsilon \tensor \epsilon_0}$ then follows from the compatibility of the map $\Delta: \Z[\pi_1(L)] \to (\Z[\pi_1(L)]^{tw})^{op} \tensor \Z[\pi_1(L)]^{tw}$ with the above tensor decomposition.
\end{proof}

Now, we verify that the condition proposed by Witten implies Assumption (\ref{assumption:omega}):
\begin{lemma}
	\label{lemma:witten-implies-assumption}
	Let $L$ be a $Spin^c$ manifold admitting a $Spin^c$ connection such that the connection on the complex line bundle $\lambda$ associated to the $Spin^c$ structure is flat. Then $\Omega w_2 (L) = 0$.
\end{lemma}
\begin{proof}
	For a reference on $Spin^c$ structures and the definition of the line bundle $\lambda$, see \cite[Appendix~D]{lawson1989spin}. The complex line bundle $\lambda$ admits a flat connection if and only if the its first chern class $c_1(\lambda)$ is a torsion class. But $c_1(\lambda) = w_2(L) \mod 2$; this follows, for example, from the argument proving that a manifold is $Spin^c$ if and only if its second Steifel-Whitney class is the reduction of an integral class \cite[Theorem~D.2]{lawson1989spin}. So Witten's condition implies that $w_2(L)$ is the reduction of a torsion integral class. However, there is a commutative diagram
		\begin{equation*}
	\begin{tikzcd}
	H^2(L, \Z) \arrow[r, "\bar\Omega"]
	\arrow[d, "/2"] & H^1(\Omega L, \Z) \arrow[d, "/2"] \\
	H^2(L, \Z/2) \arrow[r, "\Omega"] & 
	H^1(\Omega L, \Z/2).
	\end{tikzcd}
	\end{equation*}
	Here the vertical arrows are reduction of coefficients modulo $2$, and the map $\bar\Omega$ is defined like $\Omega$ by evaluating elements of $H^2(L, \Z)$ on $2$-cycles coming from $1$-cycles in $\Omega L$. But $H^1(\Omega L, \Z)$ is torsion-free by the universal coefficient theorem; so $\bar{\Omega}c_1(\lambda) = 0$, and
	so $\Omega w_2(L) = 0$ by the commutativity of the diagram.
\end{proof}

We give an interesting example of a manifold satisfying Witten's condition:
\begin{lemma}
	Let $L$ be an Enriques surface. Then $L$ is $Spin^c$ but not Spin and $w_2(L)$ is the reduction of a torsion integral class.
\end{lemma}
\begin{proof}
	By definition, $L$ is a complex manifold and so is orientable. Moreover, Enriques surfaces satisfy
	\[ H^1(L, \mathcal{O}) = H^2(L, \mathcal{O}) = 0\]
	where $\mathcal{O}$ denotes the holomorphic structure sheaf of $L$. The long exact sequence of sheaf cohomology applied to the exponential exact sequence 
	\[ 0 \to \Z \to \mathcal{O} \to \mathcal{O^*}\]
	shows that the Picard group of $L$ agrees with $H^2(L, \Z)$. One of the characteristic properties of an enriques surface is that the canonical bundle $K_L$ is a non-trivial square root of the trivial line bundle; since the Picard group agrees with cohomology, this means $c_2(K_L) \neq 0$ but is $2$-torsion. One has 
	\[ w_2(TL) = c_1(TL) \mod 2 = c_1(K_L) \mod 2\]
	so $w_2(TL)$ is the reduction of a torsion integral class, and in particular $L$ is $Spin^c$ by  \cite[Corollary~D.4]{lawson1989spin}. It is classical that $L$ is not spin; for example, this can be seen by showing that the signature of the intersection form on $L$ is $-8$, which is not divisible by $16$, contradicting Rokhlin's theorem on the signatures of Spin manifolds.
\end{proof}

Finally, we prove the small lemma about $\RP^2$ used in the introduction:
 \begin{lemma}
	\label{lemma:RP2-good}
	Let $p \neq 2$ be a prime number.
	$\RP^2$ is a manifold such that $w_2(\RP^2 ) \neq 0$, but $\Omega w_2(\RP^2) = 0$. Moreover, there exists an augmentation $\epsilon: \Z[\pi_1(L)]^{tw} \to k$, where $k$ is a field with $\text{char } k = p$, with the local system $k_\epsilon$ associated to $\epsilon$ by Proposition \ref{prop:example} having the property that $\dim_k H_*(L, k_\epsilon) > 1$.
\end{lemma}
\begin{proof}
	Writing $H^*(\RP^2, \Z/2) = (\Z/2)[a]/a^3$, the $k$-th Steifel-Whitney class of $\RP^2$ is the degree $k$-coefficient of $(1+a)^{3}$, see Milnor-Stasheff \cite{milnor1974characteristic}. 
	As in the proof of Lemma \ref{lemma:witten-implies-assumption}, there is a commutative diagram
	\begin{equation*}
	\begin{tikzcd}
	H^2(\RP^2, \Z) \arrow[r, "\bar\Omega"]
	\arrow[d, "/2"] & H^1(\Omega \RP^2, \Z) \arrow[d, "/2"] \\
	H^2(\RP^2, \Z/2) \arrow[r, "\Omega"] & 
	H^1(\Omega \RP^2, \Z/2).
	\end{tikzcd}
	\end{equation*}
	The universal coefficient theorem shows that $H^1(\Omega \RP^2, \Z)$ is torsion-free; since $H^2(\RP^2, \Z) = \Z/2$ this implies that $\bar\Omega$ is the zero map. But $w_2(\RP^2)$, which lives in the bottom-left corner of the commutative diagram, lifts to $H^2(\RP^2, \Z)$, and thus $\Omega w_2(\RP^2)$ must be zero.

	One has manifest isomorphisms $\Z[\pi_1(\R P^2)] \simeq \Z[\Z/2] = \Z[x]/(x^2-1)$; using the cocycle description of the twisted fundamental group given in Section \ref{sec:twisted-fundamental-group}, one computes that the nontriviality of $w_2(\RP^2)$ introduces a sign, and so  $\Z[\pi_1(\R P^2)]^{tw} \simeq \Z[i] := \Z[x]/(x^2+1)$. There are exactly two nonzero homomorphisms $\epsilon$  from this ring to any field $k$ of characteristic not equal to two, and two nonisomorphic $k$-local systems $k_\epsilon$ over $\RP^2$, which can be checked to correspond via the map $\epsilon \mapsto k_\epsilon$ (Definition \ref{def:k-epsilon}): one is the tensor product of the other with the orientation local system of $\RP^2$. The homology of both local systems is rank $2$ over $k$: for the trivial local system this follows from the fiber sequence $S^2 \to \RP^2 \to B\Z/2$, and for the nontrival local system it follows by Poincare duality. 
\end{proof}

	\section{The Floer complex with Loop Space Coefficients}
\label{sec:floer-theory-start}
In this section we set up the Floer complex $CF_*(\Omega L_0, \Omega L_1; H, J)$, a derived analog of the complex $CF(L_0, L_1; \Z[\pi_1(L_0)]^{tw} \tensor \Z[\pi_1(L_1)]^{tw}; H, J)$ defined in Section \ref{sec:floer-complex-simple}. We begin in Section \ref{sec:loopy-floer-assumptions} by stating the assumptions and choices needed to define the Floer complex. We then describe, in Sections \ref{sec:compatible-fundamental-cycles} and \ref{sec:natural_home}, how to relate the fundamental cycles of \emph{higher-dimensional moduli spaces} of Floer trajectories to the category $C_*(\P L, \p)$. Finally, in Section \ref{sec:floer_complex} we define the complex. We conclude by describing the algebraic properties of the complex in Section \ref{sec:bimodule-structure}, and prove the invariance properties of the complex in Section \ref{sec:invariance-of-the-floer-complex}.

\subsection{Assumptions and Choices} 
\label{sec:loopy-floer-assumptions}
As before in Section \ref{sec:simplified-setup}, we have $(M, \omega, \theta)$ a Liouville domain, containing a pair $L_0, L_1$ of closed exact Lagrangian submanifolds.  We now make an \emph{assumption}, which will hold throughout the rest of the paper:
\begin{equation}
\label{assumption:orientability}
\begin{minipage}{0.9 \textwidth}
Assume that $L_0, L_1$ are \emph{oriented}.
\end{minipage}
\end{equation}

We emphasize that we do \emph{not} make Assumption \ref{assumption:trivial-local-system} in this section. 

The Assumption (\ref{assumption:orientability}) is not essential, but it allows us to avoid developing some straightforward but non-standard homological algebra for discussing the resulting algebraic structures. We expect appropriate of the results in the next three sections to hold without Assumption (\ref{assumption:orientability}).   Unfortunately, Assumption (\ref{assumption:orientability}) prevents us from recovering the results of Section \ref{sec:simplified-construction} from the  more general construction presented in this section.

We now make several \emph{choices}: namely, we choose regular Floer data as in (\ref{eq:chose-regular-floer-data}) and Pin structures as in (\ref{eq:pin-structures-choice-floer}). We do not need to make the choice in equation \ref{eq:trivialize-shift-line} due to the standing assumption in (\ref{assumption:orientability}).

\subsection{Choosing compatible collections of fundamental cycles}
\label{sec:compatible-fundamental-cycles}

In this section we will describe a technique for choosing compatible collections of fundamental cycles for Floer theoretic moduli spaces. As in the rest of the paper, we use \emph{singular homology} as our homology theory, and denote the singular chains on a space $X$ by $C_*(X)$. Given local systems $\mathcal{F}_X, \mathcal{F}_Y$ on spaces $X, Y$, we let 

\begin{equation}
EZ: C_*(X; \mathcal{F}_X) \tensor C_*(Y; \mathcal{F}_Y) \to C_*(X \times Y; \mathcal{F}_X \tensor \mathcal{F}_Y)
\end{equation}

denote the associated Eilenberg-Zilber map. 

We review the definition and basic properties of \emph{topological manifolds with corners} in Appendix \ref{sec:topological-preliminaries}. 
If $X$ is a manifold, or a topological manifold with corners, then we write $\mathfrak{o}_X$ for the orientation local system of $X$. If $X$ and $Y$ are topological manifolds with corners, we have a \emph{product isomorphism} of local systems $\mathfrak{o}_X \tensor \mathfrak{o}_Y \simeq \mathfrak{o}_{X \times Y}$ on $X \times Y$. If $Y \subset \partial X$ is a codimension zero submanifold of the boundary of a topological manifolds with boundary $X$, then there is a canonical 
\emph{boundary isomorphism} $\mathfrak{o}_X \simeq \mathfrak{o}_Y$ of local systems on $Y$ described in Eq. \ref{eq:boundary-map-orientation-lines} in Appendix \ref{sec:topological-preliminaries}. 

Let $\gamma_-, \gamma_+ \in \mathcal{C}(L_0, L_1; H)$ be a pair of Hamiltonian chords from $L_0$ to $L_1$. The moduli space of broken Floer trajectories $\overline{\M}^F(\gamma_-, \gamma_+)$, reviewed in Appendix \ref{sec:floer-moduli-spaces} and Appendix \ref{sec:regularity}, 
is a topological manifold with corners that has a recursive decomposition of its boundary strata into products of other Floer theoretic moduli spaces  (see Eq. \ref{eq:floer-bdry-1}). One would like to choose collections of fundamental cycles for the $\overline{\M}^F(\gamma_-, \gamma_+)$ as the $\gamma_\pm$ vary, so that the fundamental cycles satisfy relations corresponding to these recursive decompositions.  The key property of the moduli spaces is that 
\[\partial^2\overline{\M}^F(\gamma_-, \gamma_+) = \bigcup_{\gamma_0, \gamma_1 \in \mathcal{C}(L_0,L_1)} \overline{\M}^F(\gamma_-, \gamma_0) \times \overline{\M}^F(\gamma_0, \gamma_1) \times \overline{\M}^F(\gamma_1, \gamma_+)\]
which is also equal to 
\[ \left(\partial^1 \overline{\M}^F(\gamma_-, \gamma_1)\right) \times \overline{\M}^F(\gamma_1, \gamma_+) = \overline{\M}^F(\gamma_-, \gamma_0) \times \left(\partial^1 \overline{\M}^F(\gamma_0, \gamma_+)\right)\]
Writing $\mathfrak{o}_{\gamma, \gamma'}$ for $\mathfrak{o}_{\overline{\M}^F(\gamma, \gamma')},$  it is easy to check that the diagram 
\begin{equation}
\label{eq:boundary-map-commutativity}
 \begin{tikzcd}
\mathfrak{o}_{\gamma_-, \gamma_0} \tensor \mathfrak{o}_{\gamma_0, \gamma_1} \tensor \mathfrak{o}_{\gamma_1, \gamma_+}\arrow{r} \arrow{d} 
& \mathfrak{o}_{\gamma_-, \gamma_0}  \tensor \mathfrak{o}_{\gamma_0, \gamma_+} \arrow{d}\\
\mathfrak{o}_{\gamma_-, \gamma_1}  \tensor \mathfrak{o}_{\gamma_1, \gamma_+} \arrow{r} 
& \mathfrak{o}_{\gamma_, \gamma_+}
\end{tikzcd}
\end{equation}
where each map is an application the composition of a product isomorphism and a boundary isomorphsm of orientation lines, commutes \emph{up to a sign of $(-1)^{\dim \overline{\M}^F(\gamma_-, \gamma_0) + 1}$}. 

Let $\overline{EZ}$ denote the composition 
\begin{equation}
\label{eq:ez-1}
\overline{EZ}: C_*(\overline{\M}^F(\gamma_-, \gamma_0); \mathfrak{o}_{\gamma_-, \gamma_0}) \tensor C_*(\overline{\M}^F(\gamma_0, \gamma_+) ;\mathfrak{o}_{\gamma_0, \gamma_+}) \to C_*(\overline{\M}^F(\gamma_-, \gamma_+); \mathfrak{o}_{\gamma_-, \gamma_+})
\end{equation}

of the Eilenberg-Zilber map on orientation lines composed with product and boundary isomorphisms of orientation lines.  Assume that we have chosen fundamental cycles for every factor in the decomposition of  $\partial^2 \overline{\M}^F(\gamma_-, \gamma_+)$, as well as fundamental cycles for each factor in the decomposition of $\partial^1\overline{\M}^F(\gamma_-, \gamma_+)$ which bound the ($\overline{EZ}$-images of the)  fundamental cycles  of the factors in $\partial^2 \overline{\M}^F(\gamma_-, \gamma_+)$. Then the sign in the diagram in Eq. (\ref{eq:boundary-map-commutativity}) cancels with the sign in coming from the Leibniz rule, making the $\overline{EZ}$-image of the product of the fundamental cycles of the terms in $\partial^1\overline{\M}^F(\gamma_-, \gamma_+)$  a cycle supported on the boundary of $\overline{\M}^F(\gamma_-, \gamma_+)$ and representing the fundamental class of its boundary. So we can choose a fundamental cycle for $\overline{\M}^F(\gamma_-, \gamma^+)$ that bounds this boundary-supported cycle. 

Thus, by first choosing fundamental cycles for those moduli spaces which lack boundary and then repeatedly applying the argument in the previous paragraph to construct fundamental cycles for those moduli spaces for which fundamental cycles for its boundary components have already been chosen, one proves the

\begin{lemma}There exists a simultaneous choice of fundamental cycles $[\overline{\M}^F(\gamma_-, \gamma^+)]$ for $\overline{\M}^F(\gamma_-, \gamma^+)$  for all $\gamma_\pm \in \mathcal{C}(L_0, L_1)$ such that the fundamental cycles satisfy the relation
\begin{equation}\partial [\overline{\M}^F(\gamma_-, \gamma_+)] = \sum_{\gamma_0 \in \mathcal{C}(L_0, L_1)} \overline{EZ}([\overline{\M}^F(\gamma_-, \gamma_0)] \tensor  [\overline{\M}^F(\gamma_0, \gamma_+)]).
\end{equation}
$\qed$
\end{lemma}

\subsection{The natural home of the fundamental classes}
\label{sec:natural_home}
Once we have made all the choices and assumptions described in Section \ref{sec:loopy-floer-assumptions}, every Hamiltonian chord $y \in \mathcal{C}(L_0, L_1)$ is $\Z/2\Z$-graded by the standard grading theory for Lagrangian Floer homology, as reviewed in Appendix \ref{sec:gradings-review}. We write $|y| \in \Z/2\Z$ for the grading of $y$. Moreover, we can write
\begin{equation}
\label{eq:indeterminate_orientation_line}
o(y) = o^n(y)
\end{equation} 
with $o^n(y)$ the orientation line of $y$ (see Appendix \ref{sec:def-orientation-lines}) and $n \in \Z$ with parity equal to the grading of $y$; the orientation lines for all these choices of $n$ are canonically isomorphic under (\ref{eq:tensor-square}). For any $y' \in \mathcal{C}(L_0, L_1)$, the space $\overline{\M}^F(y', y)$ is a disjoint union of topological manifolds with boundary, and we can write $\overline{\M}^F_n(y', y)$ for the union of those components which are $n$-dimensional. The moduli space $\overline{\M}^F_n(y', y)$ is potentially nonempty for those $n$ which have the opposite parity as the parity of $|y'| - |y|$. 

We can use the choices made in Section \ref{sec:loopy-floer-assumptions} to define the category $C_*(\P L, \p)$ for $L = (L_0, L_1)$. By parametrizing the boundaries of Floer trajectories $u$ using the metrics on $(u|_{\R \times \{i\}})^*TL_i$ induced by the choice of Floer data, one gets canonical evaluation maps 
\[\M^F_n(y', y) \to \P L(y', y)\]
which extend to the Gromov-Floer bordification 
\[ \phi: \overline{\M}^F_n(y', y) \to \P L(y', y).\]
Abusing notation, we let $\mathfrak{p}$ denote the local system on $\overline{\M}^F_n(y', y)$ that is the pullback of the local system of Pin structures $\mathfrak{p}$ (see Eq. \ref{eq:def-local-system-of-pin-structures} in Section \ref{sec:natural-local-system}) by $\phi$. 

Then, as in Section \ref{sec:floer-complex-simple}, the linear gluing theory for Cauchy-Riemann operators gives a topological description of the orientation lines of Floer mduli spaces (\ref{eq:outcome_of_linear_gluing_theory}), allowing us to view the fundamental class for $\overline{\M}^F_n(y,y')$ with coefficients in the orientation line as an element of 
\[\Z^\check_{\partial/\partial s} \tensor o(y_-) \tensor o(y_+)^\check \tensor H_n(\overline{\M}^F_n(y', y), \p) \]
 
By the definition of $\mathfrak{\phi}$, we have a chain map
\[ \phi_*: C_*(\overline{\M}^F_n(y', y), \p) \to C_*(\P L, \p)(y', y).\]
Thus we can view a fundamental chain of $\overline{\M}^F(y', y)$ as giving an element 
\begin{equation}
\label{eq:fundamental-class-for-MF}
\phi_*[\overline{\M}^F(y', y)]\in \Z^{\check}_{\partial/\partial s} \tensor o(y_-) \tensor o(y_+)^\check \tensor C_*(\P L, \p)(y', y).
\end{equation}

\subsection{The complex}
\label{sec:floer_complex}

Define
\begin{equation}
\label{eq:floer_complex}
CF_*(\Omega L_0, \Omega L_1; H, J) := \bigoplus_{y \in \mathcal{C}(L_0, L_1)}  \Z_y^\check \tensor o(y)\tensor C_*(\P L, \p)(y, y_b).
\end{equation}  
where $\Z_y^\check$ is just the dual to a trivial line with index $y$, with $o(y)$ defined as as in (\ref{eq:indeterminate_orientation_line}).
The Floer complex will be a $\Z/2\Z$-graded complex with the grading defined by
\begin{definition}
	For $\gamma \in \Z_y^\check \tensor o(y)\tensor C_*(\P L, \p)(y, y_b)$, let $|\gamma|$ denote the dimension of the singular chain underlying $\gamma$, and let $\ind y$ denote the index of $y$, which is only defined mod $2$.  We define \begin{equation}\deg \gamma = |\gamma|  - |y|.\end{equation} 
\end{definition}
\begin{remark}
	While the distinction between homological and cohomological gradings is meaningless for a $\Z/2\Z$-graded complex, in certain cases one can use gradings in Floer homology to lift this $\Z/2\Z$-grading to a $\Z/2f\Z$ grading for some $f > 1$, and then using the definition of grading given above will make the differential decrease degree.
\end{remark}

Choose compatible choices of fundamental chains for the moduli spaces $\overline{\M}^F(y', y)$ as in Section \ref{sec:compatible-fundamental-cycles}.

The differential $d$ on the Floer complex $CF_*(\Omega L_0, \Omega L_1; H, J)$ is defined as in the simpler setting of Section \ref{sec:floer-complex-simple}, but using \emph{all} moduli spaces of Floer trajectories instead of just the zero-dimensional hones, and incorporating the natural singular boundary map on the complex. Explicitly, we define
\begin{equation} 
\label{eq:floer-complex-differential}
\Z_y^\check \tensor o(y)\tensor C_*(\P L, \p)(y, y_b) \ni \gamma \mapsto d\gamma = (-1)^{|y|} \partial \gamma + (-1)^{|y|}\sum_{y' \in \mathcal{C}(L_0, L_1)}\sum_{k} \phi_*[\overline{\M}_k^F(y', y)] *_1 \gamma 
\end{equation}
where $\partial \gamma \in \Z_y^\check \tensor o(y)\tensor C_*(\P L, \p)(y, y_b)$ is just the application of the boundary operator to the third factor of the tensor product, and the binary operation $*_1$ (see Figure  \ref{fig:illustration-of-differential-a} for a graphical illustration) outputs elements
\begin{equation}
\label{eq:first_operation}
\phi_*[\overline{\M}_k^F(y', y)] *_1 \gamma \in \Z_{y'}^\check \tensor o(y)\tensor C_*(\P L, \p)(y', y_b)
\end{equation}
and is defined, as in (\ref{eq:def-star1-simple}), by taking the element $\phi_*[\overline{\M}_k^F(y', y)] \in \Z_{\partial/\partial s}^\check \tensor o(y') \tensor o(y)^\check \tensor C_*(\P L, \p)(y', y)$, composing in $C_*(\P L, \p)$ with $\gamma$, and rearranging the various orientation lines via the isomorphism 
\[ \Z_{\partial/\partial s}^\check \tensor o(y') \tensor o(y)^\check \tensor \Z_y^\check \tensor o(y) \simeq 
\Z_y^\check \tensor \Z_{\partial/\partial s}^\check \tensor o(y') \tensor o(y)^\check \tensor o(y) \simeq
\Z_{y'}^\check \tensor o(y') \]
where the last isomorphism uses the isomorphism $o(y)^\check \tensor o(y) \simeq \Z$ and the grading-shifting isomorphism $\Z_y^\check \tensor \Z_{\partial/\partial s}^\check \simeq \Z_{y'}^\check$ of $1$-dimensional vector spaces.

\begin{proposition}
	The operator $d$ is a differential, i.e. 
\[ d^2 = 0. \]
\end{proposition}
\begin{proof}
Since $\partial^2\gamma = 0$ it suffices to check that 
\begin{equation}
\label{eq:d2-zero-proof}
\begin{split} (-1)^{|y|+|y|+k}\partial (\phi_*[\overline{\M}_k^F(y'', y)] *_1 \gamma) = &-(-1)^{|y|+|y|+1}\phi_*[\overline{\M}_k^F(y'', y)] *_1 \partial \gamma \\
&- \sum_{i+j+1 = k} (-1)^{|y|+|y|+j}\phi_*[\overline{\M}_i^F(y'', y')] *_1 (\phi_*[\overline{\M}_j^F(y', y)] \cdot \gamma). 
\end{split}
\end{equation}
Applying the Leibniz rule to the left hand side of (\ref{eq:d2-zero-proof}) gives two terms; the second term, which is of the form $(\phi_*[\overline{\M}_k^F(y'', y)] \cdot \partial \gamma)$ picks up a $(-1)^k$ from the Koszul sign coming from commuting the differential with $\phi_*[\overline{\M}_k^F(y'', y)]$, giving exactly the first term on the right hand side of (\ref{eq:d2-zero-proof}). The first term after applying the Leibniz rule to the left hand side of (\ref{eq:d2-zero-proof}) is 
\begin{equation} (-1)^{|y|+|y|+k}(\partial \phi_*[\overline{\M}_k^F(y'', y)] *_1 \gamma) = (-1)^{|y|+|y|+k} \sum_{i+j+1=k}((\phi_*[\overline{\M}_i^F(y'', y')] *_1 \phi_*[\overline{\M}_j^F(y', y)])*_1 \gamma). 
\end{equation}
The trivialization (up to a choice of Pin structure) of the orientation line of $(\phi_*[\overline{\M}_i^F(y'', y')] *_1 \phi_*[\overline{\M}_j^F(y', y)])$, 
\[ \Z_{\partial/\partial s_2}^\check \tensor \Z_{\partial/\partial s_1}^\check \tensor o(y'') \tensor o(y')^\check \tensor o(y') \tensor o(y) \]
(where $\partial/\partial s_1$ is the vector corresponding to translation in the $s$ direction in the $\M^F_i$ moduli space, while $\partial/\partial s_2$ is the vector corresponding to translation in the $s$ direction in the $\M^F_j$ moduli space), differs from the product orientation of 
\[\Z_{\partial/\partial s_1}^\check \tensor o(y'') \tensor o(y')^\check  \tensor \Z_{\partial/\partial s_2}^\check  \tensor o(y') \tensor o(y)  \]
by $(-1)^{i}$; thus 
\begin{equation}
\begin{split}
(-1)^{|y|+|y|+k} ((\phi_*[\overline{\M}_i^F(y'', y')] *_1 \phi_*[\overline{\M}_j^F(y', y)])*_1 \gamma) =\\ (-1)^{|y|+|y|+k+i}(\phi_*[\overline{\M}_i^F(y'', y')] *_1 (\phi_*[\overline{\M}_j^F(y', y)]*_1 \gamma))
\end{split}
\end{equation}
which is what appears on the right hand side of (\ref{eq:d2-zero-proof}), thus confirming the proposition.
\end{proof}

\subsection{The Floer complex as a bimodule}
\label{sec:bimodule-structure}
The Floer complex that we have defined is substantially ``larger'' than the usual complex defining Lagrangian Floer homology for exact Lagrangian submanifolds; however, this is compensated by the existence of extra algebraic structures that act on this Floer complex.

First, note that $CF_*(\Omega L_0, \Omega L_1; H, J)$ is a right module over the algebra $C_*(\P L, \p)(y_b, y_b)$; an element $\gamma_b \in C_*(\P L, \p)(y_b, y_b)$ acts on the Floer complex by the tensor product of $1 \in End(\Z^\check_y \tensor o(y))$ and of right-composition with $\gamma_b$ in $C_*(\P L, \p)$. Lemma \ref{lemma:left-compostion-is-module-map} states that left-composition with a morphism in $C_*(\P L, \p)$ commutes with this action. Moreover the pieces $\Z_y^\check \tensor o(y)\tensor C_*(\P L, \p)(y, y_b)$ of the Floer complex, with the differential given by $\partial$ as in (\ref{eq:floer-complex-differential}), are \emph{dg}-modules over $C_*(\P L, \p)(y_b, y_b)$. Since the differential on the Floer complex is given by a signed sum of the map $\partial$ and left-compositions by morphisms in $C_*(\P L, \p)$, the previous two statements combine to prove

\begin{lemma}
	\label{lemma:loopy-floer-is-a-bimodule}
	The structure of $CF_*(\Omega L_0, \Omega L_1; H, J)$ as a right module over $C_*(\P L, \p)(y_b, y_b)$ described in the above paragraph makes $CF_*(\Omega L_0, \Omega L_1; H, J)$ into a \emph{right $dg$-module} over $C_*(\P L, \p)(y_b, y_b)$. In particular, via composition with the algebra morphism in (\ref{eq:tensor-product-to-hom-map}), $CF_*(\Omega L_0, \Omega L_1; H, J)$ is naturally a $(A_0, A_1)$-bimodule. $\qed$
\end{lemma}

In fact, something stronger is true. In Appendix \ref{sec:dg-algebra-conventions} we recall the notion of a \emph{iterated extension of free bimodules}. 
\begin{lemma}
	The complex $CF_*(\Omega L_0, \Omega L_1; H, J)$ with the bimodule structure of Lemma \ref{lemma:loopy-floer-is-a-bimodule} is an iterated extension of free bimodules.
\end{lemma}
\begin{proof}
	The $\R$-filtration on $CF_*(\Omega L_0, \Omega L_1; H, J)$ needed to define the structure of an iterated extension is the \emph{action filtration}: we say that 
	\[ CF_*(\Omega L_0, \Omega L_1; H, J)^{\leq \ell} := \bigoplus_{\substack{y \in \mathcal{C}(L_0, L_1)\\\mathcal{A}(y) \leq \ell}}  \Z_y^\check \tensor o(y)\tensor C_*(\P L, \p)(y, y_b), \]
	where $\mathcal{A}$ is the \emph{symplectic action}, the function on paths in $M$ which has Floer's equation as a gradient flow. In our conventions we can take 
	\begin{equation} \mathcal{A}(y) = \int y^*\theta - H(t, y(t))dt - h_{L_1}(y(1)) + h_{L_0}(y_0)
	\end{equation}
	where $\theta$ is the primitive of the symplectic form on our Liouville domain, and $h_{L_k}$ is a primitive of the restriction of $\theta$ to the exact Lagrangian $L_k$.
	
	The Floer trajectories contributing to the differential decrease this quantity and so this is a subcomplex, and indeed a sub-bimodule.
	The subquotients are just the bimodules 
	\[\Z_y^\check \tensor o(y)\tensor C_*(\P L, \p)(y, y_b)\]
	which are quasi-isomorphic to free bimodules by Lemma \ref{lemma:left-compostion-is-module-map}.
\end{proof}

\subsection{Invariance of the Floer complex}\label{sec:invariance-of-the-floer-complex}
In this section we construct continuation maps for the Floer complex defined in Equation \ref{eq:floer_complex}.
\subsubsection{The continuation map}
\label{sec:floer_continuation_map}
Suppose that a pair of regular Floer data $(H_0, J_0)$, $(H_1, J_1)$ has been specified for which the auxiliary choices in Section \ref{sec:loopy-floer-assumptions} have been made.  Equip $Z$ with boundary Floer data given by $(H_1, J_1)$ at the positive end and $(H_0, J_0)$ at the negative end, and choose a regular perturbation datum on $Z$ compatible with these boundary Floer data.

In Appendix \ref{sec:floer-moduli-spaces}, we recall that for every pair  $y' \in  \mathcal{C}(L_0, L_1; H_0),  y \in \mathcal{C}(L_0, L_1; H_1)$ and regular perturbation datum $(K, J)$ on $Z$ compatible with the boundary Floer data, there is a moduli space $\M^C(y', y)$ of ``continuation'' solutions to the inhomogeneous pseudoholomorphic map equation with Gromov compactification $\overline{\M}^C(y', y)$. This moduli space admits a map
\[\phi: \overline{\M}^C(y', y) \to \P L(y', y)\]
that parametrizes boundaries of broken solutions to the inhomogeneous pseudoholomorphic map equation by their lengths in the $t$-dependent metrics on $M$ induced by the perturbation data.

Exactly as in Section \ref{sec:natural_home}, we write $\p$ for the pullback of the local system of Pin structures by $\phi$; linear gluing theory then gives an invariant isomorphism of local systems
\[ \mathfrak{o}_{\overline{\M}^C(y', y)} \simeq o(y') \tensor o(y)^\check \tensor \mathfrak{p}.\ \]
Therefore, a fundamental chain for $\overline{\M}_k^C(y', y)$, the union of the $k$-dimensional components of $\overline{\M}^C(y', y)$, gives an element
\begin{equation}
\label{eq:fundamental-class-MC}
\phi_*[\overline{\M}^C(y', y)] \in o(y_-) \tensor o(y_+)^\check \tensor C_*(\P L, \p)(y', y).
\end{equation}

Using the argument in Section \ref{sec:compatible-fundamental-cycles}, choose fundamental chains for the moduli spaces $\overline{\M}^C(y', y)$ extending the choices of fundamental chains for the Floer moduli spaces of the boundary Floer data.

 Define a continuation map
\[  C: CF^*(\Omega L_0, \Omega L_1; H_1, J_1) \to CF^*(\Omega L_0, \Omega L_1; H_0, J_0) \]
via, for $y \in \mathcal{C}(L_0, L_1; H_1, J_1)$, 
\[ \Z_y^\check \tensor o(y)\tensor C_*(\P L, \p)(y, y_b) \ni \gamma \mapsto\\  C \gamma \in   \Z_{y'}^\check \tensor o(y')\tensor C_*(\P L, \p)(y, y_b) \]
and
\[ C \gamma = \sum_{k \geq 0} C_k \gamma,  \]
\begin{equation}
\label{eq:second-operation}C_k \gamma = \sum_{y' \in \mathcal{C}(L_0, L_1; H_0, J_0)} (-1)^{k} \phi_*[\overline{\M_k}^C(y', y)] *_2 \gamma. 
\end{equation}
Here, the product $*_2$ is defined as in (\ref{eq:def-star2-simple}), by composition
 in $C_*(\P L, \p)$, evaluation $o(y)^\check \tensor o(y) \to \Z$, and simply identifying $\Z_y^\check$ with $\Z_{y'}^\check$ as trivial lines after commuting $\Z_y^\check$ through to the left. 
 
\begin{proposition}
\label{prop:floer-continuation-map}
The continuation map is a chain map, i.e. 
\begin{equation}
\label{eq:continuation-is-chain-map}
dC\gamma = Cd\gamma.
\end{equation} 
\end{proposition}
\begin{proof}
We analyze the left hand side of (\ref{eq:continuation-is-chain-map}) by applying the Leibniz rule to the term
\begin{equation}
\label{eq:continuation-is-chain-map-expand}
(-1)^{|y|}\partial(\phi_*[\overline{\M_k}^C(y', y)]*_2 \gamma).
\end{equation}
When commuting the $\partial$ over to the $\gamma$, we pick up a sign of $(-1)^k$, thus getting the expression
$(-1)^{|y|+k} \phi_*[\overline{\M_k}^C(y', y)] *_2\gamma$ which also appears on the right hand side of (\ref{eq:continuation-is-chain-map}). Analyzing the first term the Leibniz expansion of (\ref{eq:continuation-is-chain-map-expand}) we get two kinds of terms corresponding to the two kinds of boundary strata of $\overline{\M}^C$. Half of the terms take the form
\begin{equation}
\label{eq:continuation_1} (-1)^{|y|} (\phi_*[\overline{\M_i^C}] *_2 \phi_*[\overline{\M_j^F}]) *_1 \gamma 
\end{equation}
where the induced orientation (up to a choice of Pin structure) of the orientation line associated to $(\phi_*[\overline{\M_i^C}] *_2 \phi_*[\overline{\M_j^F}])$ 
\[\Z_{\partial/\partial s}^\check \tensor o(y'') \tensor o(y')^\check \tensor o(y') \tensor o(y)^\check, \]
differs from the product orientation of 
\[  o(y'') \tensor o(y')^\check  \tensor \Z_{\partial/\partial s}^\check \tensor o(y') \tensor o(y)^\check \]
by $(-1)^i$ with $i = \deg o(y'')+ \deg o(y')^\check$. Therefore,
\begin{equation}(-1)^{|y|}(\phi_*[\overline{\M_i^C}] *_2 \phi_*[\overline{\M_j^F}]) *_1 \gamma   = (-1)^{|y|+i}(\phi_*[\overline{\M_i^C}] *_2\cdot (\phi_*[\overline{\M_j^F}] *_1 \gamma)),
\end{equation}
and where the right hand side is exactly a term appearing in the right hand side of (\ref{eq:continuation-is-chain-map}). Thus we have found the terms on the right hand side of (\ref{eq:continuation-is-chain-map}) on the left hand side; it suffices to show that the remaining terms on the left hand side cancel.  A calculation keeping track of orientation lines shows that the remaining term in the Leibniz expansion of (\ref{eq:continuation-is-chain-map-expand}) is 
\begin{equation}
\label{eq:continuation_3} - (-1)^{|y|}\phi_*[\overline{\M_i}^F(y'', y')]*_1 (\phi_*[\overline{\M_j}^C(y', y)]*_2 \gamma) 
\end{equation}
where the \emph{minus sign} comes from the fact the vector $- \partial/\partial s$ corresponds to an \emph{outwards-pointing vector} on the moduli $\M^C_k$, where $\partial/\partial s$ is the translation in the $s$ direction in the linearized Fredholm problem defining $\widetilde{\M}_i^F(y'', y')$.  The terms in (\ref{eq:continuation_3}) cancel the terms of $dC\gamma$ coming from contributions of Floer moduli spaces to $d$, proving the proposition.
\end{proof}

\subsubsection{Homotopy of continuation maps}
Given two different choices $(K_0, J_0), (K_1, J_1)$ used to define two different continuation maps 
\[ C_0, C_1:  CF^*(\Omega L_0, \Omega L_1; H_1, J_1) \to CF^*(\Omega L_0, \Omega L_1; H_0, J_0)\]
we define a map 
\[H:  CF^*(\Omega L_0, \Omega L_1; H_1, J_1) \to CF^*(\Omega L_0, \Omega L_1; H_0, J_0)\]
which will be a homotopy between the continuation maps.

Choose a $1$-parameter family of perturbation data $(K_t, J_t)$ on $Z$ as in Proposition \ref{prop:homotopy-moduli-space} to define the moduli spaces 
$\overline{\M}^H(\gamma^+, \gamma^-) := \overline{\M}^H(\gamma^+, \gamma^-; Z, \mathfrak{B}, K_t, J_t)$
as in Proposition \ref{prop:homotopy-moduli-space-compactification}.
Using the argument in Section \ref{sec:compatible-fundamental-cycles}, choose fundamental cycles for the components of $\overline{\M}^H(\gamma^+, \gamma^-)$ that are compatible with the previously made choices of fundamental cycles for the boundary strata of $\overline{\M}^H(\gamma^+, \gamma^-)$, which are described in Proposition \ref{prop:homotopy-moduli-space-compactification}. 

Define the following homotopy between $C_0$, the continuation map associated to $(K_0, J_0)$, and $C_1$, the continuation map associated to $(K_1, J_1)$: 
\[ C_*(\P L, \p)(\gamma^-, \gamma^b) \ni \gamma \mapsto H\gamma = \sum_{\gamma^+ \in \mathcal{C}(H_0, J_0)} (-1)^{|y|}\phi_*[\overline{\M}^H](\gamma^+, \gamma^-) *_2 \gamma^+. \]
 Here, the operation $*_2$ is defined as in (\ref{eq:second-operation}). We make sense of the element 
 \begin{equation}\phi_*[\overline{\M}^H](\gamma^+, \gamma^-) \in \mathfrak{o}(\gamma^+) \tensor \mathfrak{o}(\gamma^-)^\check \tensor C_*(\P L, \p)(\gamma^+, \gamma^-)
 \end{equation} as follows: the standard transversality arguments for parametrized moduli spaces (which prove Prop. \ref{prop:homotopy-moduli-space}) also show that at a point $u \in \overline{\M}^H(\gamma^+, \gamma^-)$ lying over $t \in [0,1]$, the orientation line of the tangent space at $u$ of the ambient moduli space is a tensor product of the orientation line $o(T_t[0,1])$ with the determinant line of the linearized operator for the Floer equation defining 
 \begin{equation}u \in \overline{\M}^C(\gamma^+, \gamma^-; Z, \mathfrak{B}, K_t, J_t);
 \end{equation} we then trivialize the orientation local system $o(T_t[0,1])$ by choosing $\partial/\partial_t$ to be the positive orientation, and use linear gluing theory to conclude that determinant line of this (possibly not surjective) linearized operator is canonically isomorphic to $\mathfrak{o}(\gamma^+) \tensor \mathfrak{o}(\gamma^-)^\check \tensor \p$ as required. 
\begin{proposition}
	\label{prop:floer-homotopy-map}
	The map $H$ defined above is a homotopy between the continuation maps $C_0$ and $C_1$, i.e. 
\begin{equation}
\label{eq:homotopy-is-homotopy}dH\gamma + Hd\gamma = C_0\gamma - C_1\gamma. 
\end{equation}
\end{proposition}

\begin{figure}[h]
		\centering
		\resizebox{\textwidth}{!}{
			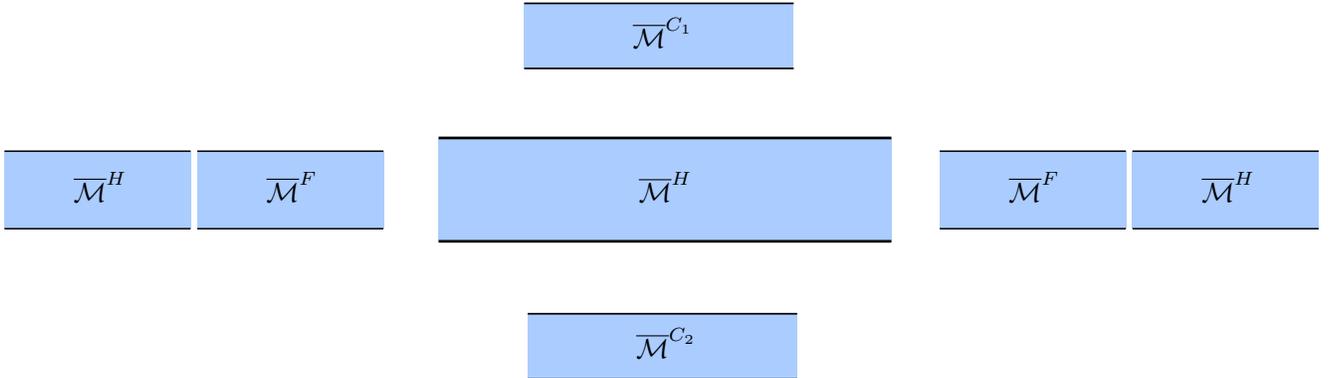}
		\caption{The codimension $1$ components of the boundary of $\overline{\M}^H$. }
		\label{fig:homotopy-method}
\end{figure}
\begin{proof}
The  components of the moduli space $\overline{\M}^H$ have codimension $1$ boundary components that break up into the four types displayed in Figure \ref{fig:homotopy-method}, and the identity is verified by checking signs for each type. Thus, we write $\overline{\M}_k^H$ for the union of the $k$-dimensional components of $\overline{\M}^H$, and consider the Leibniz expansion of 
\begin{equation}
\label{eq:leibniz-expand-homotopy}
(-1)^{|y|+k+1}\partial\left((-1)^{|y|}\phi_*[\overline{\M}_{k+1}^H] *_2 \gamma \right)
\end{equation}
for $k \geq 0$, which appears in the expansion of $dH\gamma$. A quick check with outwards pointing vectors shows that the $\overline{\M}^C$ components of the boundary of $\overline{\M}^H$ (see Fig. \ref{fig:homotopy-method}) contribute exactly $(C_0)_k\gamma - (C_1)_k \gamma$. Thus, we have identified the right hand side of (\ref{eq:homotopy-is-homotopy}) as a subset of the terms in the left hand side, and it remains to demonstrate that the rest of the terms on the left hand side cancel. 

When we commute the $\partial$ in (\ref{eq:leibniz-expand-homotopy}) through to $\gamma$, we pick up a Koszul sign of $(-1)^{k+1}$ giving us a total sign of $(-1)^0$, which is the opposite of the sign $(-1)^{|y|+|y|+1}$ carried by the corresponding term contributing to $Hd\gamma$. So these types of terms cancel.

 Finally, the remaining contributions of the boundary of $\overline{\M}^H$ to (\ref{eq:leibniz-expand-homotopy}), which take the schematic form  $\M_i^H \M_j^F \gamma$, $\M_i^F \M_j^H \gamma$ as in Figure \ref{fig:homotopy-method}. We will will argue that these terms cancel with the contributions of Floer moduli spaces to $d$ in the expansion of $dH\gamma$ and $Hd\gamma$. Indeed, the sign carried by the term of the form $\M_i^H \M_j^F \gamma$ in $Hd\gamma$ is $(-1)^{|y|+j+|y|}(-1)^i$, with the $(-1)^i$ comes from commuting the extra trivial line hidden by the grading-shifting isomorphism through the orientation lines of $\overline{\M}^H_i$. The sign carried by the term of the form $\M_i^F \M_j^H \gamma$ in $dH\gamma$ is $(-1)^{|y|+j+|y|}$ because there is no extra commutation of lines coming from an grading-shifting isomorphism; however, when comparing this term with the corresponding term in the Leibniz expansion of (\ref{eq:leibniz-expand-homotopy}) one picks up an extra sign $(-1)^i$  because during the Leibniz expansion one has to commute the $\partial/\partial_t$ of the base parameter space $[0,1]_t$ of the parametrized moduli space $\overline \M_j^H$ through the $\overline \M_i^F$ to be able to compare orientations. Thus both terms carry a total sign of $(-1)^{i+j}$ which is the opposite of the opposite of the sign $(-1)^{k+1} = (-1)^{i+j+1}$ in (\ref{eq:leibniz-expand-homotopy}), and so these types of terms cancel in the left hand side of (\ref{eq:homotopy-is-homotopy}) as well.
\end{proof}

\begin{proposition}
\label{prop:floer-compose-continuations}
Let $(H_i, J_i)$, $i = 0,1, 2$ be a triplet of regular Floer data, and let $C_0, C_1$ be choices of continuation maps 
\[ C_i: CF^*(\Omega L_0, \Omega L_1; H_i, J_i) \to CF^*(\Omega L_0, \Omega L_1; H_{i+1}, J_{j+1}). \]
Then there is a choice of data defining a continuation map 
\[ C: CF^*(\Omega L_0, \Omega L_1; H_0, J_0) \to CF^*(\Omega L_0, \Omega L_1; H_2, J_2) \]
such that $C_1 C_0$ is homotopic to $C$. 
\end{proposition}
\begin{proof}
There are only a finite number of Hamiltonian chords for each of the choices of Floer data $(H_i, J_i)$; thus, the gluing theorem says that there exists an $\epsilon > 0$ such that for every triplet $y_i \in \mathcal{C}(H_i, J_i)$, $i = 0,1,2,$ the parametrized moduli space
\[ \bigsqcup_{t \in [0, \epsilon]} \overline{\M}^C(y_2, y_1) \#_t \overline{\M}^C(y_1, y_0) \]
is a disjoint union of topological manifolds with corners. It admits a map to $[0, \epsilon]$ by definition, and this map is continuous; the inverse images of $0$ and $\epsilon$ are codimension $1$ strata, where the inverse image of $\epsilon$ is $\overline{\M}^C(y_2, y_0)$ for a perturbation datum defining a continuation map $C$. The remaining codimension $1$ strata are products of similar parametrized moduli spaces with moduli spaces of solutions to Floer's equation with Floer data $(H_2, J_2)$ or $(H_0, J_0)$. This moduli space defines a homotopy between $C$ and $C_1C_0$ exactly in the same way that 
the moduli space $\overline{\M}^H$ defines a homotopy between $C_0$ and $C_1$. 
\end{proof}

\begin{proposition}
\label{prop:floer-continuation-identity}
The continuation map
\[ C: CF^*(\Omega L_0, \Omega L_1; H, J) \to CF^*(\Omega L_0, \Omega L_1; H, J) \]
is homotopic to the identity. 
\end{proposition}
\begin{proof}
We can simply choose $H, J$ as the data defining our continuation map; this is manifestly regular. The zero-dimensional moduli spaces make $C_0 = id$. By filtering the complex by the energy of $y$ we see that $C$ gives the identity page on the $E_1$ page of a convergent spectral sequence, and thus must be homotopic to the identity. 
\end{proof}

\begin{corollary}
The complex $CF^*(\Omega L_0, \Omega L_1; H, J)$ is independent of $(H, J)$ up to a homotopy equivalence which is canonical in the homotopy category of chain complexes.
\end{corollary}

Since the Floer complex is generally infinite rank over $\Z$ even homologically, the above does not seem so useful. However, as in Section \ref{sec:bimodule-structure}, more is true:

\begin{proposition}
	\label{prop:bimodule-structures-invariance}
	The maps $C$ and $H$ are morphisms of $(C_*(\Omega L_0, \p), C_*(\Omega L_1, \p))$-bimodules. Thus, the Floer complex $CF^*(\Omega L_0, \Omega L_1; H, J)$ is independent of $(H, J)$ up to a quasi-isomorphism  of bimodules which is canonical in the homotopy category of $(C_*(\Omega L_0, \p), C_*(\Omega L_1, \p))$-bimodules.
\end{proposition}
\begin{proof}
	We refer to Appendix \ref{sec:dg-algebra-conventions} for the definition of morphisms of bimodules and the associated homotopy category. The statement of the proposition follows from the combination of Propositions \ref{prop:floer-continuation-map}, \ref{prop:floer-homotopy-map}, \ref{prop:floer-compose-continuations}, and \ref{prop:floer-continuation-identity}, with the claim that the maps $C$ and $H$ are bimodule morphisms; but $C$ and $H$ are both defined as a sum of tensor products of morphisms of lines and left-compositions in $C_*(\P L, \p)$, which by Lemma \ref{lemma:left-compostion-is-module-map} implies the claim.
\end{proof}
\section{The Morse complex with Loop Space Coefficients}
\label{sec:morse_complex}
To calculate the complex defined in Section \ref{sec:floer_complex} when $L_0 = L_1$, we define an analogous Morse-theoretic complex. In Section \ref{sec:morse-theory-fundamental-classes} we explain how to relate  Morse-theoretic moduli spaces to the categories $C_*(\P L, \p)$ for $L = (L_0, L_0)$, and in Section \ref{sec:morse-complex-def} we define the Morse complex $CM_*(\Omega L_0; f, g)$ and state its invariance properties. Finally in Section \ref{sec:morse-floer-comparison} we compare the Floer complex to the Morse-theoretic complex defined in this section.

\subsection{Fundamental chains in Morse theory}
\label{sec:morse-theory-fundamental-classes}

In Appendix \ref{sec:morse-theory-technical}, we review definitions of moduli spaces related to Morse theory:
\begin{itemize}
	\item $\overline{\M}^M(p, q; f, g)$ (Prop. \ref{prop:morse-theory-technical1}), the moduli space of broken downwards gradient trajectory for a Morse function $f$; 
	\item $\overline{\M}^{MC}(p, q; f_{t}, g)$ (Prop. \ref{prop:morse-theory-technical2}), a compactification of the moduli space of downards gradient trajectories for a time-dependent interpolation $f_t$ between a pair of Morse functions $f_0, f_1$;  and 
	\item $\overline{\M}^{H}(p,q; f_{s,t}, g)$ (Prop. \ref{prop:morse-theory-technical3} ), a compactification of the parametrized moduli space $\overline{\M}^{MC}(p, q; f_{s, t}, g)$  with parameter $s \in [0,1]$, associated to a homotopy $f_{s,t}$ between a pair of interpolations $f_t, f'_t$ between a pair of Morse functions $f_0, f_1$. 
\end{itemize}
For any one of these moduli spaces, which we will denote generically in the next two paragraphs by $W$, there is a map  $\phi: W  \to \P L(p, q)$ defined by sending a broken gradient trajectory to a pair of paths between  $p$ and $q$ going in opposite directions and parametrized by length with respect to the auxiliary metric $g$ used in the definition of $W$. As in Section \ref{sec:morse-theory-analog},  one can define a map $\phi_*: C_*(W, \Z) \to C_*(\P L, \p)(p, q)$ defined as the composition $\Delta \circ C_*(\phi)$, where $\Delta$ is the functor defined in Definition \ref{def:delta-functor}.

In Appendix \ref{sec:morse-theory-orientations}, we review how to orient the moduli spaces defined in Appendix \ref{sec:morse-theory-technical}; to do so one must chooseorientations of $o(T^-_p; f)$, the orientation lines of the negative eigenspaces of the Hessians of $f$ at critical points $p \in Crit(f)$. Orienting a space trivializes its orientation local system, so the map $\phi$ defined in the previous paragraph allows us to evaluate fundamental classes for the moduli spaces $W$ to elements
\[ \phi_*[\M^M(p, q; f, g)] \in \Z_{\partial/\partial t}^\check \tensor o(T^-_p; f) \tensor \det o(T^-_p; f)^\check \tensor C_*(\P L, \p) (p, q), \]
\[ \phi_*[\M^{MC}(p, q; f_t, g)] \in o(T^-_p; f_0) \tensor\det  o(T^-_q; f_1)^\check \tensor C_*(\P L, \p) (p, q), \]
\[ \phi_*[\M^{H}(p, q; f_{s,t}, g)] \in o(T^-_p; f_0) \tensor\det  o(T^-_q; f_1)^\check \tensor C_*(\P L, \p) (p, q). \]
Henceforth, we will drop the $f_i$ from the notation, as the relevant Morse functions will be clear from the context. 
\subsection{The (Morse) complex}
\label{sec:morse-complex-def}
Let $L_0$ be an oriented manifold. Make choices an in Section \ref{sec:morse-theory-analog} and Equations \ref{eq:morse-smale-choice}, \ref{eq:morse-theory-pin-choice}. Use the choice of Pin structures in (\ref{eq:morse-theory-pin-choice}) to define a category $C_*(\P L, \p)$ (see Section \ref{sec:natural-local-system}) for $L = (L_0, L_0)$. Define an abelian group
\[ CM_*(\Omega L_0; f, g) = \bigoplus_{p \in Crit(f)} \Z_{p}^\check \tensor  o(T^-_p) \tensor C_*(\P L, \p)(p, b). \]

\begin{definition}
	Given  $\gamma \in \Z_{p}^\check \tensor o(T^-_p) \tensor C_*(\P L, \p)(p,b)$ let $|\gamma|$ denote the dimension of the singular chain underlying $\gamma$, and let $|p|$ denote the index of the critical point $p$. 
	
	 \emph{Define} $\deg \gamma = |\gamma|  - |p|$. 
\end{definition}
This definition makes $CM_*(\Omega L_0; f, g)$ into a $\Z$-graded abelian group; this convention will make the differential we now introduce decrease degree. 

Choose compatible fundamental chains for $\overline\M^M(p, q;f, g)$, over $p, q \in Crit f$, as in Section \ref{sec:compatible-fundamental-cycles}. We define the differential (see Figure \ref{fig:illustration-of-differential-b})  via 
\begin{equation}
\label{eq:morse-complex-differential}
\begin{gathered}
\Z_y^\check \tensor o(T^-_q)\tensor C_*(\P L, \p)(q, b) \ni \gamma \mapsto d\gamma  \in CM_*(\Omega L_0; f, g)\\
d \gamma  = (-1)^{|q|} \partial \gamma + (-1)^{|q|}\sum_{p \in Crit(f)}\sum_{k} \phi_*[\overline{\M}_k^M(p,q)] *_1 \gamma 
\end{gathered}
\end{equation}
where $\partial \gamma \in \Z_q^\check \tensor o(T^-_q) \tensor C_*(\P L, \p)(q, b)$ is just the application of the boundary operator to the third factor of the tensor product, and $*_1$ is defined as in  (\ref{eq:first_operation}) by commuting the trivial line $\Z_p^\check$ to the left, applying a grading-shifting isomorphism, and canceling orientation lines.

Suppose that we are given a pair of Morse functions $f_0, f_1$ that are both Morse-Smale with respect to a Riemannian metric $g$, together with choices of Pin structures (\ref{eq:morse-theory-pin-choice}) for both Morse functions. Choose data $f_t$ as needed in Prop. \ref{prop:morse-theory-technical2} to define $\overline{\M}^C$.  Define a continuation map
\[  C: CM_*(\Omega L_0; f_1, g) \to CM_*(\Omega L_0; f_0, g), \]
by fixing
\[ C \gamma := \sum_{k \geq 0} C_k \gamma, \text{ for } \gamma \in \Z_q^\check \tensor o(T^-_q) \tensor C_*(\P L, \p)(q, y_b) \]
where
\[ C_k \gamma := \sum_{p \in Crit(f)} (-1)^{k} \phi_*[\overline{\M_k}^{MC}(p, q)] *_2 \gamma. \]
Here $*_2$ is defined as in Equation \ref{eq:second-operation} by moving the trivial line to the left and canceling orientation lines.

Given a given two choices of data $f_t, f'_t$ used to define continuation maps $C, C'$ as above, choose data $f_{s,t}$ as needed in Prop \ref{prop:morse-theory-technical3} to define $\overline{\M}^H$, and define a homotopy 
\[H:  CM_*(\Omega L_0; f_1, g) \to CM_*(\Omega L_0; f_0, g)\]
between the continuation maps by 
\[ \Z_y^\check \tensor o(T^-_q)\tensor C_*(\P L, \p)(q, b) \ni \gamma \mapsto H\gamma := \sum_{\gamma \in Crit(f')} (-1)^{|q|}\phi_*[\overline{\M}^H(p, q; f_{s,t}, g)] *_2 \gamma. \]

The sign computations needed to verify the following proposition are exactly identical to those in the section on the Floer complex, and we will omit them.
\begin{proposition}
We have that $d^2 = 0$, and that $C$ is a chain map and $H$ is a homotopy.
\end{proposition}

\begin{remark}
One can show that the composition of two continuation maps is a continuation map, by an analogous parametrized moduli space argument as in the section on the Floer complex, and that the continuation maps 
\[ CM_*(\Omega L_0; f, g) \to CM_*(\Omega L_0; f, g) \]
are homotopic to the identity. Of course, this also follows from the Morse-Floer comparison in Section \ref{sec:morse-floer-comparison} and the results of Section \ref{sec:floer-theory-start}. 
\end{remark}
\begin{remark}
	Suitable modifications of the arguments and results of of Sections \ref{sec:bimodule-structure} and Prop. \ref{prop:bimodule-structures-invariance}  all apply here: $CM_*(\Omega L_0; f, g)$ is an iterated extension of free $(C_*(\Omega L_0, \p),C_*(\Omega L_0, \p))$ bimodules and the maps $C, H$ are compatible with this bimodule structure, making $CM_*(\Omega L_0; f, g)$ independent of the pair $(f,g)$ up to quasi-isomorphism of bimodules that is canonical in the homotopy category of bimodules. 
\end{remark}
\subsection{Morse-Floer-comparison}
\label{sec:morse-floer-comparison}
\begin{proposition}
	\label{prop:morse-floer-comparison}
	Let $L_0$ be an oriented exact Lagrangian submanifold of a Liouville domain. 
	
	Given the choices in (\ref{eq:morse-smale-choice}), (\ref{eq:morse-theory-pin-choice}), there exist regular Floer data $\bar{H}, \bar{J}$ for $(L_0, L_0)$, choices as in Section \ref{sec:loopy-floer-assumptions}, and an isomorphism
	
	\[Y: CF(\Omega L_0, \Omega L_0;\bar{H}, \bar{J}) \to CM(\Omega L_0; f, g)
	\]
	of $\Z/2$-graded iterated extensions of free $C_*(\P L, \p)(y_b, y_b)$-modules. 
\end{proposition}
\begin{proof}
	This follows from the same argument as in the proof of Lemma \ref{lemma:morse-floer-easy}. We make the choices of Floer data $\bar{H}, \bar{J}$ and auxiliary choice of Pin structures (\ref{eq:pin-structures-choice-floer}) as in the proof of that Lemma. As in that lemma, there is an identification of critical points for $f$ with Hamiltonian chords of $\bar{H}$, and under this identification, the $\Z/2$-graded abelian groups $CF(\Omega L_0, \Omega L_0;\bar{H}, \bar{J})$, $CM(\Omega L_0; f, g)$, are \emph{equal}, so $Y$ is the identity map. The  structures of the complexes as modules over $C_*(\P L, \p)(y_b, y_b)$ are also \emph{equal}; Lemma \ref{lemma:floers-lemma} gives a bijection between the terms in the respective differentials, and the signs agree by the same argument as in Lemma \ref{lemma:morse-floer-easy}. 
\end{proof}

\section{Computing the Morse Complex}
\label{sec:computing-the-morse-complex}
In this section, we compute $CM_*(\Omega L; f, g)$ in terms of algebraic topology. Using a Morse-to-simplicial comparison, we will show that this complex is isomorphic, as a $dg$-bimodule, to the diagonal bimodule of $A_L := C_*(\P_{b,b}L, \mathfrak{p}_{b,b})$. The argument in this section is similar in spirit to the work of Barraud-Cornea, ``Lagrangian Intersections and the Serre Spectral Sequence'' \cite{barraud-cornea}, and we will similarly use a moduli space considered by Barraud-Cornea (\emph{loc. cit.}) and Hutchings-Lee \cite{hutchings-lee}.

\subsection{A well-known moduli space}
It is well-known that Morse function on a manifold gives a CW structure on the manifold, with the open cells given by the (un)stable manifolds of the critical points. The proof of this statement requires the description of \emph{closed cells} for this CW structure, which can be constructed by a certain compactification of the (un)-stable manifolds of the  Morse function, which we describe in this section.  A careful treatment of a \emph{smooth} manifold-with-corners structure on the moduli space described in the subsequent definition is given in \cite{burghlea}, and a more recent, self-contained treatment can be found in \cite{wehrheim}. 
\begin{definition}
	Let $f$ be a Morse function on a closed manifold $L$ with a unique \emph{maximum} $m$, and let $x$ be a critical point of $f$. Choose a metric $g$ that is Morse-Smale with respect to $f$. A broken negative gradient trajectory $v$ of $f$ can naturally be thought of as a subset $\Gamma_\gamma$ of $L$. The \emph{blow-up of the \textbf{stable} manifold of $x$} is a space
	\[ \widehat{\M}(x) := \overline\M^M(m, x; f, g) \times [f(x), f(m)] / \sim \]
	where $\sim$ is the equivalence relation which identifies pairs $(v,t), (v', t')$ of a (possibly broken) negative gradient trajectories $v, v'$ of $f$ and number $t, t'$ of $f$ if $t = t'$ and $\Gamma_v \cap f^{-1}((-\infty, t)) = \Gamma_{v'} \cap f^{-1}((-\infty, t)).$
\end{definition}
\begin{remark}
	\label{remark:stable-vs-unstable}
	In \cite{barraud-cornea}, the Morse function $f$ are required to have unique \emph{minima}, and the \emph{blow-up of the \textbf{unstable} manifold} is defined up instead (see Section 2.4.6 of \cite{barraud-cornea}). Similarly, in \cite{burghlea}, Definition 5, the ``completed unstable manifold'' is defined and is proven to be a smooth manifold with corners in Theorem 1 of \cite{burghlea}.  However, in our conventions for Morse theory, the Morse differential \emph{increases} the value of the Morse function at a critical point. There is no analytical difference between the two conventions.
\end{remark}

For the remainder of this section we suppress the background choice of Morse-Smale pair $f, g$. The evaluation map
\[ \overline{\M}^M(m, x) \times [f(x), f(m)] \to M \]
\[ (v, t) \mapsto \Gamma_v \cap f^{-1}(t)\]
evaluating a broken Morse trajectory at a level set of $f$ then factors through $\widehat{\M}(x)$, giving a map 
\begin{equation}
\label{eq:structure-map-for-blown-up-manifolds}
\lambda_x: \widehat{\M}(x) \to L. 
\end{equation}
Up to the modification of conventions in Remark \ref{remark:stable-vs-unstable}, it is proven in  \cite{barraud-cornea} that 
\begin{lemma}[\cite{barraud-cornea}, Lemma 2.15]
	The space $\widehat{\M}(x)$ is a topological disk of dimension $\dim L - \ind x$. Its boundary has a decomposition
	\[ \widehat{\M}(x) = \cup_{y \in Crit(f)}  \widehat{\M}(y) \times \overline{\M}^M(y, x), \]
	and the map $\lambda_x$ restricted to $\widehat{\M}(y) \times \overline{\M}^M(y, x)$ is equal to projection to $\widehat{\M}(y)$ followed by $\lambda_y$.
	
	In particular, the maps $\{\lambda_x\}_{x \in Crit(f)}$ give a CW decomposition of $L$. 
\end{lemma}

For any critical point $x$ of $f$, let  $o(T^+_x)$ be the orientation line of the positive eigenspace of the Hessian of $f$ at $x$; this agrees with the orientation line of $\overline{\M}(x)$, which compactifies the stable manifold of $x$.   Because $L$ is oriented, there is an isomorphism
\begin{equation}
\label{eq:orientation_line_and_dual}
o(T^+_x) \simeq o(T^-_x)^\check
\end{equation}
 for every critical point $p$, and we will use $o(T^-_x)^\check$ to orient $\widehat{\M}(x)$. Thus, for every pair $x, y \in Crit(f)$, there is an isomorphism of orientation lines
\begin{equation}
\label{eq:compactified-morse-orientations}
o(T^+_y) \tensor \Z^\check_{\partial/\partial s} \tensor o(T^-_y) \tensor o(T^-_x)^\check \simeq o(T^+_x)
\end{equation} 
where one multiplies by $(-1)$, commutes the $\Z^\check_{\partial/\partial s}$ to the left hand side and removes it via a grading-shifting isomorphism, and then applies (\ref{eq:orientation_line_and_dual}) twice and cancels the two orientation lines at $y$. (The multiplication by $(-1)$ comes from the fact that this is a moduli space of downwards Morse flows, but we orient it with a vector that points up the flow.) Define the map
\begin{multline}
\label{eq:ez-final}
\overline{EZ}: \left(C_*(\widehat{\M}(y), \Z) \tensor o(T^+_y)\right) \tensor 
\left(C_*(\overline{\M}^M(y, x), \Z) \tensor \Z^\check_{\partial/\partial s} \tensor o(T^-_y) \tensor o(T^-_x)^\check) \right)  \\\to  C_*(\widehat{\M}(x), \Z) \tensor o(T^+_x) 
\end{multline}
as the tensor product of the Eilenberg-Zilber map on the singular chain complexes by map of singular chains induced by the inclusion 
\begin{equation}
\label{eq:stable-mfld-boundary-component}
\widehat{\M}(y) \times \overline{\M}^M(y, x) \to \partial\widehat{\M}(x) \to \widehat{\M}(x)
\end{equation}
and the isomorphism (\ref{eq:compactified-morse-orientations}) on orientation lines. Orientation theory for Morse moduli spaces, reviewed in Appendix \ref{sec:morse-theory-orientations}, gives an identification of the two tensor factors in the domain of $\overline{EZ}$ with the complexes of singular chains on $\widehat{\M}(y)$ and $\overline{\M}^M(y, x)$ with coefficients in the orientation lines of the respective moduli spaces. Moreover, due to the multiplication by $(-1)$, the map $\overline{EZ}$ is the same map, under this identification, as the composition of the Kunneth map on orientation lines followed by the product and boundary isomorphisms of orientation lines induced by the inclusion into the boundary   (\ref{eq:stable-mfld-boundary-component}).

Thus, we are in the setting of Section \ref{sec:compatible-fundamental-cycles}, and the argument of that section shows that
\begin{lemma}
	Given a compatible system of fundamental chains for the moduli spaces $\overline{\M}^M(x,y))$ in the sense of Section \ref{sec:compatible-fundamental-cycles}, we can extend this system to a compatible system of fundamental chains $[\widehat{\M}(x)] \in C_*(\widehat{\M}(x), \Z) \tensor o(T^+_x)$, in the sense that these chains satisfy the identity 
	\begin{equation}
	\label{eq:unstable_manifold_eq}
	\partial [\widehat{\M}(x)] = \sum_{y} \overline{EZ}([\widehat{\M}(y)] \tensor [\overline{\M}^M(y, x)])
	\end{equation} 
	where $\overline{EZ}$ is the map defined in (\ref{eq:ez-final}).
\end{lemma}

\subsection{A convenient model for the diagonal bimodule}
Write $\P_{b,*} L$ for the space of Moore paths with starting point at $b \in L$, and write $\P_{*, b} L$ for the space of Moore paths with endpoint at $b$: namely, 
\[ \P_{b,*} := \bigcup_{x \in L} \P_{b, x}, \]
\[\P_{*, b} := \bigcup_{x \in L} \P_{x, b},\]
topologized with subspace topology under the inclusion into $C^0([0, \infty), L)$. We will call the point $x$ in the above disjoint union the \emph{marked point} of a Moore path in $\P_{b, x}$ or $\P_{x, b}$.

There are evaluation maps 
\begin{equation}
ev: \P_{b, *} L \to L, \text{ and } ev: \P_{*, b} L \to L,
\end{equation} 
sending a Moore path to its marked point, 
 and there is a fibration $\pi: \P_{b, *} L \times_L \P_{*,b} L \to L$ given by the fiber product of the evaluation maps. 
 This the total space of this fibration is the domain of a map $\epsilon:  \P_{b, *} L \times_L \P_{*,b} \to \P_{b,b}L$ which fiberwise is the concatenation of paths at their marked points, and by ``splitting a Moore path at its midpoint'', one can construct a homotopy inverse, proving the
 \begin{lemma}
 	The map $\epsilon$ is a weak equivalence. 
 \end{lemma} 

 Write $\mathfrak{p}$ for the pullback by $\epsilon$ the local system $\mathfrak{p}_{x,x}$ over $\P_{b,b}L$ defined in Section \ref{sec:twisted-fundamental-group}. 
We can then consider the chain complex $C_*(\P_{b, *} L \times_L \P_{*, b} L; \mathfrak{p})$. This is an $(A_L, A_L)$-bimodule, via the map on singular chain complexes induced by concatenation of paths at $b$ \[\P_{b, b}L \times \left(\P_{b, *} L \times_L \P_{*, b} L \right) \times \P_{b, b}L \to \P_{b, *} L \times_L \P_{*, b} L\]
together with concatenation of Pin structures over the paths. The map $\epsilon$ is a map of topological monoids which is an equivalence of underlying spaces, and since $\mathfrak{p}$ is pulled back from $\mathfrak{p}_{x,x}$ in a way that is compatible with the isomorphisms covering the left and right topological monoid structures of the underlying spaces over $\P_{b,b}L$, this proves the
\begin{lemma}
	The map 
	\[ \epsilon_*: C_*(\P_{b,*} L \times_L \P_{*,b} L, \mathfrak{p}) \to C_*(\P_{b,b}L, \mathfrak{p}_{b,b}) \]
	is a quasi-isomorphism of bimodules from its domain to the diagonal bimodule of $A_L$. $\qed$
\end{lemma}

We will view this quasi-isomorphism of bimodules as a map of right modules over $A_L^{op} \tensor A_L$ via the equivalence between bimodules and right modules described in Section \ref{sec:bimodule-structure}.
\subsection{The computation, continued}

Now, by definition, every point $r \in \widehat{\M}(x)$ corresponds to a Moore path $\gamma_r \in \P_{x, *}L$, and structure maps $\lambda_x$ factor through the map $r \mapsto \gamma_r$ (see Equation \ref{eq:structure-map-for-blown-up-manifolds}). Writing $i^*\gamma_r$ with for the inverse of the Moore path $\gamma_r$, we get a map 
\begin{equation}
\widehat{\M}(x) \ni r \mapsto \bar{\gamma_r} := (\gamma_r, i^*\gamma_r) \in \P_{x, *} L \times_L \P_{*, x} L.
\end{equation}
 As in the definition of the functor $\Delta$ (Def. \ref{def:delta-functor}), there is a canonical section of $\p$ over $\bar{\gamma_r}$, and as in Section \ref{sec:morse-theory-fundamental-classes} we can use this to define a chain map

 \begin{equation}\phi: C_*(\widehat{\M}(x)) \tensor o(T_p^+) \to C_*(\P_{x, *} L \times_L \P_{*, x}; \mathfrak{p}) \tensor o(T_p^+).
 \end{equation} 

We now define a map 
\begin{equation}
\mathcal{F}: CM_*(\Omega L; f, g) \to C_*(\P_{b, *} L \times_L \P_{*, b} L, \mathfrak{p}).
\end{equation}

Let $p$ be a critical point of $f$. Consider an element 
\begin{equation}\gamma \in \Z^\check_p \tensor o(T^-_p) \tensor C_*(\P L, \p)(p, b) \subset CM_*(\Omega L; f, g).
\end{equation} We define 
\begin{equation}\mathcal{F}(\gamma) := (-1)^nE ( \phi_*[\M(p)] *_1 \gamma)
\end{equation} 
where $*_2$ is defined analogously as in Equation \ref{eq:second-operation}: one commutes the trivial line that is part of $\gamma$ through the line $o(T_p^-)^\check $, which has degree $\text{ind }p$, and applies the concatenation operation
\begin{equation}
C_*(\P_{p, *} L \times_L \P_{*, p}; \mathfrak{p}) \tensor C_*(\P L, \p)(p, b) \to  C_*(\P_{b, *} L \times_L \P_{*, b} L, \mathfrak{p}) 
\end{equation} 
pairs off the $o(T_p^-)$ and $o(T_p^+)$ using (\ref{eq:orientation_line_and_dual}), and removes the extraneous trivial line. 

Note that the map $\mathcal{F}$ cannot  (yet!) be a chain map because it does not even preserve degree. However, as in Section \ref{sec:bimodule-structure}, the map $\mathcal{F}$ commutes with the right-module structure on $CM_*(\Omega L; f, g)$ because of Lemma \ref{lemma:left-compostion-is-module-map}. We now show that $\mathcal{F}$ commutes with the differentials up to a global sign
\begin{lemma}
	The map $\mathcal{F}$ satisfies
	\[ \mathcal{F}d = (-1)^{n} d \mathcal{F}.\]
\end{lemma} 
\begin{proof}
The dimension of $\widehat{\M}(p)$ is $n - \ind p$. Thus when we commute the boundary operator on $C_*(\P_{b, *} L \times_L \P_{*, b} L, \mathfrak{p})$ through a copy of $\phi_*(\widehat{\M}(p))$ we pick up a sign of $(-1)^{n -\ind p}$, which together with the sign in the first term on the differential (\ref{eq:morse-complex-differential})on $CM_*(\Omega L; f, g)$ gives a total sign of $(-1)^n$. Similarly, the $(-1)$ in the map (\ref{eq:compactified-morse-orientations}) together the commutation of the trivial line through $o(T^+_p) \tensor \Z^\check_{\partial/\partial s} \tensor o(T^-_p) \tensor o(T^-_q)^\check$, combined with the sign in second term in the differential on $CM_*(\Omega L; f, g)$, again give a total sign $(-1)^n$.


\end{proof}
Thus, after shifting the Morse complex, thought of as a $dg$-module, up by degree $n$, we get that $\mathcal{F}$ defines a map of $dg$-modules. 
\begin{lemma}
	The resulting map $\mathcal{F}$ is a quasi-isomorphism.
\end{lemma}
\begin{proof}
  Equip $CM_*(\Omega L; f, g)$ with the filtration given by $\dim L - \ind x$ for $x \in Crit f$, and the target with the filtration coming from the filtration on $\P_b L \times_L \P_b L$ pulled back by the filtration on $C_*(L)$ induced by the CW filtration of $L$ coming from  the blown-up stable manifolds. With these filtrations, $\mathcal{F}$ is a map of filtered complexes. It is an isomorphism on the $E_1$ page, which on both sides is just 
  \begin{equation}
  \bigoplus_{x \in Crit f, q \geq 0} H_q(\P_{b, x}L \times \P_{x,b}L; \mathfrak{p} \tensor \mathfrak{f}) \M(x).
  \end{equation} So it is an equivalence. 
\end{proof}

Combining the three lemmata of this section, we have proved the 
\begin{proposition}
	\label{prop:convenient-model}
	When defined, the complex
	\[CM(\Omega L_0; f, g)[n]\]
	is quasi-isomorphic as a bimodule to the diagonal bimodule of $C_*(\Omega L_0, \p)$. 
\end{proposition}
\subsection{Proof of Proposition \ref{prop:resolution}}\label{sec:proof-of-proposition-resolution} 
\begin{proof}
	Let $L_1$ be a transversely intersecting Hamiltonian isotopy of $L_0$  obtained from the time-dependent hamiltonian $H$. By Lemma \ref{prop:hamiltonian-is-part-of-regular-floer-datum}, there exists a $J$ so that $(H, J)$ is a regular Floer datum. Making the choices as in Section \ref{sec:loopy-floer-assumptions}, we can define $CF(\Omega L_0, \Omega L_0; H, J)$, which is an iterated extension of free bimodules of size equal to the number of intersection points of $L_0$ and $L_1$ (see Sec. \ref{sec:bimodule-structure}). By Propositions \ref{prop:bimodule-structures-invariance} and \ref{prop:morse-floer-comparison}, it is quasi-isomorphic as a $(C_*(\Omega L_0, \p),(C_*(\Omega L_0, \p))$  to $CM(\Omega L_0; f, g)$ for some Morse-Smale pair $(f, g)$. The latter bimodule is quasi-isomorphic to the diagonal bimodule by Proposition \ref{prop:convenient-model} . The quantity in the proposition is by definition a lower bound on the size of $CF(\Omega L_0, \Omega L_0; H, J)$.
\end{proof}
\section{Appendix}
\renewcommand{\thesubsection}{\Alph{subsection}}

\subsection{Conventions in $dg$-algebra}
\label{sec:dg-algebra-conventions}
In this section we describe our conventions on $dg$-algebras, $dg$-modules, and iterated extensions of free modules.

For this section, $k$ is a natural number. Write $S$ for $\Z/2k\Z$ and write $| \, \cdot \, |: S \to \Z/2\Z$ for the projection. Let $R$ be a commutative ring.

\begin{definition}
	An $S$-graded $R$-module is an $R$-module $M$ equipped with a decomposition $M = \oplus_{s \in S} M_s$, where $M_s$ denotes the part of \emph{degree} $s$. We say $m \in M$ is \emph{homogeneous} if it lies in one of the $M_s$. The $R$-module of $R$-linear homomorphisms $Hom_R(M, N)$ between a pair of $R$-modules $M, N$ is $S$-graded: an endomorphism $E$ has degree $s$ if $E$ maps $M_{s'}$ to $N_{s+s'}$ for all $s \in S$. 
	
	An $S$-graded $R$-chain complex is an $S$-graded $R$-module $M$ equipped with an endomorphism $d: M \to M$ of degree $-1$, called the differential, such that $d^2 = 0$. A \emph{map} of chain complexes is a map of $S$-graded $R$-modules of degree zero which commutes with the differential. In the rest of this section, we will call a $S$-graded $R$-chain complex a \emph{chain complex}. A \emph{morphism} of chain complexes is an element $f \in Hom_R(M, N)$; the set of morphisms is a chain complex, with $(df)(m) = d(f(m)) - (-1)^{|m|}f(dm)$. 
	
	There are obvious notions of a subcomplex, a quotient complex, and of a direct sum of complexes; in the latter, the degree of $a \oplus b$ is defined only if $a$ and $b$ are homogeneous of equal degree, in which case it is the degree of $a$.
	
	The \emph{tensor product} $M \tensor N$ of two chain complexes $M, N$, is the $R$-module $M \tensor_R N$, with the grading characterized by the property  that if $m \in M_s, n \in N_{s'}$, then $m \tensor n \in M_{s+s'}$, and the differential 
	$d$ characterized by the property that $d(m \tensor n) = dm \tensor n + (-1)^{|m|}m \tensor dn$. This makes chain complexes into a symmetric monoidal category using the usual Koszul sign rule for the braiding.

	A $S$-graded $dg$-algebra over $R$ is a $S$-graded $R$-chain complex $A$ equipped with a map of chain complexes $A \to A$, $(a, b) \mapsto a \cdot_A b$ which satisfies the obvious associativity relation for an associative product. We will call an $S$-graded $dg$-algebra over $R$ a \emph{dg-algebra} for the remainder of this section. Notice that if we forget the differential and the grading, then $A$ is an $R$-algebra. We will write $a \cdot_A b = ab$ whenever the context is clear.
	
	A \emph{dg-category} is a category enriched in chain complexes; thus $dg$-algebras are $dg$-categories with one object.
	
	A \emph{right module} $M$ over a dg-algebra $A$ is a chain complex $M$ equipped with a map of chain complexes $M \tensor A \to M$, $m \tensor a \mapsto m a$ which makes $M$ into a right $A$-module over the $R$-algebra $A$. Similarly, a \emph{left module} $M$ over a $dg$-algebra $A$ is a chain complex $M$ equipped with a map of chain complexes $A \tensor M \to M$, $a \tensor m \mapsto am$,   which makes $M$ into a left $A$-module over the $R$-algebra $A$. A \emph{map}, or a \emph{morphism}, of right modules is a map or morphism $\phi$ of chain complexes which commutes with multiplication in the obvious way, i.e. $\phi(ma) = \phi(m)a$; the morphisms between two right modules naturally form a chain complex.
	
	There are obvious notions of a submodule, of a quotient module, and of a direct sum of modules. A \emph{free} module over $A$ is a module isomorphic to $A$ thought of as an $A$-module.
\end{definition}

\begin{definition}
	A \emph{homotopy} between two maps of chain complexes $f_0, f_1: M \to N$ is an element $f\in Hom_R(M, N)_1$ with $df = f_0-f_1$. Chain complexes under composition of maps form a category, and the existence of a homotopy betwen a pair of maps is an equivalence relation on this category; the quotient by this equivalence relation is the \emph{homotopy category of $S$-graded $R$-chain complexes}. 
	
	Similarly, a homotopy between two maps $f_0, f_1: M \to N$ of right modules over a $dg$-algebra $A$ is a degree $1$ morphism of modules with $df = f_0 - f_1$, and the \emph{homotopy category of right $A$-modules} is the quotient of the category of maps of right $A$-modules by the equivalence relation given by the existence of a homotopy.
\end{definition}

\begin{definition}
	Taking cohomology with respect to $d$ defines a functor from the category of chain complexes and maps between them to the category of $S$-graded $R$-modules.
	
	A \emph{quasi-isomorphism} of chain complexes is a map of chain complexes inducing an isomorphism on cohomology; a quasi-isomorphism of right $A$-modules is a map of $A$-modules which is a quasi-isomorphism of the underlying chain complexes.
\end{definition}

\begin{definition}
	\label{def:tensor-product-dga}
	The \emph{tensor product} of two $dg$-algebras $A, B$ is, as a chain complex, the tensor product of chain complexes, with the multiplication characterized by the property that for $a, a' \in A$, $b, b' \in B$ all of pure degree, we have that 
	\[ (a' \tensor b')(a \tensor b) = (-1)^{|b'||a|} (a'a) \tensor (b'b).\]
\end{definition}
\begin{definition}
	Given a $dg$-algebra $A$, the \emph{opposite $dg$-algebra} $A^{op}$ is the $dg$ algebra with the same underlying 
	chain complex as $A$, but with $a \cdot_{A^{op}} b  = (-1)^{|a||b|}b \cdot_A a$ for elements $a, b \in A$ of pure degree.

	There is a bijection between right modules $M$ over $A$ and left modules $M^{op}$ over $A^{op}$: one sets the underlying chain complex of $M^{op}$ to the be that of $M$, and one defines $am = (-1)^{|m||a|}(am)$. 
\end{definition}
\begin{remark}
		Maps, morphisms, etc. of left modules are defined such that they naturally give the already-defined notion for the corresponding right modules; this introduces certain Koszul signs into the theory of left modules over a dg-algebra, which we will not write out explicitly.
\end{remark}

\begin{definition}
	The \emph{shift} $M[k]$ of a \emph{right} module $M$ over a $dg$-algebra $A$ is the right $A$-module obtained by redefining the degree of a degree $s$ element of $M$ to be $s + k$, and the differential is multiplied by $(-1)$.
	
	The \emph{shift} $M[k]$ of a \emph{left} module $M$ over a $dg$-algebra $A$ is defined as $(M^{op}[k])^{op}$. 
\end{definition}

\begin{definition}
	Given two dg-algebras $A, B$, we define an $(A, B)$-\emph{bimodule} to be a chain complex equipped simultaneously with 
	the structure of a left $A$-module and a right $B$-module, which commute in the sense that $(am)b = a(mb)$ for $a \in A, m \in M, b \in B$. The sign in Def. \ref{def:tensor-product-dga} gives a bijection between $(A, B)$-bimodules, right $A^{op} \tensor B$ modules, and left $A \tensor B^{op}$ module. 
	
	The multiplication on $A$ makes $A$ into an $(A, A)$-bimodule called the \emph{diagonal} bimodule. For most of this paper we will think of the diagonal bimodule as a right $A^{op} \tensor A$-module. 
\end{definition}

\begin{definition}
	\label{def:twisted-complex}
	Suppose we are given a right module $M$ over a $dg$-algebra $A$ admitting an $\R$-filtration by submodules $M^{\leq r} \subset M$ for $r \in \R$, write $M^{<r}$ for the smallest submodule containing all $M^{\leq \ell}$ for $\ell < r$; this makes sense because the set-theoretic intersection of submodules is a submodule.
	
	A \emph{iterated extension of free modules} is an $M$ as above with the property that the $A$-modules $M^{\leq r}/M^{<r}$ are quasi-isomorphic as right $A$-modules to direct sums of shifts of $A$ considered as a right $A$-module, i.e. 
	\[ M^{\leq r}/M^{<r} \simeq \oplus_{i=1}^{a_r} A[b_i]\]
	for some cardinals $a_r$ and some integers $b_i$; and more over that the quotients $M^{\leq r}/M^{<r}$ are nonzero for only finitely many $r$, and when they are nonzero the cardinals $a_r$ are finite.
	
	The \emph{size} of an iterated extension of free modules $M$ is
	\[ \sum_{ M^{\leq r}/M^{<r} \neq 0} a_r.\]
	
	Maps and quasi-isomorphisms of iterated extensions of free modules are simply maps and quasi-isomorphisms of the underlying $A$-modules.
	
	A interated extension of free $(A, B)$-bimodules is an iterated extension of free modules over $A^{op} \tensor B$.
\end{definition}

\begin{remark}
	What we call an iterated extension of free modules is  quasi-isomorphic to the usual notion of a twisted complex, where one requires that the subquotients $M^{\leq r}/M^{<r}$ are \emph{equal} to direct sums of shifts of $A$; with that change, the category of twisted complexes can be described very explicitly. In particular, \emph{an iterated extension of free modules of size $s$ is quasi-isomorphic to a twisted complex made of $s$ underlying free modules.} The paper \cite{bondal_kapranov} used the notion of a twisted complex to give an explicit $dg$-model for derived categories with a finite number of generators. Section I of \cite{seidel} gives a nice introduction to the machinery of $\Z$-graded twisted complexes in the $A_\infty$ setting. 
	
    The differential in our dg-algebras decreases degree, which means that our conventions disagree with the above two sources, but there should be no trouble converting between the conventions. Our basic conventions agree with those in the Stacks project \cite[\href{https://stacks.math.columbia.edu/tag/09JD}{Chapter 09JD}]{stacks-project}. Our notion of a an iterated extension of free modules is equivalent to that of \cite{seidel} when $S = \Z$ after only remembering the ordering of the nonzero sub-quotients of the $\R$-filtration, and replacing the condition $M^{\leq r}/M^{<r}$ is quasi-isomorphic to $A$ by the condition that it is equal to $A$. We use quasi-isomorphism primarily because in Lemma \ref{lemma:left-compostion-is-module-map}, the map in (\ref{eq:tensor-product-to-hom-map}) is only a quasi-isomorphism rather than an equality, of chain complexes; see also Remark \ref{rk:homotopy-vs-qis}.
\end{remark}

\subsection{Pin groups and Pin structures}
\label{sec:pin-conventions}
In this section we clarify what we mean when we discuss $Pin$ structures, because there are several conventions in the literature. In particular, in Appendix \ref{sec:gluing_pin_structures_relative_to_ends} we explain the language of relative $Pin$ structures that is used throughout the text in.

\subsubsection{Pin structures}

The Lie group $Spin(n)$ is the universal cover of the Lie group $SO(n)$. The Lie group $O(n) \supset SO(n)$ is not connected, so there is not a unique group structure on its universal cover. The Lie group $Pin(n)$ is characterized by the property that it admits a map of Lie groups
$\pi: Pin(n) \to O(n)$ which is a universal cover on each component of $O(n)$, such that the inverse image under $\pi$ of the group generated by a reflection in $O(n)$ is isomorphic to $\Z/4\Z$. An explicit description of the map $\pi$ is given in Appendix \ref{sec:pin-minus-construction}

Given a real vector bundle $E \to B$ over a topological space $B$ with a Riemannian metric on $E$, we can form the \emph{frame bundle} of $E$, $O(E) \to B$, which is a principal $O(n)$-bundle over $B$. We say that a \emph{Pin structure} on $E$ is a principal $Pin(n)$-bundle $P(E) \to B$ which admits a map over $B$ to $O(E)$ that respects the $Pin(n)/O(n)$ torsor structures on both sides. A map of Pin structures on $E$, is a map of principal $Pin(n)$-bundles over $B$, $P(E) \to P'(E)$, over $O(E)$. 

\begin{proposition}[\cite{kirby_taylor_1991}, Lemma 1.3, Remark on p. 184]
	\label{prop:pin_structure_classification}
	Suppose that $B$ is paracompact Hausdorff and has the homotopy type of a $CW$ complex. Then
	 bundle $E$ admits a Pin structure iff $w_2(E) = 0$. If $E$ admits a Pin structure, then the set of Pin structures on $E$ up to isomorphsim is a torsor over $H^1(B, \Z/2\Z)$. 
\end{proposition}

\subsubsection{The group $Pin$}
\label{sec:pin-minus-construction}
Our definition of $Pin(n)$ agrees with that of $Pin^+(n)$ in \cite{kirby_taylor_1991}. 

Let $V$ be a real vector space with a positive-definite inner product $(\cdot, \;\cdot)$. Let $Cliff^+(V)$ be the universal associative $\R$-algebra generated by $V$ satisfying the relation
\[ vw + wv = 2(v, w). \]
Let $Pin(V) \subset Cliff^+(V)$ be the multiplicative submonoid generated by the unit sphere in $V$; this is, in fact, a Lie group. An element $x \in Pin(V)$ acts on $V$ via
\[ v \mapsto -xvx; \]
this action preserves the norm on $V$, and so gives a homomorphism $\pi: Pin(V) \to O(V)$, which is the covering map used in Appendix \ref{sec:pin-conventions}.

=\subsubsection{Relative Pin structures}
\label{sec:gluing_pin_structures_relative_to_ends}
Let $\ell > 0$ and $E \to [0, \ell]$ be a Riemannian bundle equipped with a pair of Pin structures $\mathfrak{p}_0$, $\mathfrak{p}_1$ on $E_0 \to 0$, $E_\ell \to \ell$, respectively; we say that $E$ is a bundle equipped with \emph{Pin structures at the ends}.  A \emph{Pin structure relative to the ends} on $E$ is a choice of Pin structure $\p$ on $E$ together with a pair of isomorphisms of pin structures $\lambda_\p^0: \p|_{E_0} \to \p_0$, $\lambda_\p^1: \p|_{E_1} \to \p_1$. An \emph{isomorphism of Pin structures relative to the ends} on $E$ is an isomorphism of Pin structures $\phi: \p \to \p'$ such that for $k = 0, 1$, $\lambda_{\p'}\phi|_k = \lambda_\p$. 

\begin{proposition}
	\label{prop:pin_structure_relative_classification}
	Let $E$ be a bundle equipped with Pin structures at the ends. The set of Pin structures relative to the ends on $E$ up to isomorphism is a $\Z/2$-torsor. 
\end{proposition}

\begin{proof}
	An isomorphism of Pin structures $\p_0 \to \p_1$, which always exists, gives a map $\tau$ from Pin structures relative to the ends on $E$ to Pin structures on a certain vector bundle $\bar{E}$ on a circle. This map sends isomorphic Pin structures relative to the ends on $E$ to isomorphic Pin structures on that vector bundle. The isomorphism classes of the second kind are a torsor over $H^1(S^1, \Z/2) = \Z/2$; so it suffices to check that $\tau$ is a bijection on isomorphism classes. Surjectivity is straightforward, and injectivity follows from the fact that any isomorphism of Pin structures on a vector bundle $\bar{E}$ over $S^1$, where one already has another specified isomorphism $\bar{\lambda}$ between the Pin structures restricted to $E|_{0}$, gives rise via a gauge transformation to another isomorphism between the Pin structures that is equal to $\bar{\lambda}$ at $E|_{0}$.
\end{proof}

\begin{definition}
\label{def:set-of-pin-structures-at-ends}
If $E \to [0, \ell]$ is a bundle with Pin structures at the ends, let 
\begin{equation}
\label{eq:def-set-of-pin-structures-at-ends}
\Pi(E) 
\end{equation}denote the set of isomorphism classes of Pin structures relative to the ends on $E$.
\end{definition}

Given a pair of Riemannian vector bundles $E \to [0, \ell]$, $E' \to [0, \ell']$ equipped with Pin structures at the ends $\p^E_k, \p^{E'}_k$, $k = 0, 1$, together with identifications of Riemannian vector spaces $E|_{\ell} = E'_0$, and an identification of the corresponding Pin structures at the ends $\p^E|_1 = \p^{E'}|_0$, one can \emph{glue} the vector bundles to get a vector bundle $E \# E' \to [0, \ell + \ell']$ equipped with pin structures at the ends given by $\p^E_0, \p^{E'}_1$,  and one has a corresponding operation 
\begin{equation}
g_{E, E'}: \Pi(E) \times \Pi(E') \to \Pi(E \# E').
\end{equation} 
Given a third bundle $E'' \to [0, \ell'']$ equipped with Pin structures at the ends $\p^{E''}_0, \p^{E''}_1$, and an identification $\p^{E'}_1 \to \p{E''}_0$, one has that 
\begin{equation}E \# E' \# E'' := (E \# E') \# E'' = E \# (E' \# E'')
\end{equation} and 
the two possible gluing maps 
\begin{equation}\Pi(E) \times \Pi(E') \times \Pi(E'') \to \Pi(E \# E' \# E'')\end{equation} given by 
\begin{equation}
\begin{gathered}
g_{E \# E', E''}\circ (g_{E, E'} \times 1), \text{ and } \;g_{E, E' \# E''}\circ ( 1 \times g_{E', E''}),
\end{gathered}
\end{equation} 
are equal. 

\begin{remark}
	\label{rk:unital-pin-structure}
	The notion of Pin structure relative to the ends on $E \to [0,\ell]$ can be extended to the case when $\ell = 0$. In that case, $E_0 = E_1$ and we require that the pin structures $\p_0$ and $\p_1$ on $E_0$ and $E_1$ respectively satisfy $\p_0 = \p_1$. In this case we define $\Pi(E) = \Z/2$; the element $0 \in \Z/2$ is the pin structure $\p = \p_0 = \p_1$ on $E$, and the element $1$ is a ``formal inverse'' of that pin structure. Clearly $\Pi(E)$ is still a $\Z/2$-torsor; moreover, given another bundle $E' \to [0, \ell']$, there are gluing maps 
	\[ \Pi(E) \times \Pi(E') \to \Pi(E'),\]
	\[ \Pi(E') \times \Pi(E) \to \Pi(E),\]
	defined using the action of $\Pi(E) = \Z/2$ on $\Pi(E')$, and these satisfy the associativity axioms with the previously defined gluing maps.
	
	This special case is defined so as to make sense of the $dg$-algebras $C_*(\P_{x,x}L, \p_{x,x})$ (\ref{eq:multiplication}) and $C_*(\P L, \p)(y_b, y_b)$(\ref{eq:tensor-product-to-hom-map}); in particular, the \emph{unit} in $C_*(\P_{x,x}L, \p_{x,x})$ is the ``Pin structure'' over the constant loop at the basepoint corresponding to $0 \in \Z/2 = \Pi(T_{x}L)$. 
\end{remark}

\subsection{Technical conventions for Floer theory}
\label{sec:technical-floer-appendix}
In this appendix we describe our conventions for the equations and moduli spaces of Lagrangian Floer homology.

Let $(M, \omega, \theta)$ be a Liouville domain containing a pair of closed exact Lagrangian submanifolds $L_0, L_1$, as in Section \ref{sec:floer-theory-start}. 
Let $\mathcal{J}$ be the set of $\omega$-compatible almost-complex structures on $M$ which, in some neighborhood of $\partial M$, are $X$-invariant and satisfy $dh \circ J = \theta$.

\subsubsection{Riemann surfaces}

All Riemann surfaces will be (possibly punctured) Riemann surfaces with boundary.

Our standard coordinate on $\C$ is $z = s+it$. 

Define the following Riemann surfaces:
\begin{align*}
&\mathbb{\D} \;\;\,:= \{z \in \C: |z|^2 \leq 1 \} \\
&Z \;\;\,:= \R \times [0,1] \\
&Z^{\pm} := \R^\pm \times [0, 1]\\
&H \;\;:= \{z \in \C | \text{Im}(z) \geq 0 \}.
\end{align*}

\begin{remark}
	Sometimes in this paper $H$ will also denote a Hamiltonian perturbation term; the meaning of the symbol should be clear from the context.
\end{remark}

\begin{definition}
	A \emph{boundary-marked Riemann surface} is a triplet $(\hat{S}, \Sigma^+, \Sigma^-)$, where $\hat{S}$ is a a Riemann surface with boundary $\partial \hat{S}$ and $\Sigma^+, \Sigma^-$ are non-intersecting finite subsets of $\partial \hat{S}$. We will call the sets $\Sigma^+$, $\Sigma^-$ the \emph{outgoing} and \emph{incoming} points, respectively. 
	
	A boundary-marked Riemann surface canonically defines a Riemann surface $S := \hat{S} \setminus (\Sigma^+ \cup \Sigma^-)$. We will denote the boundary-marked Riemann surface by $S$, supressing additional notation.
	
	We give the Riemann surfaces $Z$ and $Z^\pm$ the structure of boundary-marked Riemann surfaces by compactifying them by adding the points at $s = \pm\infty$. We think of the point at $s = + \infty$ as $\Sigma^+$, and the point at $s = -\infty$ as $\Sigma^-$, whenever those points were added to the compactification. Similarly we give $H$ the structure of a boundary-marked Riemann surface by adding on the point at infinity and choosing it as $\Sigma^-$. 
\end{definition}

\begin{definition}
	A choice of \emph{strip-like ends} for a boundary-marked Riemann surface $\Sigma$ is, for every $\zeta \in \Sigma^\pm$, a proper holomorphic embedding $\epsilon_\zeta: Z^{\pm} \to S$ such that $\epsilon_\zeta^{-1}(\partial S) = \R^\pm \times \{0,1\}$, with the images of $\epsilon_\zeta$ pairwise disjoint, and such that $\lim_{s \to \pm \infty} \epsilon_\zeta(s, \cdot) = \zeta$. 
\end{definition}

\subsubsection{Floer data}
Let $\mathcal{J}$ be the set of almost complex structures compatible with $M$, and let $\mathcal{H}$ be the space of smooth functions of $M$ which are zero on a neighborhood of $\partial M$.

\begin{definition}
	A \emph{Floer datum} for the pair $(L_0, L_1)$ is a pair 
	\[(H, J) \in C^\infty([0,1], \mathcal{H}) \times C^\infty([0,1],\mathcal{J})\]
	such that under the time-$1$ flow of the Hamiltonian $H$, $L_0$ intersects $L_1$ transversely.
\end{definition}

\begin{definition}
	A \emph{boundary-marked Riemann surface equipped with boundary Floer data} is 
	\begin{itemize}
		\item a boundary-marked Riemann surface $S$ equipped with a choice of strip-like end $\epsilon_\zeta: Z^{\pm} \to S$ for every end $\zeta \in \Sigma^{\pm}$ of $S$ 
		\item an assignment of a closed exact Lagrangian submanifold $L_i$ for each boundary component $(\partial S)_i$ of $\partial S$,
		\item For every end $\zeta \in \Sigma^{\pm}$,  a choice of Floer data $(H_\zeta, J_\zeta)$ for the pair of Lagrangians $(L'_0, L'_1)$, with $L'_i$ the Lagrangian assigned to the connected component of $\partial S)$ containing $\epsilon_\zeta(*, i)$. 
	\end{itemize}  
	We will denote a choice of boundary Floer data $(\{L_i\}_{i \in \pi_0(\partial S)}, \{(H_\zeta, J_\zeta)\}, \{\epsilon_\zeta\}\}$ by $\mathfrak{B}$. 
\end{definition}

\subsubsection{Perturbation data}
Let $X_t$ be the vector field corresponding to the Hamiltonian flow of $H(t)$ on $M$, namely, the one satisfying the condition
\begin{equation}
\omega(X_t, \cdot) = dH.
\end{equation} 
Let $\mathcal{C}(L_0, L_1; H)$ be the finite set of time-$1$-trajectories of $X_t$ from $L_0$ to $L_1$. 

\begin{definition}
	A \emph{perturbation datum} for a boundary-marked Riemann surface equipped with boundary Floer data is a pair 
	\[(K, J) \in \Omega^1(S, \mathcal{H}) \times C^\infty(S, \mathcal{J}) \] 
	such that:
	\begin{itemize}
		\item 
		for each component $\partial_iS$ of $\partial S$ with corresponding Lagrangian submanifold $L_i$, we have that 
		$K(T\partial_iS)|_{L_i} = 0$; and also such that
		\item the perturbation datum is compatible with the boundary Floer data, in the sense that 
		\begin{align*}&\epsilon^*_\zeta K = H_\zeta(t) dt, \text{ and}\\
		&J(\epsilon_\zeta(s,t)) = J_\zeta(t) \text{ for every } \zeta \in \Sigma.
		\end{align*}
	\end{itemize}
	
	We denote by $Y \in \Omega^1(S, C^\infty(TM))$ the vector valued one form obtained from $K$ by considering the Hamiltonian flows of the corresponding Hamiltonians. 
\end{definition}

\subsection{Floer moduli spaces}
\label{sec:floer-moduli-spaces}
For any $\gamma_\pm \in \mathcal{C}(L_0, L_1; H)$ we can consider the set of solutions to Floer's equation
\begin{equation}
\label{eq:floer_eq}
\widetilde\M^F(\gamma_-, \gamma_+; H, J) := \left \{
\begin{tabu}{c|c}
& \partial_s u + J(t, u)(\partial_t u - X(t,u)) = 0; \\
u \in C^\infty(Z, M) & u(s,i) \in L_i, i = 0,1 \\
& (u(s, \cdot)) \xrightarrow[s \to \infty]{e} \gamma_{\pm}(\cdot)  \\
\end{tabu}
\right\}.
\end{equation}
In the above equation, the notation 
\[ (u(s, \cdot)) \xrightarrow[s \to \infty]{e} \gamma_{\pm}(\cdot),\]
means that the function $u(s, \cdot)$ converges exponentially fast to $\gamma^\pm(\cdot)$ in a local chart on $M$ near $\gamma^\pm$. We topologize $\widetilde\M^F(\gamma_-, \gamma_+; H, J)$ with the topology of uniform convergence. 

The moduli space $\widetilde\M^F(\gamma_+, \gamma_-; H, J)$ admits a continuous $\R$ action given by translation in the $s$ coordinate; let $\M^F(\gamma_+, \gamma_-; H, J)  := \widetilde\M^F(\gamma_+, \gamma_-; H, J) /\R$. We will omit $H$ and $J$ from the notation when the dependence on them is clear from the context.

Given a boundary-marked Riemann surface $\Sigma$ equipped with boundary Floer data $\mathfrak{B} := (\{L_i\}, \{(H_\zeta, J_\zeta) \}$ and perturbation data $(K, J)$, then for any collection $\{\gamma_\zeta\}$ with $\gamma_\zeta \in \mathcal{C}(L_\zeta, L'_\zeta; H)$  we can consider the set of solutions to the \emph{inhomogeneous pseudoholomorphic map equation} 
\begin{equation}
\label{eq:inhomogeneous_map}
\M^C(\{\gamma_\zeta\}; \Sigma, \mathfrak{B}, K, J ) :=   \left\{
\begin{tabu}{c|c}
& (Du - Y)_J^{0,1} = 0, \\
u \in C^\infty(\Sigma, M) & u(\partial_i C) \subset L_i;\;   \\
& \forall_{\zeta \in \Sigma^\pm} u(\epsilon_\zeta(s, \cdot) \xrightarrow[s \to \pm \infty]{e} y_\zeta(\cdot)\\
\end{tabu}
\right\}
\end{equation}
which is topologized in the topology of uniform convergence.

\begin{remark}
	\label{rk:floer_gives_perturbation_data}
	A choice of Floer data $(H, J)$ for $(L_0, L_1)$ equips $Z$ with a perturbation datum given by the Floer datum itself, denoted by $(H, J)^\#$. In that setting the inhomogeneous pseudoholomorphic map equation reduces to Floer's equation, and (\ref{eq:inhomogeneous_map}) reduces to (\ref{eq:floer_eq}).
\end{remark}

\begin{definition}
	\label{def:stratified-top-space}
	A \emph{stratified topological space} $T$ is a topological space $T$ together with a collection of closed subspaces $\partial^i T \supset \partial^{i+1}T$ for $i \geq 0$ and $\partial^0 T = T$, called the \emph{strata} of $T$. The \emph{open strata} of the space are then the subspaces $\partial^i_o T := \partial^{i}T \setminus \partial^{i+1}T$ for $i \geq 0$. The index labels the \emph{codimension} of the stratum.
\end{definition}
The Gromov-Floer bordifications
\begin{equation}
\label{eq:bordification}
\overline{\M}^F(\gamma_+, \gamma_{-}; H, J),  \overline{\M}^C(\gamma_+, \gamma_-; \Sigma, \mathfrak{B}, K, J)
\end{equation}
are stratified topological spaces which contain 
\[\M^F(\gamma_+, \gamma_{-}; H, J),\M^C(\gamma_+, \gamma_-; \Sigma, \mathfrak{B}, K, J) \]
as codimension $0$ open strata, respectively. Suppressing the Floer data and perturbation data from the notation, these satisfy
\begin{equation}
\label{eq:floer-bdry-1} \partial^1\overline{\M}^F(\gamma_-, \gamma_+) = \bigcup_{\gamma_0 \in \mathcal{C}(L_0, L_1)} \overline{\M}^F(\gamma_-,\gamma_0) \times \overline{\M}^F(\gamma_0,\gamma_+);
\end{equation}
and if $\Sigma = Z$ with boundary Lagrangians $L_0, L_1$, 
\begin{equation*}
\begin{split}
\partial^1\overline{\M}^C(\gamma_-, \gamma_{+}) = \bigcup_{\gamma_0 \in \mathcal{C}(L_0, L_1)} 
&\overline{\M}^F(\gamma_-,\gamma_0) \times \overline{\M}^C(\gamma_0,\gamma_+) \;\cup \overline{\M}^C(\gamma_-,\gamma_0) \times \overline{\M}^F(\gamma_0,\gamma_+) 
\end{split}
\end{equation*}  
and if $\Sigma = H$ with boundary Lagrangian $L_0$ then
\begin{equation*}
\partial^1\overline{\M}^C(\{\gamma_-\}) = \bigcup_{\gamma_0 \in \mathcal{C}(L_0, L_1)} 
\overline{\M}^F(\gamma_-,\gamma_0) \times \overline{\M}^C(\{\gamma_0\}).
\end{equation*}  

\subsection{Some topological preliminaries}
\label{sec:topological-preliminaries}
Let $X$ be a topological $n$-manifold. Then $X$ carries an \emph{orientation sheaf} $\mathfrak{o}_X$, which is a $\Z$-local system with stalk at $p$ given by $H_n(X, X\setminus p)$. If $X$ is closed, then there is a fundamental class $[X]$ in
\[ H_n(X, \mathfrak{o}_X^\check ) \]
which is a generator for this group. It is a convenient fact that for the right dual 
\[\mathfrak{o}^\check_X := Hom(\mathfrak{o}^\check_X, \Z)\] 
we have a canonical isomorphism $\mathfrak{o}^\check_X \simeq \mathfrak{o}_X$ of local systems, and moreover when we equip $\mathfrak{o}_X$ with the $\Z/2$-grading given by the dimension of $X$, then this is an isomorphism of $\Z/2$-graded local systems. In the rest of the paper we use $\mathfrak{o}_X$ instead of $\mathfrak{o}_X^\check$ everywhere.

Let $\bar{X}$ be a topological $n$-manifold with boundary $\partial \bar X$; namely a second countable Haussdorff topological space locally modeled on $\R_{> 0} \times \R^{n-1}$.  Let $j: X \to \bar{X}$ be the inclusion of the interior. Then there there is a natural sheaf on $\bar{X}$ given by 

\[ \mathfrak{o}_{\bar{X}} := j_*\mathfrak{o}_{X}, \text{ the orientation sheaf of }\bar{X} \]
which is also a $\Z$-local system. There is an identification 

\begin{equation}
\label{eq:boundary-map-orientation-lines}
\mathfrak{o}_{\bar{X}}|_{\partial X} \simeq \mathfrak{o}_{\partial \bar{X}}
\end{equation} 
which arises, in a local chart near $p \in \partial \bar{X}$, from the isomorphism
\renewcommand{\o}{\mathfrak{o}}
\[ H_n(\R^n, \R^n \setminus 0) = H_1(\R, \R\setminus 0) \tensor H_{n-1}(\R^{n-1}, \R^{n-1} \setminus 0) \simeq \Z \tensor H_{n-1}(\R^{n-1}, \R^{n-1} \setminus 0) \]
coming from the Kunneth formula and the trivialization of the orientation of $\R$ coming from taking a tangent vector pointed in the negative direction. The identification gives rise to the usual boundary map

\[ H_n(\bar{X}, \partial \bar X, \o_{\bar{X}}) \to H_{n-1}(\partial \bar {X}, \o_{\partial \bar X}). \]
The domain of this map is generated by the relative fundamental class $[\bar X]$, 
and map takes $[\bar X]$ to $[\partial \bar X]$. 

\begin{definition}
	\label{def:top-manifold-with-corners}
	A \emph{topological manifold with corners} of dimension $n$ is a second-countable Hausdorff stratified topological space $X$ such that each point has a neighborhood equipped with a homeomorphism to an open subset of $A_{n, k} := [0, \infty)^k \times \R^{n-k} $ for some $n, k$, where the homeomorphism takes the stratification on $X$ to the standard stratification on $A_{n, k}$. Thus $\partial_0^k X$ is a topological manifold of dimension $n-k$ if $k \leq n$ and is empty otherwise.
\end{definition}

Notice that every topological $n$- manifold with corners $\bar{X}$ is a topological $n$-manifold with boundary, and has a fundamental class in $H_n(\bar{X}, \partial \bar{X}; \o_{\bar{X}}^\check)$. 


\subsection{Regularity}
\label{sec:regularity}
Fixing $2 < p < \infty$, one considers the Banach manifold $\mathcal{B}_S$ of maps $u: S \to M$ which are locally $W^{1,p}$ and converge to some $\{y_\zeta\}$ on the strip like ends with exponential speed in the $C^1$ norm. There is a Banach vector bundle $\mathcal{E}_S$ with fibre at $u$ given by $L^p(S, \Omega^{0,1}S \tensor u^*TM)$, and the quantity appearing in the inhomogeneous pseudoholomorphic map equation, $(du-Y)^{0,1}$, can be seen as a section of this bundle; thus the zero-set of the section is set of solutions that we are interested in, since any solutions lie in the image of $C^{\infty}(Z, M)$ by elliptic regularity.  At a zero $u$ of the section, we have a Cauchy Riemann problem $D_u: (T\mathcal{B}_S)_u = W^{1,p}(S,u^*TM,\cup_i u^*TL_i) \to L^p(S, \Omega^{0,1}S \tensor u^*TM)$ which is the linearization of $(du - Y)^{0,1}$ (see e.g. \cite[Section~II.8h]{seidel}). Given a Fredholm operator $D_u$ we write 
\begin{equation}
\label{eq:determinant-line-def}
\det D_u = \left(\Lambda^{top} \coker D_u\right)^\check \tensor \Lambda^{top} \ker D_u
\end{equation}
for its \emph{determinant line}. At any zero $u$ of the section, by elliptic regularity results, the kernel and cokernel of $D_u$ are both finite dimensional and independent of $p$, and thus the determinant lines of the linearizations patch together to a canonical line bundle with a canonical associated $\Z$-local system $\mathfrak{d}_{\M^C(\gamma_+, \gamma_-; \Sigma, \mathfrak{B}, K, J )}$ over $\M^C(\gamma_+, \gamma_-; \Sigma, \mathfrak{B}, K, J )$ called the \emph{determinant local system}. 


We say that a solution $u$ to the inhomogeneous pseudoholomorphic map equation is \emph{regular} if the linearized operator $D_u$ is surjective; we say that a perturbation datum for $S$ is \emph{regular} if every solution to the inhomogeneous pseudoholomorhic map equation for that perturbation datum is regular. The following fact asserting the existence of enough regular perturbation data is standard \cite{seidel}:

\begin{proposition}
	Regular perturbation data exist for any pair of closed exact Lagrangians in an exact symplectic manifold with convex boundary. In the same setting, given a boundary-marked Riemann surface equipped with regular boundary Floer data, a regular perturbation datum for this boundary-marked Riemann surface exists. $\qed$
\end{proposition}

Moreover, 
\begin{proposition}
	\label{prop:hamiltonian-is-part-of-regular-floer-datum}
	Given a pair of closed exact Lagrangians $L_0, L_1$ in a Liouville domain $M$ and a time-dependent Hamiltonian $H \in C^\infty([0,1] \times M)$, if the image of $L_0$ under the time-1 flow of $H$ intersects $L_1$ transversely, then $H$ is the Hamiltonian part of a regular Floer datum for the Lagrangians. $\qed$
\end{proposition}

Finally, it is known that 
\begin{proposition}
	Let $(\Sigma, \mathfrak{B}, K, J)$ be a choice of boundary marked Riemann surface equipped with regular boundary Floer data and regular perturbation data. Then every connected component of $\overline{\M}^C(\gamma_+, \gamma_-; \Sigma, \mathfrak{B}, K, J )$, with the induced stratification,
	is a topological manifold with corners. 
	
	Moreover the determinant local system on the open stratum of each connected component extends to the entire connected component and is canonically isomorphic to the corresponding orientation sheaf of that component. In particular, if $(\Sigma, \mathfrak{B}, K, J) = (H, J)^\#$ (see Remark \ref{rk:floer_gives_perturbation_data}), then there is a canonical isomorphism \[\mathfrak{d}_{\widetilde\M^F(\gamma_+, \gamma_-; H, J)} \simeq \R_{\partial/\partial s} \tensor \o_{ \M^F(\gamma_+, \gamma_-; H, J)},\]
	due to the definition of $\M^F$ as the quotient of $\M^C(\Sigma, \mathfrak{B}, K, J) = \widetilde{\M}^F(H, J)$ by the translation $\R$-action (see Appendix \ref{sec:orientation-conventions} for our conventions about determinant lines of quotients). $\qed$
\end{proposition}

\subsection{Orientation lines}
\subsubsection{Conventions on orientations}
\label{sec:orientation-conventions}
Given a vector space $V$ we have the associated determinant line $\det V = \Lambda^{top} V$. 

Given an exact sequence of vector spaces
\[ 0 \to V' \to V \to V'' \to 0 \]
there is a canonical (up to a multiplicative action of the positive real numbers) isomorphism of determinant lines
\[ \det V \simeq \det V' \tensor \det V'' \]
arising from a choice of decomposition $V \simeq V' \oplus V''$ coming from a choice of section $V'' \to V$. 

This gives a canonical isomorphism of associated orientation groups
\[|\det V| \simeq |\det V'| \tensor_\Z |\det V''|.\]

There is also an isomorphism 
\[ \det V^\check \tensor \det V \to \R \]
given by evaluation, which induces an isomorphism of orientation groups
\[ |\det V^\check | \tensor |\det V| \to \Z. \]

Given a pair of vector spaces $V, V'$, we have an isomorphism 
\[ V \times V' \simeq V' \times V \]
but the associated isomorphism of determinant lines 
\[ \det V \tensor \det V' \simeq \det V' \tensor \det V \]
differs from the standard braiding in the tensor category of vector spaces by a sign $(-1)^{\dim V \dim V'}$. In other words, if we equip an orientation group $|\det V|$ with the $\Z/2= \{\pm 1\}$ grading $(-1)^{\dim V}$, then the associated isomorphim of orientation groups becomes the braiding in the tensor category of free rank $1$ abelian supergroups. An important point to keep in mind is that isomorphisms like $|\langle v \rangle^\check| \tensor |\det V| \simeq |\det V|$ or $|\R_{a_1}| \tensor |\R_{a_2}| \simeq |\R|$ do \emph{not} respect the mod-2 grading of these orientation groups, altough they are well defined isomorphisms of groups. We will thus call them \emph{grading-shifting isomorphisms} of orientation groups. They will be necessary in the definition of the Floer complexes of this paper.

\subsubsection{Some recollections on the topology of the Lagrangian Grassmannian}

Let $W$ be a symplectic vector space, and let $LGr(W)$ be the space of (unoriented) Lagrangian subspaces of $W$. We introduce in this paragraph a certain fiber bundle $\widehat{LGr(n)}$ over $LGr(W)$. There are canonical cohomology classes
\begin{equation}
\begin{gathered}
\mu \in H^1(LGr(W), \Z), \; \text{the Maslov class, and} \\
w_2 \in H^2(LGr(W), \Z/2), \; \text{the second Steifel-Whitney class},
\end{gathered} 
\end{equation}
which define isomorphisms
$\pi_1(LGr(W))\simeq\Z$ and  $\pi_2(LGr(W))\simeq \Z/2$, respectively. We write 
\[LGr(n) := LGr(\C^n).\] 

We have the fiber sequences 
\[ \Z/2 \to Pin(n) \to O(n) \to B \Z/2 \to BPin(n) \to BO(n) \]
and
\[ O(n) \to U(n) \to LGr(n) \to BO(n)\]
the fiber product 
\begin{equation}\widehat{LGr(n)} := BPin(n) \times_{BO(n)} LGr(n)
\end{equation} is a $B \Z/2$-fibration over $LGr(n)$.  The map $\widehat{LGr(n)} \to LGr(n)$ is an isomorphism on $\pi_k$ for $k \neq 2$; however, $\pi_2(\widehat{LGr(n)})= 0$.  Thus, the based loop space of $\widehat{LGr(n)}$ has $\Z$ connected components indexed by the Maslov class, and each connected component is simply connected. 

By construction, picking a Pin structure on $L \in LGr(n)$ is equivalent to choosing a lift of $L$ to $\widehat{LGr(n)}$. 

\subsubsection{Definitions of orientation lines}
\label{sec:def-orientation-lines}
Let $(\widehat{\Lambda_1}, \widehat{\Lambda_2})$ be a pair of elements of $\widehat{LGr(W)}$ so that the corresponding pair of Lagrangian subspaces $(\Lambda_1, \Lambda_2)$ in $LGr(W)$ intersect transversely. The construction in the previous section shows that every 
component of the space $\P_{\widehat{\Lambda_1}, \widehat{\Lambda_2}}$ of paths from $\widehat{\Lambda_1}$ to $\widehat{\Lambda_2}$,  is simply connected. For any element $\tilde{\gamma} \in \P_{\Lambda_1, \Lambda_2}$ with image $\gamma$ in the space of paths 
in the Lagrangian Grassmannian, we can construct a Cauchy-Riemann problem as follows: let $S = H$, $E = H \times W$, $\psi$ be a nondecreasing smooth function $\R \to [0,1]$ with $\psi(s) = 0$ 
for $s << 0$, $\psi(s) = 1$ for $s >> 1$; $\Lambda_s = \gamma(\psi(s))$, the connection is the trivial connection; the strip-like end is just the unique conformal parameterization of the complement to a large semicircle in the upper half plane, the limiting bundle 
$E$ is also trivial with $\phi$ the identity and the remaining limiting data are the only ones possible for this setup. The resulting Fredholm operators $D$ have a locally constant index $ind(D) \in \Z$
and give rise to a determinant line $\det D \to \P_{\widehat{\Lambda_1}, \widehat{\Lambda_2}}$; the indices bijectively label the connected components of $\P_{\widehat{\Lambda_1}, \widehat{\Lambda_2}}$ by integers, and the determinant lines are trivial on each component
because each component is simply connected. Let $\P^n_{\widehat{\Lambda_1}, \widehat{\Lambda_2}}$ be the component of $\P_{\widehat{\Lambda_1}, \widehat{\Lambda_2}}$ corresponding to the index $n$ Fredholm operators. Let $\det^n(\widehat{\Lambda_1}, \widehat{\Lambda_2})$ denote the (trivializable) determinant line bundle $\det D|_{\P^n_{\widehat{\Lambda_1},\widehat{\Lambda_2}}}$, and let $o^n(\widehat{\Lambda_1}, \widehat{\Lambda_2})$ denote the free abelian group corresponding to the $\Z/2$-torsor of orientations of $\det^n(\widehat{\Lambda_1}, \widehat{\Lambda_2})$. 

There is a slight generalization of this construction which must be used to describe the Floer complex. Namely, suppose we are given a \emph{limiting datum}: a Hermitian vector bundle $E' \to [0,1]$ with symplectic form $\omega_{E'}$ and complex structure $J_{E'}$, a as well as a symplectic connection $\Grad_{E'}$ and a pair of of Lagrangian subspaces $\Lambda_{E'}^0 \subset E'|_0,  \Lambda_{E'}^1 \subset E'|_1$,  such that the $\Grad_{E'}$ parallel transport of $\Lambda_{E'}^0$ to $E'|_1$ is transverse to $\Lambda_{E'}|_1$. Moreover, assume that this limiting datum is \emph{equipped with a Pin structure}, namely, assume that $\Lambda_{E'}^k$ have given Pin structures for $k = 0,1$, corresponding to lifts $\widehat{\Lambda_{E'}^k}$ of $\Lambda_{E'}^k$ to $\widehat{LGr}(E'|_k)$.  Then we can define an analog of the orientation line by trivializing $E'$ using $\Grad_{E'}$, viewing this as a new limiting datum on a trivial bundle $W \times [0,1] \to [0,1]$ with trivial connection (for definiteness, choose $W = E'|_1$); extending this trivial bundle to $H$ with the same strip-like end and trivial connection, and choosing boundary conditions by choosing a path between the images of $\widehat{\Lambda_{E'}^k}$ in $W$ under this trivialization, all as in the previous paragraph.  This procedure is discussed in more detail in \cite[Section~II.11.l]{seidel}. The space of limiting data retracts onto the space of pairs of Lagrangian subspaces $(\widehat{\Lambda_{E'}^0}, \widehat{\Lambda_{E'}^1})$ equipped with Pin structures, and the determinant line of the operator $D$ defined by a limiting datum equipped with a Pin structure is canonically isomorphic to $o^{ind D}(\widehat{\Lambda_{E'}^0}, \widehat{\Lambda_{E'}^1})$ for the corresponding pair $(\widehat{\Lambda_{E'}^0},\widehat{\Lambda_{E'}^1})$.

A Hamiltonian chord $y$ for a Floer datum gives a limiting datum equipped with a Pin structure; let the corresponding pair of Lagrangian subspace equipped with Pin structures be $(\widehat{\Lambda^0_y}, \widehat{\Lambda^1_y})$, and call
\begin{equation}
o^n(y) := o^n(\widehat{\Lambda_1}, \widehat{\Lambda_2})
\end{equation}
the \emph{orientation line of $y$ of index $n$}.

\subsubsection{Shift line}
\label{sec:shift-line}
Finally, there is the \emph{shift line}, which is the determinant line of a certain family of index $1$ Cauchy-Riemann problems on $Z$. Given a limiting datum, define a Hermitian bundle over $Z$ by the pullback via projection to the second coordinate of the bundle $E' \to [0,1]$ given by the limiting datum, and equip it with the pullback connection. Choose Lagrangian boundary conditions by choosing the constant Lagrangian boundary condition on $\R \times \{1\}$ (so it is just $L_{E'}^1$, and choose the projection to the Lagrangian grassmannian of any loop from $\widehat{\Lambda_{E'}^0}$ to itself, composed with $\psi$, as the boundary condition for $\R \times \{0\}$, such that the associated Fredholm operator has index $1$. The determinant lines of these operators naturally form bundle over a connected component of $\Omega \widehat{LGr}(E'|_0)$ based at $\widehat{\Lambda_{E'}^0}$; this space is simply connected and so the line bundle is trivializable. The free abelian group corresponding to the torsor of orientations of this line bundle is the \emph{shift line} associated to the limiting datum. If the limiting datum arises from a Hamiltonian chord $y$ of a Floer datum equipped with Pin structures as in (\ref{eq:pin-structures-choice-floer}), then we will denote this line  by $o^S(y)$. 

\subsubsection{The essential trick}
\label{sec:gluing_and_orientations}
Suppose that we are goven a Cauchy-Riemann problem arising as the linearization of the pseudoholomorphic map equation at a map $u$, with the target a manifold of dimension $2n$, and the domain a Riemann surface equipped with boundary Floer data and perturbation data, such that $u$ is asymptoting to Hamiltonian chords $\{y_\zeta\}$ at each of the ends $\zeta$ of the Riemann surface, and the chords are equipped with Pin structures as in (\ref{eq:pin-structures-choice-floer}). Then it makes sense to apply the gluing theorem to $D_u$ and $\tensor_{\zeta \in \Sigma^+}$ $o^{n_\zeta}(y_\zeta)$, where the $n_\zeta$ are arbitrary integers. If the Lagrangian boundary conditions of the pseudoholomorphic map $u$ are equipped with Pin structures relative to the given Pin structures on $\{y_\zeta\}$, then the standard argument giving coherent orientations in Lagrangian Floer theory shows, if there was exactly one incoming point with asymptotic Hamiltonian chord $y_-$, then  
\begin{equation}
\label{eq:gluing-1}
o^{f_1}(y_-) \simeq  |\det D_u| \bigotimes \tensor_{\zeta \in \Sigma^+} o^{n_\zeta}(y_\zeta)
\end{equation} 
where $f_1 = \text{ind} D_u + \sum_{\zeta \in \Sigma^+} n_\zeta$. The argument for this described in \cite[Section~II.11.13]{seidel}. If there are no incoming points then
\[ \lambda^{top}(\Lambda_0) \simeq  |\det D_u| \bigotimes \tensor_{\zeta \in \Sigma^+} o^{n_\zeta}(y_\zeta). \]
if we choose $n_\zeta$ such that $\sum_\zeta n_\zeta = n - \text{ind} D_u$; because then the glued Cauchy Riemann problem on the disk has maslov index $n$ and a Pin structure along its boundary, and the space of such Cauchy-Riemann problems is simply connected and contains the one with the constant boundary condition at $u(0)$, for which a standard computation (e.g. \cite[Lemma II.11.7]{seidel}) shows that this Cauchy-Riemann problem gives rise to a surjective Fredholm operator with kernel isomorphic to the vector space $\Lambda_0$ via the evaluation map. The Pin structure along the boundary of $u$ usually comes from choices of global Pin structures on various Lagrangians. The essential idea of this paper is that one can do away with global Pin structures, and instead build all possible local choices of Pin structures into the algebra of Floer theory.

Now, in the above setting, the gluing theorem also gives canonical isomorphisms 
\[ o^{n_\zeta + 1}(y_\zeta) \simeq o^S(y_\zeta) \tensor o^{n_\zeta}(y_\zeta) \]
and thus canonical isomorphisms 
\begin{equation}
\label{eq:tensor-square}
o^{n_\zeta + 2}(y_\zeta) \simeq o^{n_\zeta}(y_\zeta)
\end{equation}
since the tensor square of a line is canonically trivial. The proof of the isomorphism in (\ref{eq:gluing-1}) also shows when the underlying Riemann surface is $Z$, we have that
\begin{equation}
\label{eq:equivariance-of-gluing}
\begin{minipage}{0.9 \textwidth}
the isomorphism (\ref{eq:gluing-1}) is equivariant under the canonical isomorphisms \[o^{f_1}(y_-) \simeq o^{f_1 + 2}(y_-), \; o^{n_+}(y_+) \simeq o^{n_++2}(y_+)\]
	\end{minipage}
\end{equation}

However, the isomorphims of the previous discussion are proven \emph{using a choice of Pin structure relative to the ends on the boundary conditions} of the pseudoholomorphic map; the isomorphisms are \emph{changed by a sign by changing the Pin structures on any boundary component}. Thus, suppose the Riemann surface under consideration is $Z$. Write $\Pi = |\Pi(TL|_{\R \times \{0\}})| \tensor |\Pi(TL|_{\R \times \{1\}})|$ for the tensor product of lines associated to choices of Pin structure along the Lagrangian bundles on each boundary component of the Cauchy-Riemann problem. If the Hamiltonian chords at the input and output are $y_-$ and $y_+$, one has an invariant isomorphism
\begin{equation}
\label{eq:floer-moduli-orientation-invariant}
o^{f_1}(y_-)\simeq |\det D_u| \tensor o^{f_1 - \ind D_u}(y_+) \tensor \p 
\end{equation} 
which does \emph{not} depend on a choice of Pin structure along the boundary of the pseudoholomorphic curve.
Similarly, if the Riemann surface is $H$ with the end as an input, then the isomorphism is 
\begin{equation}
\label{eq:floer-moduli-orientation-invariant-no-inputs}\lambda^{top}(u(0)) \simeq |\det D_u| \tensor o^{n - \ind D_u}(y_+) \tensor |\Pi(TL|_{\partial H})|.
\end{equation}

\subsection{Technical setup for Morse theory}
\label{sec:morse-theory-technical}
Let $f$ be a Morse function on a manifold $L$. Let $|p|$ denote the index of a critical point $p$ of $f$ and let $Crit(f)$ denote the set of critical points of $f$.  Then the following  facts are standard:

\begin{proposition}
	\label{prop:morse-theory-technical1}
	There exists a nonempty Banach manifold of smooth metrics on $L$ and a comeagre set of metrics in this manifold such that for any such metric $g$, $f$ is Morse-Smale, and therefore, for any pair $p$ and $q$ of critical points of $f$, the set
	\[\widetilde\M^M(p, q; f, g) = \{ \tilde\gamma \in C^\infty(\R, M); \frac{d\tilde\gamma}{dt} = - \Grad f; \lim_{s \to -\infty} \tilde\gamma(s) = p; \lim_{s \to + \infty} \tilde\gamma(s) = q\}\]
	topologized with the subspace topology under the inclusion into Moore paths on $L$ given by sending a Morse trajectory to the corresponding path parametrized by length, is a topological manifold of dimension $|p| - |q|$, and the quotient by the (continuous, fixed point free) $\R$ action on this set $\M^M(p, q; f,g) = \widetilde \M^M(p, q; f, g)/\R$ is the open stratum of a topological manifold with corners
	\[ \overline{\M}^M(p, q; f, g) = \bigcup_{k \geq 1} \M^M(q_0, q_1; f,g) \times \ldots \times \M^M(q_{i}, q_{i+1}; f, g) \times \ldots \times \M^M(q_{k-1}, q_k; f, g). \]
\end{proposition}

\begin{proposition}
	\label{prop:morse-theory-technical2}
	Choose a smooth function $\phi: \R \to [0,1]$ such that $\phi$ is nondecreasing and has $\phi(s) = 0$ for $s <<0$ and $\phi(s) = 1$ for $s >>0$. 
	Simiarly, given $f_0, f_1 \in C^\infty(L)$ Morse-Smale with respect to a metric $g$, there exists a nonempty Banach manifold of functions $f_t \in C^\infty([0,1], C^\infty(L))$ agreeing with $f_0$ and $f_1$ at the ends of the interval, together with a comeagre subset of this manifold, such that for any $f_t$ in this subset,
	\begin{equation}
	\label{eq:perturbed_morse_equation}\M^{MC}(p, q; f_t, g) = \{ \tilde\gamma \in C^\infty(\R, M); \frac{d\tilde\gamma}{dt} = - \Grad f_{\phi(t)}; \lim_{s \to -\infty} \tilde\gamma(s) = p; \lim_{s \to + \infty} \tilde\gamma(s) = q\} 
	\end{equation}
	topologized with the subspace topology under its inclusion into Moore paths on $L$ given by associating to a gradient trajectory the corresponding path parametrized by length, is a topological manifold of dimension $|p|-|q|$ for all $p \in Crit(f_0)$, $q \in Crit(f_1)$, and its compactification is a topological manifold with corners with codimension $1$ strata of the form 
	\[ \overline\M^M(p, p'; f_0, g) \times \overline\M^{MC}(p', q; f_0, g), \overline\M^{MC}(p, q'; f_t, g)\times \overline\M^M(p, q',q; f_1, g).\]
\end{proposition}
\begin{proposition}
	\label{prop:morse-theory-technical3}
	Finally, given any pair of $f_t, f'_t$ as in Proposition \ref{prop:morse-theory-technical2} above, with asymptotics $(f_t)_{t = \eta} = (f'_t)_{t = \eta} = f_\eta$ for $\eta = 0, 1$, there exists a function $f_{s,t} \in C^\infty([0,1]^2_{s,t}, \C^\infty(M))$ agreeing with $f_t$ for $s=0$ and with $f'_t$ for $s = 1$, and such that $f_{s, 0} = f_0 = f'_0, f_{s, 1} = f_1 = f'_1$, where moreover, the set
	\[ \overline{\M}^H(p, q; f_{s,t}, g) := \cup_{s \in [0,1]}\overline\M^{MC}(p, q; f_{s,t}, g) \]
	topologized with the weakest topology making the projection to $s$ continuous and making all the subsets $\overline\M^{MC}(p, q; f_{s,t}, g)$ into subspaces, is a topological manifold with corners, with codimension $1$ strata given by 
	\begin{gather*}
	\overline\M^{MC}(p, q; f_t, g), \overline\M^{MC}(p, q; f'_t, g), \\\overline\M^M(p, p'; f_0, g) \times \overline{\M}^H(p', q; f_{s,t}, g),\\\overline{\M}^H(p, q'; f_{s,t}, g)\times \overline\M^M(q', q; f_1, g). 
	\end{gather*}
\end{proposition}

\subsection{Orienting Morse-theoretic moduli spaces}
\label{sec:morse-theory-orientations}
Unlike in Floer homology, orienting Morse-theoretic moduli spaces is straightforward. Indeed, the moduli space $\widetilde\M^M(p, q; f, g)$ recalled in Prop. \ref{prop:morse-theory-technical1} is just the fiber product of the interior of the \emph{unstable manifold} $W^u(p)$ of $p$, and the \emph{stable manifold} $W^s(q)$ of $q$, along their respective inclusion maps to $L$. The map $\widetilde\M^M(p, q; f, g) \to L$ coming from this description as a fiber product can be taken to be the map $\tilde{\gamma} \mapsto \tilde{\gamma}(0)$. The stable and unstable manifolds have parallelizable tangent bundles: parallel transport with respect to $g$ along gradient trajectories gives canonical isomorphisms
\[ TW^u(p) \simeq T^-_p \times W^u(p),\]
\[ TW^s(q) \simeq T^+_q \times W^s(q),\]
where $T^\pm_p$ is the eigenspace of the Hessian of $f$ at $p$ of the corresponding sign (see Figure \ref{fig:morse-notation}). We write $o(T^\pm_p)$ for the orientation lines of these eigenspaces; trivializations of these lines biject with orientations of the corresponding (un)stable manifolds. By the characterization of $\widetilde\M^M(p, q; f, g)$ as a fiber product, the orientation line if a point $\tilde{\gamma} \in \widetilde\M^M(p, q; f, g)$ is canonically isomorphic to $ o(T^-_p) \tensor o(T^+_q) \tensor \det TL_{\tilde{\gamma}(0)}^\check$.  Now there is also a canonical isomorphism $o(T^-_q)^\check \simeq o(T^+_q) \tensor \det TL_q^\check$. Along each fiber of $\widetilde\M^M(p, q; f, g)$ over 
$\overline\M^M(p, q; f, g)$, the bundle $\det L_{\tilde{\gamma}(0)}$ is canonically trivialized by parallel transport. Thus, the orientation local system of  $\overline\M^M(p, q; f_t, g_t)$ canonically isomorphic to the trivial local system with fiber $o(T^-_p) \tensor o(T^-_q)^\check$,  and the orientation local system of $\M^M(p, q; f, g)$ is canonically isomorphic to the trivial local system with fiber
\[ \Z_{\partial/\partial t}^\check \tensor o(T^-_p)\tensor o(T^-_q)^\check.  \]

Similarly, there is a canonical trivialization  
\begin{equation}
\mathfrak{o}_{\overline{\M}^{MC}(p, q; f_t, g)} \simeq o(T^-_p)^{f_0} \tensor (o(T^-_p)^{f_0})^\check,
\end{equation}
 where $o(T^-_r)^{f_i}$ is the determinant of the negative eigenspace of the Hessian of $f_i$ at $r \in Crit(f_i)$. Finally, the standard description of orientation lines of parametrized moduli spaces means that fixing a choice of orientation of the parameter space $[0,1]_s$ gives an isomorphism between the orientation line of the tangent space to $u \in \overline{\M}^H(p, q; f_{s,t}, g)$ (with notation $f_{s,t}$ as in Prop. \ref{prop:morse-theory-technical3}) with the orientation line of the linearization of the perturbed Morse equation (\ref{eq:perturbed_morse_equation}) satisfied by $u$, which is in turn canonically isomorphic to $\det D_p^{f_0} \tensor (\det D_q^{f_1})^\check$. (See also, for example, the discussion after Prop. \ref{prop:homotopy-moduli-space-compactification}.) We orient the parameter space $[0,1]_s$ by requiring that $\partial/\partial s$ is positively oriented.

\subsection{A very brief review of gradings}
\label{sec:gradings-review}
Usually, the Floer complex $CF(L_0, L_1)$ is \emph{not graded}. We briefly recall the theory of gradings on symplectic manifolds \cite{seidel-graded-lagrangians}. We use this in Section \ref{sec:natural_home} to grade Hamiltonian chords, and later in Section \ref{sec:floer_complex} to grade the Floer complex.

A symplectic manifold $M$ has a bundle $LGr(M)$ of Lagrangian Grassmannians. A $\Z/n\Z$ grading (for $n=2,3\ldots, \infty$ where for $n= \infty$ we set $\Z/\infty \Z := \Z$) on $M$ is a fiber bundle $LGr'(M) \to M$ with connected fibers, equipped with a bundle map $LGr'(M) \to LGr(M)$ which is an $n$-fold covering space on each fiber. Thus, for example, every symplectic manifold has a $\Z/2\Z$ grading given by the bundle of oriented Lagrangian suspaces of the tangent bundle. Given a choice of grading $LGr'(M)$ on $M$, a grading on the Lagrangian $L \subset M$ with respect to the chosen grading on $M$ is a lift of the classifying map of the Lagrangian $L \to LGr(M)$ to $LGr'(M)$. Given two Lagrangians that are graded with respect to a $\Z/n\Z$-grading on a symplectic manifold $M$, one can define a $\Z/n\Z$ grading on the Hamiltonian chords between the two Lagranians, and the indices of the pseudoholomorphic curves between these two chords will respect the gradings, in the sense that the index of any such curve modulo $n$ will be equal to the difference of the gradings of the chords at its endpoints.

\begin{figure}
	\centering
	\resizebox{\textwidth}{!}{
		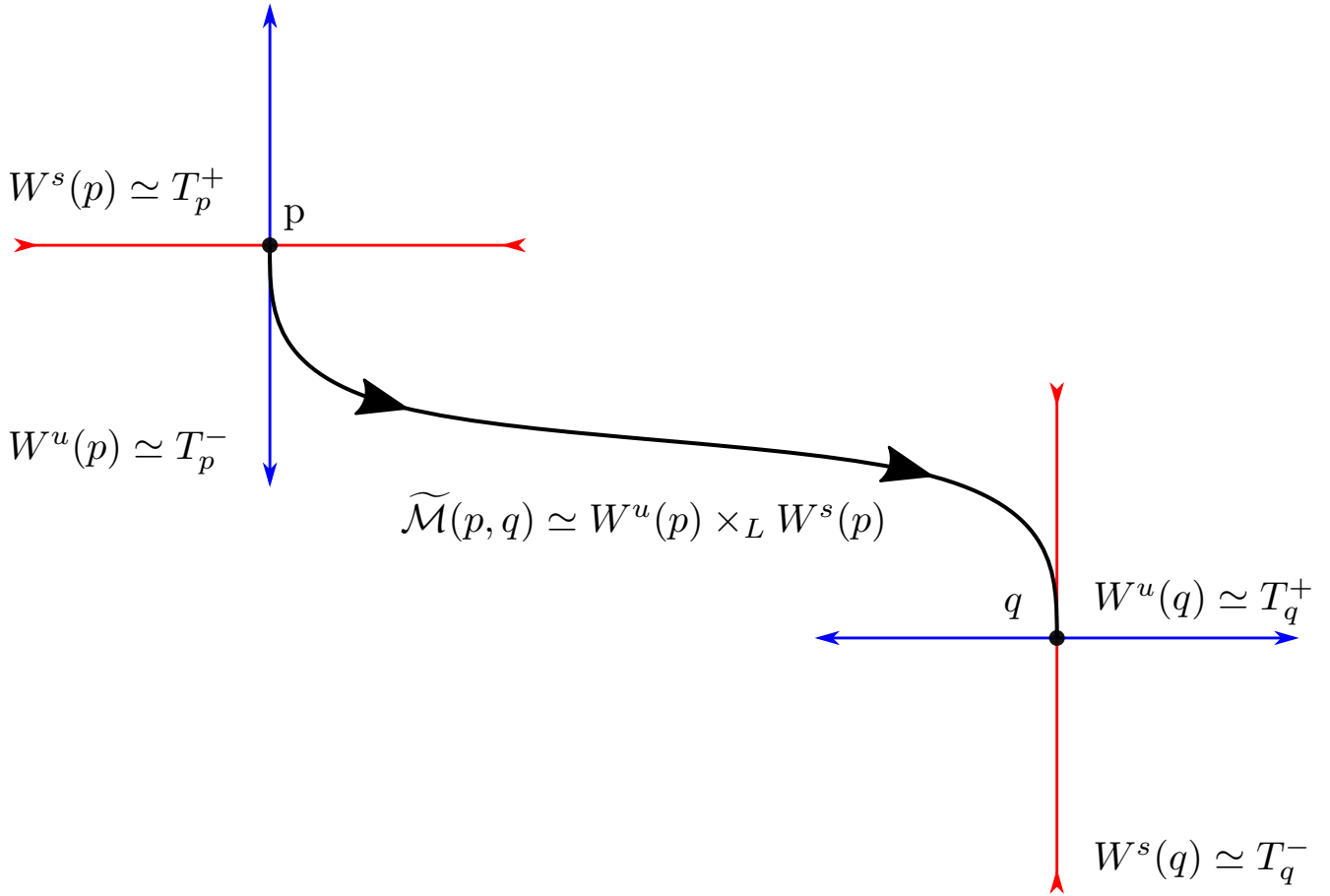}
	\caption{An illustration of how Morse-theoretic moduli spaces are oriented.}
		\label{fig:morse-notation}
\end{figure}

\bibliographystyle{amsplain}
\bibliography{signs_final}

\end{document}